\documentclass[a4paper,10pt]{amsart}
\usepackage[english]{babel}
\usepackage[leqno]{amsmath}
\usepackage{amssymb,amsthm}
\usepackage{amscd}
\usepackage{enumerate}
\usepackage{url}
\usepackage{stmaryrd}

\usepackage{layout}
\usepackage{fullpage}

\input{xy}
\xyoption{all}

\newtheorem{thm}{Theorem}[section]
\newtheorem{theoremalpha}{Theorem}

\newtheorem{lemma}[thm]{Lemma}

\newtheorem{prop}[thm]{Proposition}

\newtheorem{cor}[thm]{Corollary}

\newtheorem{fact}[thm]{Fact}
\newtheorem{conj}[thm]{Conjecture}

\theoremstyle{remark}
\newtheorem{remark}[thm]{Remark}
\theoremstyle{definition}
\newtheorem{defi}[thm]{Definition}
\newtheorem{nota}[thm]{}
\newtheorem{notation}[thm]{}
\newtheorem{ass}[thm]{Assumption}

\numberwithin{equation}{section}

\newenvironment{sis}{\left\{\begin{aligned}}{\end{aligned}\right.}

\newcommand{\w}{\widetilde}
\newcommand{\ov}{\overline}
\newcommand{\h}{\widehat}
\newcommand{\un}{\underline}

\newcommand{\Aut}{\operatorname{Aut}}

\newcommand{\Sym}{\operatorname{Sym}}
\newcommand{\Hom}{\operatorname{Hom}}

\renewcommand{\Im}{\operatorname{Im}}

\newcommand{\coker}{\operatorname{coker}}

\newcommand{\Pic}{\operatorname{Pic}}
\newcommand{\Br}{\operatorname{Br}}
\newcommand{\Cl}{\operatorname{Cl}}

\newcommand{\id}{\operatorname{id}}

\newcommand{\obs}{\operatorname{obs}}
\newcommand{\res}{\operatorname{res}}
\newcommand{\Spf}{\operatorname{Spf}}
\newcommand{\Def}{\operatorname{Def}}

\newcommand{\ch}{\operatorname{ch}}
\newcommand{\Td}{\operatorname{Td}}

\newcommand{\CC}{\mathbb{C}}

\def\L{\mathcal L}
\def \M{\mathcal M}
\def\N{\mathbb N}
\def\O{\mathcal O}
\def \I {\mathcal I}
\def \D{\mathcal D}
\def \C{\mathcal C}

\newcommand{\Q}{\mathbb{Q}}

\def\X{\mathcal X}
\def\Y{\mathcal Y}
\def \F{\mathcal F}
\newcommand{\Z}{\mathbb{Z}}
\renewcommand{\P}{\mathbb P}

\def\NN{\mathcal N}

\newcommand{\mcd}{(d-g+1,2g-2)}
\newcommand{\kdg}{k_{d,g}}
\newcommand{\edg}{e_{d,g}}

\def\pdb{\overline{\mathcal{J}}_{d,g}}
\def\pd{{\mathcal{J}_{d,g}}}
\def \pdo{{\mathcal J}^0_{d,g}}

\def\pmid{{\mathcal J}_{g-1,g}}
\def\pmidt{{\mathcal Jac}_{g-1,g}}

\def\pmidtb{\ov{{\mathcal Jac}}_{g-1,g}}
\def\unimid{{\mathcal Jac}_{g-1,g,1}}

\def\pdt{{\mathcal Jac}_{d,g}}
\def\pdtb{{\overline{\mathcal Jac}}_{d,g}}

\def\mg{\mathcal{M}_g}

\def\mgb{\overline{\mathcal{M}}_g}

\def \mguniv{\ov{M}_{g,1}}

\def \skd{\mathcal{S}_g^{1/\kdg}}

\def\univb{\overline{\mathcal Jac}_{d,g,1}}
\def \univ{{\mathcal Jac}_{d,g,1}}

\def \u{{\mathcal J}_{d, g, 1}}

\def \RPic{\mathcal{R}{\rm Pic}}

\def \Pict{\Pic^{\rm taut}}

\def \Pics{{\mathcal Jac}}

\newcommand{\Gm}{\mathbb{G}_m}

\newcommand{\Hol}{\rm Hol}

\begin{document}

\title[Picard group of the compactified universal Jacobian]
% over {$\mgb$}]
{The Picard group of the compactified universal Jacobian}

\author{Margarida Melo and Filippo Viviani}
\address{Departamento de Matem\'atica da Universidade de Coimbra,
Largo D. Dinis, Apartado 3008, 3001 Coimbra (Portugal) \text{ and } Dipartimento di Matematica,
Universit\`a Roma Tre,
Largo S. Leonardo Murialdo 1,
00146 Roma (Italy).}
\email{mmelo@mat.uc.pt, filippo.viviani@gmail.com}

%\author{Filippo Viviani}
%\address{
%Dipartimento di Matematica,
%Universit\`a Roma Tre,
%Largo S. Leonardo Murialdo 1,
%00146 Roma (Italy)}
%\email{viviani@mat.uniroma3.it}

\thanks{The first author was supported by the FCT project \textit{Espa\c cos de Moduli em Geometria Alg\'ebrica} (PTDC/MAT/111332/2009), by the FCT project \textit{Geometria Alg\'ebrica em Portugal} (PTDC/MAT/099275/2008) and by the Funda\c c\~ao Calouste Gulbenkian program ``Est\'imulo \`a investiga\c c\~ao 2010''. The second author is a member of the research center CMUC (University of Coimbra) and he was supported by  the FCT project \textit{Espa\c cos de Moduli em Geometria Alg\'ebrica} (PTDC/ MAT/111332/2009) and by MIUR--FIRB project \textit{Spazi di moduli e applicazioni}. }



\maketitle

\begin{center}
\emph{Dedicated to the memory of Torsten Ekedahl, with great admiration.}
\end{center}

\begin{abstract}
We determine explicitly the Picard groups of the universal Jacobian stack
%over the stack of smooth curves
and of its compactification over the stack of stable curves.
Along the way, we prove some results concerning the gerbe structure of the universal Jacobian stack over its rigidification by the natural action of the multiplicative group
%by scalar multiplication on the line bundles
and relate this with the existence of generalized Poincar\'e line bundles.
%on the universal Jacobian stack.
We also compare our results with Kouvidakis-Fontanari computations of the divisor class group of the universal (compactified) Jacobian scheme.
\end{abstract}

\section{introduction}

%\begin{nota}

The Picard group of a given moduli stack
%is one of the first invariants one tries to compute, since it gives
carries important informations on the geometry of the moduli problem one is dealing with.
Since Mumford's pioneer work in \cite{Mum},
the subject has been widely developed and nowadays the literature on the computation of the Picard group of moduli stacks is quite vast.
Remarkable examples are the Picard groups of the
%There is a vast literature
%on the computation of the Picard group of moduli stacks, e.g.
moduli stacks of curves possibly with level structures
(see e.g. \cite{arbcorn}, \cite{Cor1}, \cite{kou3}, \cite{Jar}, \cite{Mor},
\cite{Cor2}, \cite{GV},  \cite{Put}) and of the moduli stacks of principal bundles over curves
(see  e.g. \cite{DN}, \cite{kouvidakis}, \cite{kou2}, \cite{BL}, \cite{KN}, \cite{LS}, \cite{BLS}, \cite{Sor}, \cite{Fal}, \cite{BK},  \cite{BH}).
% and the references therein.

\vspace{0,2cm}

The aim of this paper is to compute and give explicit generators for the Picard group  of the degree-$d$ universal Jacobian stack $\pdt$ over the moduli stack $\mg$ of smooth curves of
genus $g$ and of its  compactification
$\pdtb$ over the moduli stack $\mgb$ of stable curves of genus $g$, constructed by Caporaso in \cite{cap} and \cite{capneron} and later
generalized by the first author in \cite{melo}. Moreover, we will compare our results with the computation of the divisor class group of the Caporaso's universal compactified Jacobian
scheme $\ov{J}_{d,g}$, carried out by Fontanari in \cite{fontanari} (based upon the work of Kouvidakis in \cite{kouvidakis}). The motivation for this work comes from
the wish of understanding the (log)canonical model of $\ov{J}_{d,g}$ and its relation to the different modular compactifications of the universal Jacobian. The Kodaira dimension and the
Iitaka fibration of $\ov{J}_{d,g}$ were computed by Farkas-Verra in \cite{FV} for $d=g$, by Bini, Fontanari and the second author in \cite{BFV} when $\ov{J}_{d,g}$ has finite quotient
singularities (which occurs exactly when $d+g-1$ and $2g-2$ are coprime) and  by Casalaina-Martin, Kass and the second author in \cite{CMKVb} in the general case.
An alternative compactification $\ov{J}_{d,g}^{\rm ps}$ of the universal Jacobian over Schubert's moduli space $\ov{M}_g^{\rm ps}$ of pseudo-stable curves
was recently found by G. Bini, F. Felici and the  two authors  in \cite{BMV} (see also \cite{BMV2}). We expect that $\ov{J}_{d,g}^{\rm ps}$ is the first step towards the construction of the
canonical model of $\ov{J}_{d,g}$, analogously to the fact that $\ov{M}_g^{\rm ps}$ is the first step towards the construction of the canonical model of $\ov{M}_g$ (see \cite{HH}).
Clearly, in order to verify this, one needs an explicit description of the
(rational) Picard group of $\ov{J}_{d,g}$, which naturally embeds into the (rational) Picard group  of the stack $\pdtb$.

\vspace{0,2cm}

Before describing our results, we need to  briefly recall the definitions of the stacks $\pdt$ and $\pdtb$, referring to Section 2 for more details.
%$\pdt$ is the stack whose fiber over a scheme $S$ consists of families of smooth curves $\C\to S$ over $S$ endowed
%with  a line bundle $\L$ over $\C$ of relative degree $d$ over $S$.
% and whose morphisms are the natural isomorphisms.
The degree-$d$ universal Jacobian stack $\pdt$ is the (Artin) stack whose fiber over a scheme $S$ consists of families of smooth curves $\C\to S$ over $S$ endowed with  a line bundle $\L$ over $\C$ of relative degree $d$ over $S$.
The stack $\pdt$ is contained as a dense open substack in the degree-$d$ compactified Jacobian stack $\pdtb$,  whose fiber over a scheme $S$
%was described by Melo on \cite{melo} based on previous work by Caporaso on \cite{cap} and \cite{capneron}.
consists of families of {\em quasistable} curves $\X\to S$ endowed with a {\em properly balanced} line bundle over $\X$ of relative degree $d$ over $S$ (see \ref{desc-stacks} for
the definitions).
%The stack $\pdtb$ contains the dense open substack $\pdt$,
The stack $\pdtb$ is smooth and irreducible of dimension $4g-4$, and it is endowed with a (forgetful) universally closed surjective morphism $\w{\Phi}_d$ to the stack $\mgb$ of stable curves.

The stack $\pdtb$ is  naturally endowed with the structure of a $\Gm$-stack, since
the group $\Gm$ naturally injects into the automorphism group of every object $(\C\to S, \L)\in \pdtb(S)$ as multiplication
by  scalars on $\L$. Therefore $\pdtb$ becomes a $\Gm$-gerbe over the $\Gm$-rigidification $\pdb:=\pdtb\fatslash \Gm$.
We call $\nu_d: \pdtb\to \pdb$ the rigidification map.
Analogously, $\pdt$ is a $\Gm$-gerbe over its rigidification $\pd:=\pdt\fatslash \Gm$ which is an open dense substack of $\pdb$.
%With a slight abuse of notation, we still denote by $\nu_d: \pdt\to \pd$ the rigidification map.
The stack $\pdb$ is smooth and irreducible of dimension $4g-3$, and the morphism $\w{\Phi}_d:\pdtb\to \mgb$ factors through  $\Phi_d: \pdb \to \mgb$, which is again a
universally closed surjective morphism.

Caporaso's compactification $\ov{J}_{d,g}$ of the universal Jacobian variety $J_{d,g}$ over the moduli scheme $\ov M_g$ of stable curves (see \cite{cap}) is an adequate moduli space for $\pdtb$ and for $\pdb$ (in the sense of \cite{alper2}) and even a good
moduli space (in the sense of \cite{alper}) if our base field $k$ has characteristic zero. We will call it simply the moduli space
for $\pdtb$ and for $\pdb$\footnote{In the literature, the universal (resp. universal compactified) Jacobian stack is often called the universal (resp. universal compactified) Picard stack and it is denoted by  ${\mathcal Pic}_{d,g}$ (resp. $\ov{\mathcal Pic}_{d,g}$), see e.g. \cite{capneron}, \cite{melo}, \cite{BFV}. Similarly the universal (resp. universal compactified) Jacobian scheme is often called
the universal (resp. universal compactified) Picard scheme and it is denoted by $P_{d,g}$ (resp. $\ov{P}_{d,g}$),
see e.g. \cite{cap}. Following \cite{CMKV} and \cite{BMV}, we prefer here to use the word universal (resp. universal compactified) Jacobian stack/scheme and consequently the symbols $\pdt$, $\pdtb$, $J_{d,g}$ and $\ov{J}_{d,g}$
for two reasons:
(i) the word Jacobian stack/scheme is used only for curves while the word Picard stack/scheme is used also for varieties of higher dimensions and therefore it is more ambiguous; (ii) the expression ``the Picard group of the Picard stack/scheme" seems a bit cacophonic.}.

% that we are going to use in the sequel of the paper.  In Section 2, we briefly review the definition of the stack
%$\pdb$ and of its good moduli space $\ov{J}_{d,g}$. We also recall the presentation of $\pdb$ as a quotient stack
%and the construction of $\ov{J}_{d,g}$ as a GIT-quotient.
% and it is indeed a coarse moduli space if and only if $(d-g+1,2g-2)=1$
%(see \ref{des-quot-stacks} for details) .

%\end{nota}

\vspace{0,2cm}

The main result of this paper is a description of the Picard groups of the stacks $\pdt$ and $\pd$ and of their compactifications $\pdtb$ and $\pdb$.
Since $\pdt\subset \pdtb$ and $\pd\subset \pdb$ are open
inclusions of smooth stacks, the natural restriction morphisms $\Pic(\pdtb)\to \Pic(\pdt)$ and $\Pic(\pdb)\to \Pic(\pd)$
are surjective. Moreover, since $\nu_d$ is a $\Gm$-gerbe, the pull-back morphisms $\nu_d^*:\Pic(\pdb)\to \Pic(\pdtb)$
and $\nu_d^*:\Pic(\pd)\to \Pic(\pdt)$ are injective. Therefore, the above Picard groups are related by the following
commutative diagram
\begin{equation}\label{E:4-Picard}
\xymatrix{
\Pic(\pdtb) \ar@{->>}[r] & \Pic(\pdt) \\
\Pic(\pdb) \ar@{^{(}->}[u]^{\nu_d^*} \ar@{->>}[r] & \Pic(\pd) \ar@{^{(}->}[u]^{\nu_d^*}
}
\end{equation}
in which the horizontal arrows are surjective and the vertical arrows are injective.
We will prove that the four Picard groups of diagram \eqref{E:4-Picard}
are generated by boundary line bundles and tautological line bundles, which we are now going to define.

In Section \ref{S:bound-div}, we describe the irreducible components of the
boundary divisor $\pdtb\setminus \pdt$.
Clearly, the boundary of $\pdtb$ is the pull-back via the morphism $\w{\Phi}_d:\pdtb\to \mgb$ of the boundary of $\mgb$.
Recall that $\displaystyle \mgb\setminus \mg=\bigcup_{i=0}^{[g/2]}\delta_i$, where
$\delta_0$ is the irreducible divisor whose generic point is an irreducible curve with one node
and, for $i=1,\ldots,[g/2]$, $\delta_i$ is the irreducible divisor whose generic point is the stable curve
with two irreducible components of genera $i$ and $g-i$ meeting in one point.
In Theorem \ref{bound-pdb}, we prove that $\w{\delta}_i:=\w{\Phi}_d^{-1}(\delta_i)$ is irreducible if either $i=0$ or $i=g/2$ or the number $\frac{2g-2}{(2g-2,d+g-1)}$ does not divide $(2i-1)$
and, otherwise, that $\w{\Phi}_d^{-1}(\delta_i)$ is the union of two irreducible divisors, that we call $\w{\delta}_i^1$ and $\w{\delta}_i^2$ (see Section \ref{S:bound-div} for the precise description of these two divisors).
Since $\pdtb$ is a smooth stack, the boundary divisors $\{\w{\delta}_i, \w{\delta}_i^1, \w{\delta}_i^2\}$ are Cartier
divisors and therefore they give rise to line bundles on $\pdtb$ that we denote by $\{\O(\w{\delta}_i),
\O(\w{\delta}_i^1), \O(\w{\delta}_i^2)\}$ and we call the \emph{boundary line bundles} of $\pdtb$.
Note that the irreducible components of the boundary of $\pdb$ are the divisors
$\ov{\delta}_i:=\nu_d(\w{\delta}_i)$, $\ov{\delta}_i^1:=\nu_d(\w{\delta}_i^1)$ and
$\ov{\delta}_i^2:=\nu_d(\w{\delta}_i^2)$. The associated line bundles $\{\O(\ov{\delta}_i),
\O(\ov{\delta}_i^1), \O(\ov{\delta}_i^2)\}$ are called boundary line bundles of $\pdb$ and clearly
we have that $\nu_d^*\O(\ov{\delta}_i)=\O(\w{\delta}_i)$,  $\nu_d^*\O(\ov{\delta}_i^1)=\O(\w{\delta}_i^1)$
and $\nu_d^*\O(\ov{\delta}_i^2)=\O(\w{\delta}_i^2)$ (see Corollary \ref{C:bound-rig}).

In Section \ref{S:taut-lb}, we introduce the line bundles $K_{1,0}$, $K_{0,1}$, $K_{-1,2}$ and
$\Lambda(m,n)$ (for $n,m\in \Z$) on $\pdtb$, which we call \emph{tautological line bundles}.
The tautological line bundles are defined in terms of the determinant
of cohomology $d_{\pi}(-)$ and of the Deligne pairing $\langle - , - \rangle_{\pi}$ applied to the universal family $\pi:\univb\to \pdtb$ (see \S \ref{Pic-stack} for the definition and basic properties
of the determinant of cohomology and of the Deligne pairing). More precisely, we define
\begin{equation*}
\begin{aligned}
&K_{1,0}:=\langle \omega_{\pi}, \omega_{\pi} \rangle_{\pi}, \\
&K_{0,1}:=\langle \omega_{\pi}, \L_d \rangle_{\pi}, \\
&K_{-1,2}:=\langle \L_d, \L_d \rangle_{\pi}, \\
&\Lambda(n,m)=d_{\pi}(\omega_{\pi}^n\otimes \L_d^m),
\end{aligned}
\end{equation*}
where $\omega_{\pi}$ is the relative dualizing sheaf for $\pi$ and $\L_d$ is the universal line bundle
on $\univb$. 
Following a strategy due to Mumford  \cite{Mum3}, we apply
the Grothendieck-Riemann-Roch theorem to the morphism $\pi:\univb\to \pdtb$ in order to produce relations among the
tautological line bundles, at least in the rational Picard group. In particular, we prove in Theorem \ref{T:taut-rel}
that all the tautological line bundles can be expressed in $\Pic(\pdtb)\otimes \Q$ in terms of $\Lambda(1,0)$, $\Lambda(0,1)$ and $\Lambda(1,1)$.
Therefore, we define the tautological subgroup $\Pict(\pdtb)\subseteq \Pic(\pdtb)$ as the subgroup generated
by the line bundles $\Lambda(1,0)$, $\Lambda(0,1)$, $\Lambda(1,1)$
 together with the boundary line bundles of $\pdtb$.
Similarly, we consider the subgroup $\Pict(\pdt)\subseteq
\Pic(\pdt)$ generated by the restriction of $\Lambda(1,0)$, $\Lambda(0,1)$, $\Lambda(1,1)$ to $\pdt$. 
Moreover, using the pull-back morphism $\nu_d^*$ (see diagram \eqref{E:4-Picard}), we can define
the tautological subgroups $\Pict(\pdb):=(\nu_d^*)^{-1}(\Pict(\pdb))\subseteq \Pic(\pdb)$ and
$\Pict(\pd):=(\nu_d^*)^{-1}(\Pict(\pd))\subseteq \Pic(\pd)$.

%We define the tautological subgroup $\Pict(\pdtb)\subseteq \Pic(\pdtb)$ as the subgroup generated by the tautological line bundles together with the boundary line bundles of $\pdtb$. Similarly, we can restrict the tautological line bundles to $\pdt$ and consider the subgroup $\Pict(\pdt)\subseteq \Pic(\pdt)$ generated by them. Moreover, using the pull-back morphism $\nu_d^*$ (see diagram \eqref{E:4-Picard}), we can define the tautological subgroups $\Pict(\pdb):=(\nu_d^*)^{-1}(\Pict(\pdb))\subseteq \Pic(\pdb)$ and $\Pict(\pd):=(\nu_d^*)^{-1}(\Pict(\pd))\subseteq \Pic(\pd)$. Following a strategy due to Mumford  \cite{Mum3}, we next apply the Grothendieck-Riemann-Roch theorem to the morphism $\pi:\univb\to \pdtb$ in order to produce relations among the tautological line bundles. In particular, we prove in Theorem \ref{T:taut-rel} that all the tautological line bundles can be expressed in terms of $\Lambda(1,0)$, $\Lambda(0,1)$ and $\Lambda(1,1)$. In particular, the tautological subgroup $\Pict(\pdt)$ (resp. $\Pict(\pdtb)$) is generated by the three line bundles $\Lambda(1,0)$, $\Lambda(0,1)$ and $\Lambda(1,1)$ (resp. and the boundary line bundles).

\vspace{0,2cm}

After these preliminaries, we can now state the main results of this paper, concerning the Picard groups of
$\pdt$ and $\pd$ and of their compactifications $\pdtb$ and $\pdb$. We prove that all the Picard groups in question
are free and generated by tautological line bundles and boundary line bundles (if any). More precisely, we have the following.

\begin{theoremalpha}\label{T:MainThmA}
Assume that $g\geq 3$.
\begin{enumerate}[(i)]
\item \label{T:MainThmA1} The Picard group of $\pdt$ is freely generated by
$\Lambda(1,0)$, $\Lambda(0,1)$ and $\Lambda(1,1)$.
\item \label{T:MainThmA2} The Picard group of $\pdtb$ is freely generated by the boundary line bundles
and the tautological line bundles $\Lambda(1,0)$, $\Lambda(0,1)$ and $\Lambda(1,1)$.
\end{enumerate}
\end{theoremalpha}

\begin{theoremalpha}\label{T:MainThmB}
Assume that $g\geq 3$.
\begin{enumerate}[(i)]
\item \label{T:MainThmB1} The Picard group of $\pd$ is freely generated by the tautological line bundles
$\Lambda(1,0)$ and
\begin{equation}\label{E:lb-Xi}
\Xi:= \Lambda(0,1)^{ \frac{d+g-1}{(d+g-1,d-g+1)}}\otimes \Lambda(1,1)^{-\frac{d-g+1}{(d+g-1,d-g+1)}}.
\end{equation}
\item \label{T:MainThmB2} The Picard group of $\pdb$ is freely generated by the boundary line bundles and the tautological
line bundles $\Lambda(1,0)$ and $\Xi$.
%$$\Xi:= \Lambda(0,1)^{ \frac{d+g-1}{(d+g-1,d-g+1)}}\otimes \Lambda(1,1)^{-\frac{d-g+1}{(d+g-1,d-g+1)}}.$$
\end{enumerate}
\end{theoremalpha}

Let us now sketch the strategy that we use to prove Theorems \ref{T:MainThmA} and \ref{T:MainThmB}.
Since the stack $\pdtb$ is smooth we have a natural exact sequence
\begin{equation}\label{E:exaseq-Pic}
\bigoplus_{\stackrel{\kdg \: \nmid \:\: (2i-1)}{\text{ or } i=g/2 \text{ or } i=0}}\langle \O(\w{\delta}_i)\rangle
\bigoplus_{\stackrel{\kdg \mid (2i-1)}{\text{ and } i\neq 0,  g/2}}\langle \O(\w{\delta}_i^1),\O(\w{\delta}_i^2)\rangle\to
\Pic(\pdtb)\to \Pic(\pdt)\to 0.
\end{equation}
In Theorem \ref{T:ex-seq}, we prove that the above exact sequence is also exact on the left, or in other words
that the boundary line bundles are linearly independent in the Picard group
of $\pdtb$. In order to prove this, we use the same strategy used by Arbarello-Cornalba in \cite{arbcorn} to prove the
analogous statement for the boundary line bundles of $\mgb$:  we construct some test curves $\w{F}_j\to \pdtb$,
in number equal to the number of boundary line bundles,
%from smooth and irreducible curves $\w{F}_j$
and prove that the intersection matrix between these test curves $\w{F}_j$ and the boundary line bundles of $\pdtb$
is non-degenerate. This reduces the proof of Theorem \ref{T:MainThmA}\eqref{T:MainThmA2} to the proof of Theorem
\ref{T:MainThmA}\eqref{T:MainThmA1}.
%Indeed, the test curves $\w{F}_j\to \pdtb$ that we construct are liftings of the test curves $F_j\to \mgb$
%constructed by Arbarello-Cornalba  in \cite{arbcorn}.

Moreover, using the fact that the pull-back morphism $\nu_d^*:\Pic(\pdb)\to \Pic(\pdtb)$ is injective and it sends the
boundary line bundles of $\pdb$ into the boundary line bundles of $\pdtb$, we get that also the boundary line bundles of $\pdb$ are linearly independent (see Corollary \ref{C:ex-seq-rig}), or in other words that we have an exact sequence:
 \begin{equation}\label{E:exaseq-Pic-rig}
0\to\bigoplus_{\stackrel{\kdg \: \nmid\:\: (2i-1)}{\text{ or } i=g/2 \text{ or } i=0}}\langle \O(\ov{\delta}_i)\rangle
\bigoplus_{\stackrel{\kdg \mid (2i-1)}{\text{ and } i\neq 0, g/2}}\langle \O(\ov{\delta}_i^1),\O(\ov{\delta}_i^2)\rangle\to
\Pic(\pdb)\to \Pic(\pd)\to 0.
\end{equation}
 This reduces the proof of Theorem \ref{T:MainThmB}\eqref{T:MainThmB2} to the proof of Theorem \ref{T:MainThmB}\eqref{T:MainThmB1}.

The Picard groups of $\pdt$ and of $\pd$ are related via the following exact sequence coming from the
Leray spectral sequence for the \'etale sheaf $\Gm$ with respect to the rigidification map $\nu_d:\pdt\to \pd$
(see \eqref{succ-Pic}):
$$0\to \Pic(\pd)\stackrel{\nu_d^*}{\longrightarrow} \Pic(\pdt) \stackrel{\res}{\longrightarrow}
\Pic B\Gm=\Hom(\Gm, \Gm)\cong \Z \stackrel{\obs}{\longrightarrow}  \Br(\pd).$$
The map ${\rm res}$ is the restriction to the fibers of $\nu_d$ (which are isomorphic to the classifying stack $B\Gm$
 of the multiplicative group $\Gm$) and ${\rm obs}$ sends $1\in \Z$ into the class $[\nu_d]$ of the $\Gm$-gerbe
 $\nu_d:\pdt\to \pd$ in the cohomological Brauer group $\Br(\pd):=H^2_{{\rm \acute et}}(\pd, \Gm)$ of $\pd$.
 In Theorem \ref{order-gerbe}, we prove that the order of $[\nu_d]$ is
the greatest common divisor $(d+1-g, 2g-2)$. In proving this, we interpret in Proposition \ref{Poinc-order}
the order of $[\nu_d]$ as the smallest natural number $m$ for which there exists an \textit{$m$-Poincar\'e line bundle} (in the sense of Definition \ref{m-Poincare}) on the universal family $\univ$ over $\pd$.
Using Proposition \ref{Poinc-order}, Theorem \ref{order-gerbe} follows then from a result of Kouvidakis
(see \cite[p. 514]{kou2}).
Note also that by combining Theorem \ref{order-gerbe} and Proposition \ref{Poinc-order}, we recover the well-known
result of Mestrano-Ramanan (\cite[Cor. 2.9]{MR}): there exists a Poincar\'e line bundle on $\univ$ if and only if
$(d+1-g,2g-2)=1$. We conjecture that the cohomological Brauer group $\Br(\pd)$ is generated by $[\nu_d]$
(see Conjecture \ref{C:prob-Brauer} and the discussion following it).

From the computation of the order of $[\nu_d]$ and  the  above exact sequence, we get that ${\rm res}(\Pic(\pdt))=(2g-2,d+1-g)\cdot \Z$.
Moreover, we compute the values of the map ${\rm res}$ on the generators of the tautological subgroup $\Pict(\pdt)\subseteq \Pic(\pdt)$ in Lemma \ref{exis-linebun} and deduce that
${\rm res}(\Pict(\pdt))=(2g-2,d+1-g)\cdot \Z$.  This easily reduces the proof of Theorem \ref{T:MainThmA}\eqref{T:MainThmA1}
to the proof of Theorem \ref{T:MainThmB}\eqref{T:MainThmB1}. Furthermore, it shows that $\Pict(\pd)$ is generated by $\Lambda(1,0)$ and the line bundle $\Xi$ of \eqref{E:lb-Xi}.

The Picard group of $\pd$ can be determined with the help of the following exact sequence
\begin{equation}\label{E:rigid-NS}
0\to  \Pic(\mg)\stackrel{\Phi_d^*}{\longrightarrow}\Pic(\pd)  \stackrel{\chi_d}{\longrightarrow} \Z,
\end{equation}
where the map $\chi_d$ sends a line bundle $\L\in \Pic(\pd)$ to the integer $m\in \Z$ such that the class of the restriction of $\L$ to the fiber
$\Phi_d^{-1}(C)=J^d(C)$ in the N\'eron-Severi group $NS(J^d(C))$ is isomorphic to $m$ times the class $\theta_C$ of the theta divisor (see Section \ref{S:Pic-rigid}
for more details). A well-known result of Harer and Arbarello-Cornalba says that $\Pic(\mg)$ is freely generated by the Hodge line bundle $\Lambda$ if $g\geq 3$ (see
Theorem \ref{pic-mg}) and we prove in Lemma \ref{L:comp-taut} that $\Phi_d^*(\Lambda)=\Lambda(1,0)$. On the other hand, a result of Kouvidakis in \cite{kouvidakis}
implies that $\Im(\chi_d)\subseteq \displaystyle \frac{2g-2}{(2g-2,d+1-g)}\cdot \Z$. In Theorem \ref{T:chi-taut}, we compute the values of $\chi_d$ on the generators of
the tautological subgroup $\Pict(\pd)\subseteq \Pic(\pd)$ and we deduce that $\displaystyle \chi_d(\Pict(\pd))= \frac{2g-2}{(2g-2,d+1-g)}\cdot \Z$.
From the exact sequence \eqref{E:rigid-NS}, we deduce now that $\Pict(\pd)=\Pic(\pd)$ is free of rank two; Theorem \ref{T:MainThmB}\eqref{T:MainThmB1} now follows.

\vspace{0,3cm}

In the last Section of the paper, we relate the Picard group of the moduli stack $\pdb$
with  the divisor class group $\Cl(\ov{J}_{d,g})$ of its moduli scheme
$\ov{J}_{d,g}$, which was computed by Fontanari \cite{fontanari} based upon the work of Kouvidakis \cite{kouvidakis} on the Picard group
of the open subscheme $J_{d,g}^0\subset J_{d,g}$ consisting of pairs $(C,L)$ such that  $C$ does not have non-trivial automorphisms.
Fontanari proved in \cite{fontanari} that the boundary of $\ov{J}_{d,g}$ is the union of the irreducible divisors  $\w{\Delta}_i:=\phi_d^{-1}(\Delta_i)$ for $i=1,\ldots,[g/2]$,
where $\phi_d:\ov{J}_{d,g}\to \ov{M}_g$ is the natural map towards the moduli scheme of stable curves of genus $g$ and $\Delta_i\subseteq \ov{M}_g$ is, as usual, the irreducible divisor
of $\ov{M}_g$ whose generic point is an irreducible curve with one node if $i=0$ or, for $i>0$,  the union of two irreducible components of genera $i$ and $g-i$ meeting in one point.
Moreover, Fontanari proved that there is an exact sequence
\begin{equation}\label{E:seq-Weildiv}
0 \to \bigoplus_{i=0}^{[g/2]} \Z\cdot \w{\Delta_i}\to \Cl(\ov{J}_{d,g})\to \Cl(J_{d,g})\to  0,
\end{equation}
where the last map is the restriction map and the first map sends each $\w{\Delta}_i$ into its class in $\Cl(\ov{J}_{d,g})$.
The Picard group of $\pdb$ and the divisor class group of $\ov{J}_{d,g}$ are related by the pull-back via the natural map $\Psi_d:\pdb\to \ov{J}_{d,g}$,
which induces a map from the exact sequence \eqref{E:exaseq-Pic-rig} into the exact sequence \eqref{E:seq-Weildiv}.
In Section \ref{S:rel-modspace} we prove the following result.

\begin{theoremalpha}\label{T:MainThmC}
The pull-back map $\Psi_d^*:\Cl(\ov{J}_{d,g})\to \Pic(\pdb)$ induced by the natural map $\Psi_d:\pdb\to \ov{J}_{d,g}$ fits into a commutative diagram with exact rows
$$\xymatrix{
0\ar[r] & \bigoplus_{i=0}^{[g/2]} \Z\cdot \w{\Delta_i} \ar[r] \ar[d]^{\alpha_d}&  \Cl(\ov{J}_{d,g})\ar[r] \ar^{\Psi_d^*}[d]&  \Cl(J_{d,g})\ar[r] \ar[d]^{\beta_d}&  0\\
0 \ar[r] & \bigoplus_{\stackrel{\kdg  \nmid\: \:2i-1}{\text{ or } i=g/2}}\langle \O(\ov{\delta}_i)\rangle
\bigoplus_{\stackrel{\kdg \mid 2i-1}{\text{ and } i\neq g/2}}\langle \O(\ov{\delta}_i^1),\O(\ov{\delta}_i^2)\rangle\ar[r]&
\Pic(\pdb)\ar[r] & \Pic(\pd)\ar[r] & 0,
}$$
such that:
\begin{enumerate}[(i)]
\item  \label{T:MainThmC1} the map $\beta_d$ is an isomorphism;
\item \label{T:MainThmC2} the map $\alpha_d$ satisfies
$$\alpha_d(\w{\Delta_i})=\begin{cases}
\O(\ov{\delta}_i) & \text{ if } \kdg\nmid (2i-1), \\
\O(\ov{\delta}_i^1)+\O(\ov{\delta}_i^2) & \text{ if } \kdg \mid (2i-1) \text{ and } i\neq g/2,\\
\O(2\ov{\delta}_i) & \text{ if } \kdg \mid (2i-1) \text{ and }Êi= g/2.\\
%\w{\delta}_i & \text{ if } i=0 \text{ or } i>1 \text{ and } \kdg \nmid 2i-1, \\
%2\w{\delta}_i & \text{ if } i=1 \text{ and } \kdg \nmid 1, \\
%\w{\delta}_i^1+\w{\delta}_i^2 & \text{ if } i> 1 \text{ and } \kdg \: \mid \: 2i-1, \\
%2\w{\delta}_i^1+2\w{\delta}_i^2 & \text{ if } i=1 \text{Êand }Ê\kdg \: \mid \: 1. \\
\end{cases}$$
\end{enumerate}
\end{theoremalpha}

\vspace{0,3cm}

It is likely that the same techniques used in this paper could lead to the computation of
the Picard group of the degree-$d$ compactified universal Jacobian stack $\ov{{\mathcal Jac}}_{d,g,n}$ over the stack
$\ov{\mathcal{M}}_{g,n}$ of $n$-pointed stable curves of genus $g$ constructed in \cite{melo2}
and of the universal vector bundle over $\mgb$ constructed in \cite{Pan}.
We plan to come back to these problems in a near future.

\vspace{0,2cm}

The paper is organized as follows.
In Section \ref{S:prelimi}, we summarize the known properties of the stacks
$\pdtb$ and $\pdb$ as well as the properties of their moduli scheme
 $\ov{J}_{d,g}$ (see \ref{desc-stacks}).
Moreover, we recall some basic facts about the Picard group
of a stack and how to construct natural line bundles on moduli stacks by using the determinant of
cohomology and the Deligne pairing (see \ref{Pic-stack}). Finally, we recall the computation of the Picard group of the
stack $\mgb$ of stable curves of genus $g$ by Harer and Arbarello-Cornalba (see \ref{sec-pic-mg}).
In Section \ref{S:bound-div}, we describe the boundary divisors of $\pdtb$ and we explain how they are related to
the pull-back of the boundary divisors of $\mgb$. In Section \ref{S:indip-bound}, we show that the line bundles on $\pdtb$
associated to the boundary divisors are linearly independent. In Section \ref{S:taut-lb}, we introduce the tautological
line bundles on $\pdtb$ and we study the relations among them. In Section \ref{S:comp-Pic}, we compare the Picard groups
of $\pdt$ and of $\pd$ using the Leray's spectral sequence associated to the rigidification map
$\nu_d:\pdt\to \pd$. Moreover, we compute the order of the $\Gm$-gerbe $\nu_d$ in the Brauer group of $\pd$.
In Section \ref{S:Pic-rigid}, we compute the Picard group of $\pd$ using the fibration $\Phi_d:\pd\to \mg$. Moreover, 
we investigate the relation between the line bundle $\Xi$ and the universal theta divisor (see \ref{S:rela-theta}) and we prove that the pull-back via the Abel-Jacobi map  
provides an isomorphism between the Picard groups of $\pdt$ and of the $d$-th symmetric product of the universal curve $\M_{g,1}\to \M_g$, when $d>2g-2$ (see \ref{S:univ-sym}).
In Section \ref{S:rel-modspace}, we compare the Picard group of $\pdb$ with the divisor class group of its moduli
scheme $\ov{J}_{d,g}$.

\subsection{Relation to algebraic topology}\label{S:comp-top}
After a preliminary version of this manuscript has been posted on arXiv, J. Ebert and O. Randal-Williams posted
on arXiv a preliminary version of the paper \cite{ERW}, which contains, among other things, some results that are closely related to Theorem
\ref{T:MainThmA}\eqref{T:MainThmA1} and Theorem \ref{T:MainThmB}\eqref{T:MainThmB1} in the case when our base field $k$ is the field of complex numbers. We now explain the relation between our results  and the results of \cite{ERW}.

In \cite{ERW}, the authors introduce two holomorphic stacks $\Hol_g^d$ and $\Pic_g^k$, defined as follows
(see \cite[Sec. 4.1]{ERW} for details):
$\Hol_g^d$ is the holomorphic stack whose fibers over a topological space $B$ consists of families of Riemann surfaces
$\pi:E\to B$ of genus $g$ equipped with a fiberwise holomorphic line bundle $L\to E$ of relative degree $d$; $\Pic_g^d$ is the holomorphic stack parametrizing families of Riemann surfaces of genus
$g$ equipped with a section of the associated bundle of Jacobian varieties of degree $d$. There is a morphism
$\phi_g^d: \Hol_g^d \to \Pic_g^d$ defined by sending a fiberwise holomorphic line bundle to its isomorphism class.
It turns out that $\phi_g^d$ is a gerbe with band $\CC^*$ (see \cite[Thm. 4.5]{ERW}).

The relation with our algebraic stacks $\pdt$ and $\pd$ (over the complex numbers) is provided by a commutative diagram
\begin{equation}\label{E:comp-stacks}
\xymatrix{
(\pdt)^{\rm an} \ar[r]\ar[d]^{\nu_d} & \Hol_g^d \ar[d]^{\phi_g^d} \\
(\pd)^{\rm an} \ar[r] & \Pic_g^d
}
\end{equation}
where $(\pdt)^{\rm an}$ and $(\pd)^{\rm an}$ are the analytifications of the complex algebraic stacks $\pdt$ and $\pd$.
The horizontal maps are most likely isomorphisms although we have not checked this in detail.

The authors of loc. cit. consider tautological classes $\kappa_{i,j}\in H^{2i+2j}(\Hol_g^d,\Z)$ for $i\geq -1$ and $j\geq 0$
defined by associating to every element $(\pi:E\to B, L\to E)\in \Hol_g^d(B)$ the  cohomology class
\begin{equation}\label{E:k-classes}
\kappa_{i,j}(\pi:E\to B, L\to E):= \pi_{!}(c_1(T^vE)^{i+1}\cdot c_1(L)^j)\in H^{2i+2j}(B, \Z),
\end{equation}
where $T^vE$ is the relative tangent line bundle of the family $\pi:E\to B$ of Riemann surfaces, which is of course dual
to the sheaf $\omega_{\pi}$ of relative differentials of $\pi$. In particular, the classes $\kappa_{i,0}$ are the pull-back to
$\Hol_g^d$ of the Mumford-Morita-Miller classes $\kappa_i$ on $\mg$.
Moreover, one denotes by $\lambda$ the pull-back to $\Hol_g^d$ of the Hodge class on $\mg$.

Among other beautiful results, Ebert and Randal-Williams compute the analytic N\'eron-Severi group $NS$, the topological
Picard group $\Pic_{\rm top}$ and the second cohomology group with integer values $H^2(-, \Z)$ of the above two stacks
(see \cite[Thm. C, Thm. E]{ERW}), under the assumption that $g\geq 6$.

\begin{thm}[Ebert, Randal-Williams]\label{T:ERW-thm}
Assume that $g\geq 6$. Then
\begin{enumerate}[(i)]
\item \label{T:ERW-thm1} $NS(\Hol_g^d)=\Pic_{\rm top}(\Hol_g^d)=H^2(\Hol_g^d,\Z)$ is freely generated by $\lambda$, $\kappa_{-1,2}$, and
    $\displaystyle \zeta:=\frac{\kappa_{0,1}-\kappa_{-1,2}}{2}$.
\item \label{T:ERW-thm2} $NS(\Pic_g^d)=\Pic_{\rm top}(\Hol_g^d)=H^2(\Pic_g^d,\Z)$ is the subgroup of $H^2(\Hol_g^d,\Z)$ generated by $\lambda$  and
$$\eta:=\frac{d \,\kappa_{0,1}+(g-1)\kappa_{-1,2}}{(2g-2, g+d-1)}.$$
\end{enumerate}
\end{thm}

The diagram \eqref{E:comp-stacks} gives two natural homomorphisms
\begin{equation}\label{E:comp-Pic}
\begin{aligned}
& c_1: \Pic(\pdt)\to H^2(\Hol_g^d,\Z),\\
& c_1: \Pic(\pd)\to H^2(\Pic_g^d,\Z).\\
\end{aligned}
\end{equation}

The next result is obtained by comparing Theorems \ref{T:MainThmA}\eqref{T:MainThmA1} and  \ref{T:MainThmB}\eqref{T:MainThmB1} with Theorem
\ref{T:ERW-thm}.

\begin{cor}\label{C:Pic-H2}
Assume that $g\geq 6$. The homomorphisms of \eqref{E:comp-Pic} are isomorphisms.
\end{cor}
\begin{proof}
The fact that the first map in \eqref{E:comp-Pic} is an isomorphism follows by comparing Theorem \ref{T:MainThmA}\eqref{T:MainThmA1} and Theorem \ref{T:ERW-thm}\eqref{T:ERW-thm1} by mean of the formulas
\begin{equation*}\tag{*}
\begin{sis}
& c_1(\Lambda(1,0))=\lambda,\\
& c_1(\Lambda(1,1))=\frac{\kappa_{-1,2}-\kappa_{0,1}}{2}=-\zeta,\\
& c_1(\Lambda(0,1))=\frac{\kappa_{-1,2}+\kappa_{0,1}}{2}+\lambda=\zeta+\kappa_{-1,2}+\lambda,
\end{sis}
\end{equation*}
where the first formula follows from Lemma \ref{L:comp-taut} and the last two formulas follow from Theorem \ref{T:taut-rel}
together with the facts that $c_1(K_{-1,2})=\kappa_{-1,2}$ and $c_1(K(0,1))=-\kappa_{0,1}$. Note that the minus sign appearing in this last equality is  due to the fact that in defining the classes $\kappa_{i,j}\in
H^2(\Hol_g^d,\Z)$ (see \eqref{E:k-classes}), Ebert and Randal-Williams use the relative tangent sheaf while our definition \eqref{E:taut-lb} of the tautological line bundles $K_{i,j}\in \Pic(\pdt)$  uses its dual sheaf, namely the sheaf of relative differentials.

The fact that the second map in \eqref{E:comp-Pic} is an isomorphism follows by comparing Theorem \ref{T:MainThmB}\eqref{T:MainThmB1} and Theorem \ref{T:ERW-thm}\eqref{T:ERW-thm2} using the formula
$$c_1(\Xi)=\frac{(d+g-1)c_1(\Lambda(0,1))-(d-g+1)c_1(\Lambda(1,1))}{(d+g-1,d-g+1)}
=\eta+\frac{d+g-1}{(d+g-1,d-g+1)}\lambda.$$

\end{proof}

%\vspace{0,2cm}

\emph{Acknowledgements.}

The first author would like to thank Lucia Caporaso for suggesting this problem to her while a PhD student of hers.
We thank Alexis Kouvidakis for pointing out to us the reference \cite[p. 514]{kou2}
which  simplified the proof of Theorem \ref{order-gerbe}. We are grateful to  Oscar Randal-Williams for some enlightening discussions on the paper \cite{ERW} and for pointing out to us a mistake in a previous version of Lemma \ref{L:Xi-theta}. We thank F. Poma, M. Talpo and F. Tonini for pointing out that the proof of Theorem \ref{T:taut-rel} gives a priori only relations in the rational Picard group.
We are grateful to the referee for suggesting an interesting connection between the Picard groups of the universal Jacobian and of the universal symmetric product, which is now included in subsection \ref{S:univ-sym}.
We have benefited from useful conversations with Gilberto Bini, Silvia Brannetti and Claudio Fontanari.

%Throughout this paper, we will use the following

\subsection*{Notations}

\begin{notation}
We fix two integers $g\geq 2$ and $d$: $g$ will always denote the genus of the curves
and $d$ the degree of the Jacobian varieties.
Given two integers $m$ and $n$, we set $(n,m)$ for the greatest common divisor of $n$ and $m$.
In particular the greatest common divisor
$$(2g-2,d+1-g)=(2g-2,d-1+g)=(d+1-g,d-1+g)$$
will appear often in what follows. Similarly the number
\begin{equation}\label{def-k}
\kdg:=\frac{2g-2}{(2g-2,d+g-1)}
\end{equation}
will appear repeatedly throughout the paper and hence it deserves a special notation.
\end{notation}

\begin{notation}
 We work over an algebraically closed field $k$ of characteristic $0$.
All the schemes and stacks we will deal with are of finite type over $k$.

There are two places in our work where the assumption on the characteristic of $k$ is used.
The first one is the explicit computation of the Picard group of $\mgb$ by Harer and Arbarello-Cornalba (see Theorem \ref{pic-mg} for the precise statement), which is known to be true only in characteristic zero (in positive characteristic, the same statement remains true for the {\em rational} Picard group of $\mgb$ by the work of Moriwaki in \cite{Mor}). The second one is a result of Kouvidakis \cite{kouvidakis} (see Theorem \ref{T:Kouvi}), whose proof over the complex numbers does not immediately extend to a base field $k$ of positive characteristics\footnote{We thank F. Poma, M. Talpo and F. Tonini for pointing out this to us.}.

%The only place where the assumption on the characteristic of $k$ is used is the fact that we use the explicit determination of the Picard group of $\mgb$ by Harer and Arbarello-Cornalba (see Theorem \ref{pic-mg} for the precise statement), which is known to be true only in characteristic zero. However, in positive characteristic, the same statement remains true for the {\em rational} Picard group of $\mgb$ by the work of Moriwaki in \cite{Mor}. Therefore, all our statements hold in positive characteristic for the rational Picard groups.
\end{notation}

\begin{notation}
 We will often assume, for simplicity,  that $g\geq 3$. This is the case for two of the main results of this paper, namely Theorems \ref{T:MainThmA} and \ref{T:MainThmB}.

The reason for this assumption  is that the Picard group of $\mgb$ is freely generated by the Hodge line bundle
$\Lambda$ and the boundary line bundles $\{\O(\delta_0),\ldots,\O(\delta_{[g/2]})\}$ if $g\geq 3$ (see Theorem \ref{pic-mg})
while if $g=2$ then $\Pic(\mgb)$ is still generated by $\Lambda$ and the boundary line bundles but with the relation $\Lambda^{10}\otimes \O(-\delta_0-2\delta_1)=0$ (see \ref{sec-pic-mg}).
Indeed, all the above mentioned results continue to hold for $g=2$ if we add the relation pull-backed from the relation
$\Lambda^{10}\otimes \O(-\delta_0-2\delta_1)=0$ in $\Pic(\ov{\mathcal M}_2)$ or its image $\Lambda^{10}=0$ in
$\Pic({\mathcal M}_2)$.
\end{notation}

\section{Preliminaries}\label{S:prelimi}

\begin{nota}{\emph{The stacks $\pdtb$ and $\pdb$ and their moduli space $\ov J_{d,g}$}}\label{desc-stacks}

Let $\pdt$ be the universal Jacobian stack over the moduli stack $\mg$ of smooth
curves of genus $g$. The fiber of $\pdt$ over a scheme $S$ is the groupoid whose objects are
families of smooth curves $\C\to S$ endowed with a line bundle $\L$ over $\C$ of relative degree $d$ over $S$
and whose arrows are the obvious isomorphisms.
$\pdt$ is a smooth irreducible (Artin)
algebraic stack of dimension $4g-4$ endowed with a natural forgetful morphism $\w\Phi_d:\pdt\to \mg$.

%The multiplicative group $\Gm$ naturally injects into the automorphism group of every object $(\C\to S, \L)\in \pdt(S)$
%as multiplication by scalars on $\L$ (this makes $\pdt$ into a $\Gm$-stack in the sense of \cite[Def. 3.1]{Hof1} or, equivalently,
%it endows $\pdt$ with a $\Gm$-2-structure in the sense of \cite[Appendix C.1]{AGV}).

The multiplicative group $\Gm$ naturally injects into the automorphism group of every object $(\C\to S, \L)\in \pdt(S)$
as multiplication by scalars on $\L$, endowing $\pdt$ with the structure of a $\Gm$-stack in the sense of \cite[Def. 3.1]{Hof1} or, equivalently,
with a $\Gm$-2-structure in the sense of \cite[Appendix C.1]{AGV}.

There is a canonical procedure to remove such automorphisms, called $\Gm$-{\it rigidification} (see \cite[Sec. 5]{ACV}, \cite[Sec. 5]{Rom}
and \cite[Appendix C]{AGV}).
The outcome is a new stack
$\pd:=\pdt \fatslash \Gm$  together with a smooth and surjective map
$\nu_d:\pdt\to \pd$. Indeed, the map $\nu_d$ makes $\pdt$ into a gerbe banded by $\Gm$ (or a $\Gm$-gerbe in short)
over $\pd$ (we refer to \cite{Gir} for the theory of gerbes).
The forgetful map
$\w\Phi_d$ factors via $\nu_d$ and we get a commutative diagram
$$\xymatrix{
\pdt \ar[dr]_{\w\Phi_d}\ar[rr]^{\nu_d} & &\pd \ar[dl]^{\Phi_d}\\
& \mg &\\
}
$$
The new stack $\pd$ is a smooth, irreducible and separated Deligne-Mumford stack of dimension $4g-3$ and the map
$\Phi_d$ is representable.

%Modular compactifications of the stacks $\pdt$ and $\pd$ have been described by Melo in \cite{melo},
%based upon previous results of Caporaso (see \cite{cap} and \cite{capneron}).
%Let us review these compactifications.

A modular compactification of the stacks $\pdt$ and $\pd$ was described by Caporaso in \cite{capneron} for some degrees and later by Melo in \cite{melo} for the general case,
based upon previous work of Caporaso in \cite{cap}.
Let us review this compactification.

%In particular, in loc. cit., the author describes $\pdt$ and $\pd$ (note that $\pdt$ is denoted by $\mathcal{G}_{d,g}$ in loc.
%cit.) as quotient stacks and find
%a modular compactification of them (see \cite[Thm. 3.1, Prop. 4.1]{melo}).

%For later use, we need to recall the results in \cite{melo}.
%To describe it, we need to recall the following definitions.

\begin{defi}\label{quasi-stable}\cite[Sec. 3.3]{cap}
A connected, projective nodal curve $X$ is said to be \emph{quasistable} if it is (Deligne-Mumford) semistable and
if the exceptional components of $X$ do not meet.
%one of the following equivalent conditions is satisfied:
%\begin{enumerate}[(i)]
%\item $X$ is Deligne-Mumford semistable curve such that the exceptional components of $X$ do not meet;
%\item $X$ is obtained from its stabilization $X^{{\rm st}}$ by blowing up some of its nodes;
%\item Every connected subcurve $Z$ of $X$ having arithmetic genus zero and meeting the complementary subcurve $Z^c:=\ov{X\setminus Z}$
%in less than three points is isomorphic to $\P^1$ and meets the complementary subcurve in two points.
%\end{enumerate}
%The exceptional locus of $X$, denoted by $X_{\rm exc}$, is the union of the exceptional components of $X$.
\end{defi}

\begin{defi} \label{balanced}\cite[Def. 3.5]{BMV}
Let $X$ be a quasistable curve of genus $g\ge 2$. A line bundle $L$ of degree $d$ on $X$ (or its multidegree) is said to be 
%\begin{enumerate}
 \emph{properly balanced} if
 \begin{itemize}
 \item
 for every subcurve $Z$ of $X$ the following (``Basic Inequality'') holds
\begin{equation}\label{basic}
m_Z(d):=\frac{d w_Z}{2g-2}-\frac {k_Z}2\le \deg_ZL\le\frac{d w_Z}{2g-2}+\frac{k_Z}2:=M_Z(d),
\end{equation}
where $w_Z:=\deg_Z(\omega_X)$ and $k_Z:=\sharp(Z\cap \ov{X\setminus Z})$.
\item  $\deg_EL=1$ for every exceptional component $E$ of $X$.
\end{itemize}
%The set of properly balanced line bundles of degree $d$ of a curve $X$ is denoted by $B_X^d$.
%\item We say that $L$ (or its multidegree) is \emph{strictly balanced} if it is properly balanced and if for each proper subcurve
%$Z$ of $X$ such that $\deg_ZL=m_Z(d)$, the intersection $Z\cap Z^c$ is contained in the exceptional locus $X_{\rm exc}$ of $X$.
%\end{enumerate}
\end{defi}

\begin{remark}\label{rmk-ineq}
%It is easy to check that:
%\begin{enumerate}[(i)]
%\item The basic inequality (\ref{basic}) for $Z$ is equivalent to the one for the complementary subcurve $Z^c:=\ov{X\setminus Z}$;
%\item If $Z$ is a disjoint union of the subcurves $Z_1$ and $Z_2$, then the basic inequality (\ref{basic}) for $Z_1$ and $Z_2$ implies the one for $Z$.
%\end{enumerate}
%In particular, it is enough to check the basic inequality (\ref{basic}) for all subcurves $Z$ such that $Z$ and $Z^c$ are connected.
In order to check that a line bundle is properly balanced, it is enough  to check the basic inequality (\ref{basic}) for all subcurves $Z$ such that $Z$ and $Z^c$ are connected (see \cite[Rmk. 3.8]{BMV}).
\end{remark}

%\begin{defi}\label{d-gen}
%A stable curve $X$ is said to be \emph{$d$-general} if and only if every properly balanced line bundle on $X$ is strictly balanced. We denote by $\mgbd\subset \mgb$ the open substack whose sections are families of $d$-general curves.
%\end{defi}

%For later use, we need to recall the description of the locus $\mgb\setminus \mgbd$ of $d$-special curves, given in \cite[Prop. 2.2]{melo}. Recall that a \emph{vine curve} of genus $g$ and type $(g_1,g_2)$ is a stable curve of genus $g$ formed by two smooth curves of genus $g_1$ and $g_2$ meeting at $k:=g-g_1-g_2+1$ points.

%\begin{prop}[Melo]\label{d-special}
%\noindent
%\begin{enumerate}[(i)]
%\item A stable curve $C$ is $d$-special (i.e. it belongs to  $\mgb\setminus \mgbd$) if and only if it is a specialization of a $d$-special vine curve.
%\item A stable vine curve of genus $g$ and type $(i,g-i-k+1)$ is $d$-special  if and only if
%\begin{equation*}
%k_{d,g}:=\frac{2g-2}{(2g-2,d-g+1)} \mid (2i-2+k).
%\end{equation*}
%\end{enumerate}
%\end{prop}

Let $\pdtb$ be the category fibered in groupoids whose fiber over a scheme $S$ consists of the groupoid whose objects are families of quasistable curves $\C\to S$ endowed with a line bundle $\L$ of relative degree $d$,
whose restriction to each geometric fiber is properly balanced (we say that $\L$ is properly balanced), and whose arrows are the obvious isomorphisms.
The multiplicative group $\Gm$ injects into the automorphism group of every object $(\C\to S, \L)\in \pdtb(S)$
as multiplication by scalars on $\L$.
%The rigidification $\nu_d:\pdtb\to \pdb:=\pdtb\fatslash \Gm$ is a $\Gm$-gerbe.
As in the smooth case, the rigidification morphism $\nu_d:\pdtb\to \pdb:=\pdtb\fatslash \Gm$ endows $\pdtb$ with the structure of a $\Gm$-gerbe over $\pdb$.
%$\pdtb\fatslash \Gm$.

There is a natural morphism of category fibered in groupoids $\w{\Phi}_d:\pdtb \to \mgb$ obtained by sending
 $(\C\to S, \L)\in \pdtb(S)$ into the stabilization $\C^{\rm st}\to S\in \mgb(S)$ of the family of quasi-stable
 curves $\C\to S$. Clearly, the morphism $\w{\Phi}_d$ factors through a morphism $\Phi_d:\pdb\to \mgb$.

The following theorem summarizes the known properties of $\pdtb$ and of $\pdb$, proved in \cite{capneron}
under the assumption that $(d+g-1,2g-2)=1$ and in \cite{melo} for arbitrary $d$, and of their moduli space $\ov J_{d,g}$ constructed in \cite{cap}.
%(note that in \cite{melo} the stacks $\pdt$ and $\pdtb$ are denoted, respectively, by ${\mathcal G}_{d,g}$ and
%$\ov{\mathcal G}_{d,g}$).

\begin{thm}[Caporaso, Melo]\label{T:st-sp}
\noindent
\begin{enumerate}
\item $\pdtb$ (resp. $\pdb$) is an irreducible and \emph{smooth quotient stack} of finite type over $k$ and of dimension $4g-4$ (resp. $4g-3$).
It contains the stack $\pdt$ (resp. $\pd$) as a dense open substack. 
\item The  morphism $\w{\Phi}_d:\pdtb\to \mgb$ (resp. $\Phi_d:\pdb\to \mgb$) is surjective and universally closed.
%Moreover we have the following commutative diagram:
%Moreover, the following diagram commutes :
%$$\xymatrix{
%\pdt \ar[rr] \ar@{^{(}->}[d] \ar@{}[drr]|{\square} & &\pd \ar@{^{(}->}[d]\\
%\pdtb \ar[dr]_{\w{\Phi}_d}\ar[rr]^{\nu_d} & &\pdb \ar[dl]^{\Phi_d}\\
%& \mgb &\\
%}$$
%\item The following conditions are equivalent:
%\begin{enumerate}[(i)]
%\item $(d+1-g, 2g-2)=1$;
%\item $\Phi_d$ is separated;
%\item $\Phi_d$ is (strongly) representable;
%\item $\pdb$ is separated;
%\item $\pdb$ is a Deligne-Mumford stack.
%\end{enumerate}
\item There exists a projective irreducible normal variety $\ov{J}_{d,g}$, endowed with a surjective morphism $\phi_d:\ov J_{d,g}\to \ov M_g$, which is an adequate  moduli space  in the sense of \cite{alper2} (and even a good moduli space in the sense of \cite{alper} if ${\rm char}(k)=0$) for $\pdtb$ and $\pdb$. 
%and a good moduli space (in the sense of \cite{alper}) for $\pdtb$ and $\pdb$ if ${\rm char}(k)=0$.
% Indeed, $\ov{J}_{d,g}$ is a coarse moduli space for $\pdb$ if and only if $(d+1-g, 2g-2)=1$.
%\item More generally, for all $d\in\mathbb Z$, $(i)$-$(v)$ above hold for the restriction
%$$\Phi_d:\pdb^{\rm d-gen}:=\pdb\times_{\mgb} \mgbd\to \mgbd.$$
\end{enumerate}
\end{thm}
Indeed, if (and only if) $(d+1-g, 2g-2)=1$ then $\pdb$ is a Deligne-Mumford stack,  the morphism $\Phi_d$ is proper and $\ov J_{d,g}$ is a coarse moduli space for $\pdb$.

For later use, we record the morphisms introduced in this subsection into the following commutative diagram:
\begin{equation}\label{big-dia}
\xymatrix{
\pdtb \ar^{\w{\Phi}_d}[d]\ar^{\nu_d}[r] &\pdb \ar[dl]^{\Phi_d} \ar[r]^{\Psi_d}& \ov{J}_{d,g}\ar_{\phi_d}[d] \\
\mgb \ar[rr] & & \ov{M}_g\\
}
\end{equation}

\end{nota}

\begin{nota}{\emph{The Picard and the Chow groups of a stack}}\label{Pic-stack}

In this subsection, we are going to briefly recall
%collect
the definition and the main properties of the Picard group and of the Chow group of an algebraic stack
that we are going to use later. We refer to \cite{Edi} for a nice survey on the subject.

Let $\X$ be an Artin stack of finite type over $k$. The definition of the (functorial) Picard group
of $\X$ was introduced by Mumford (see \cite[p. 64]{Mum}).

\begin{defi}[Mumford]\label{D:Pic-group}
A line bundle $L$ on $\X$ is the data consisting of a line bundle $L(f)\in \Pic(S)$ for every morphism
$f: S \to \X$ from a scheme $S$ and, for every composition of morphisms $T\stackrel{g}{\to} S \stackrel{f}{\to} \X$,  an isomorphism $L(f\circ g)\cong g^* L(f)$, with the obvious compatibility requirements.

The tensor product of two line bundles $L$ and $M$ on $\X$ is the new line bundle $L\otimes M$ on $\X$
defined by $(L\otimes M)(f):=L(f)\otimes M(f)$ together with the
isomorphisms $(L\otimes M)(f\circ g)\cong g^* (L\otimes M)(f)$ induced by those of $L$ and $M$.

The abelian group consisting of all the line bundles on $\X$ together with the operation of tensor product
is called the Picard group of $\X$ and is denoted by $\Pic(\X)$.
\end{defi}

If $\X$ is isomorphic to a quotient stack $[X/G]$, where $X$ is a scheme of finite type over $k$ and $G$ is a group scheme of finite type over $k$,
then $\Pic(\X)$ is isomorphic to the group $\Pic^G(X)$ of $G$-linearized line bundles on $X$ in the sense of \cite[I.3]{GIT} (see e.g. \cite[Prop. 18]{EGa}).

The (operational) Chow groups of an Artin stack $\X$ were introduced by Edidin-Graham in \cite[Sec. 5.3]{EGa}
(see also \cite[Def. 3.5]{Edi}), generalizing the definition of the operational (or bivariant) Chow groups of a scheme
(see \cite[Chap. 17]{Ful}).

\begin{defi}[Edidin-Graham]\label{D:Chow-group}
An $i$-th Chow cohomology class $c$ on $\X$ is the data consisting of an element $c(f)$ belonging to the $i$-th operational Chow group $A^i(S)$ for every morphism
$f: S \to \X$ from a scheme $S$ and, for every composition of morphisms $T\stackrel{g}{\to} S \stackrel{f}{\to} \X$,  an isomorphism $c(f\circ g)\cong g^* c(f)$, with the obvious compatibility requirements.

The sum of two $i$-th Chow cohomology classes $c$ and $d$ on $\X$ is the new $i$-th Chow cohomology class
$c\oplus d$ on $\X$
defined by $(c\oplus d)(f):=c(f)\oplus d(f)$ together with the
isomorphisms $(c\oplus d)(f\circ g)\cong g^* (c\oplus d)(f)$ induced by those of $c$ and $d$.

The abelian group consisting of all the $i$-th Chow cohomology classes on $\X$ together with the operation of
sum is called the $i$-th Chow group of $\X$ and is denoted by $A^i(\X)$.
\end{defi}

If $\X$ is isomorphic to a quotient stack $[X/G]$, where $X$ is a scheme of finite type over $k$ and $G$ is a group scheme of finite type over $k$,
then $A^i(\X)$ is isomorphic to the $i$-th (operational) equivariant Chow group $A^i_G(X)$ defined by Edidin-Graham in
\cite[Sec. 2.6]{EGa} (see \cite[Prop. 19]{EGa}).

The first Chern class gives an homomorphism
\begin{equation}\label{E:1Chern}
\begin{aligned}
c_1 :\Pic(\X) & \longrightarrow A^1(\X) \\
L & \mapsto c_1(L)
\end{aligned}
\end{equation}
where $c_1(L)\in A^1(\X)$ is defined by setting $c_1(L)(f):=c_1(L(f))$ for every morphism $f:S\to \X$ from a
scheme $S$.

In the sequel, we will use the following results concerning the Picard group of a smooth quotient stack.

\begin{fact}[Edidin-Graham]\label{Fact-Pic}
Let $\X$ be a \emph{smooth quotient stack}, i.e. $\X=[X/G]$ where $X$ is a smooth variety and $G$ is an algebraic group acting on $X$.
\begin{enumerate}[(i)]
\item \label{Fact-Pic1} The first Chern class map $c_1:\Pic(\X)\to A^1(\X)$ is an isomorphism.\\
In particular, every Weil divisor $\D$ on $\X$ is a Cartier divisor and hence it gives rise
to a line bundle $\O_{\X}(\D)$ on $\X$.
%We will often abuse the notation
%and denote also with $\D$ the class of the divisor $\D$ in the Picard group $\Pic(\X)$.
\item \label{Fact-Pic2} Given a Weil divisor $\D$ of $\X$ with irreducible components $\D_i$, there is an exact sequence
$$\bigoplus_i \Z\cdot \langle  \O_{\X}(\D_i) \rangle \to \Pic(\X)\to \Pic(\X\setminus \D)\to 0. $$
\item \label{Fact-Pic3} If $\Y$ is a closed substack of $\X$ of codimension greater than $1$ then there is an isomorphism
$$\Pic(\X)\stackrel{\cong}{\to} \Pic(\X\setminus \Y).$$
\end{enumerate}
\end{fact}
\begin{proof}
Part \eqref{Fact-Pic1} follows from \cite[Cor. 1]{EGa}. Part \eqref{Fact-Pic2} follows from \cite[Prop. 5]{EGa}.
Part \eqref{Fact-Pic3} follows from \cite[Lemma 2(a)]{EGa}.
\end{proof}

By Theorems \ref{T:st-sp}, all the properties stated in Fact \ref{Fact-Pic}
hold for the stacks we will deal with, namely $\pdt$, $\pd$, $\pdtb$ and $\pdb$. Moreover, it is well-known
that the same properties hold true for $\mgb$ and $\mg$.

%For a proof of the fact that $\Pic_{\rm fun}(\pdb)=\Pic^{PGL(r+1)}(H_d)$, see \cite{Mum2}.

\vspace{0.1cm}

There are two standard methods to produce line bundles on a stack parametrizing nodal curves with some extra-structure
(as $\pdtb$), namely  the determinant of cohomology (introduced in \cite{KM})
and the Deligne pairing (introduced in \cite{Del}). Let us review briefly the definition and main properties of these two constructions, following the presentation given in
\cite[Chap. 13, Sec. 4 and 5]{ACG}.

Let $\pi:X\to S$ be a family of nodal curves, i.e. a proper and flat morphism whose geometric fibers are nodal curves.
Given a coherent sheaf $\F$ on $X$ flat over $S$ (e.g. a line bundle on $X$), the {\em determinant of cohomology}
of $\F$ is a line bundle $d_{\pi}(\F)\in \Pic(S)$ defined as it follows: we choose a complex of locally free sheaves $f: K^0\to K^1$ on $S$ such that
$\ker f=\pi_*(\F)$ and $\coker f=R^1\pi_*(\F)$ (this is always possible) and we set 
$$d_{\pi}(\F):=\det K^0 \otimes (\det K^1)^{-1}.$$
% In the special case where $\pi_*(\F)$ and $R^1\pi_*(\F)$ are locally free sheaves on $S$, one sets
%$$d_{\pi}(\F):=\det \pi_*(\F)\otimes (\det R^1\pi_*(\F))^{-1}.$$
%In the general case, one can always find a complex of locally free sheaves $f: K^0\to K^1$ on $S$ such that $\ker f=\pi_*(\F)$ and $\coker f=R^1\pi_*(\F)$ and then one sets
%$$d_{\pi}(\F):=\det K^0 \otimes (\det K^1)^{-1}.$$
The determinant of cohomology is functorial, multiplicative for short exact sequence and its first Chern class is equal to
\begin{equation}\label{E:Chern-det}
c_1(d_{\pi}(\F))=c_1(\pi_{!}(\F)):=c_1(\pi_*(\F))-c_1(R^1\pi_*(\F)).
\end{equation}
For more details, the reader is referred to  \cite[Chap. 13, Sec. 4]{ACG}.

%\begin{fact}\label{F:prop-det}
%Let $\pi:X\to S$ be a family of nodal curves and let $\F$ be a coherent sheaf on $X$ flat over $S$.
%\begin{enumerate}[(i)]
%\item For every exact sequence of coherent sheaves on $X$ flat over $S$
%$$0 \to \E \to \F \to \G \to 0, $$
%there is a canonical isomorphism
%$$d_{\pi}(F)\cong d_{\pi}(\E)\otimes d_{\pi}(\F).$$
%\item If $\F$ is locally free then there is a canonical isomorphism
%$$d_{\pi}(\omega_{\pi}\otimes \F^{\vee})\cong d_{\pi}(\F),$$
%where $\omega_{\pi}$ is the relative dualizing sheaf of the family $\pi$.
%\item The first Chern class of $d_{\pi}(\F)$ is equal to
%$$c_1(d_{\pi}(\F))=c_1(\pi_{!}(\F)):=c_1(\pi_*(\F))-c_1(R^1\pi_*(\F)).$$
%\item The formation of the determinant of cohomology is functorial in the following sense: given a Cartesian diagram
%$$\xymatrix{
%X' \ar[r]^g \ar[d]_{\pi'} \ar@{}[dr]|{\square} & X \ar[d]^{\pi} \\
%S' \ar[r]^f & S
%}
%$$
%we have a canonical isomorphism
%$$f^* d_{\pi}(\F)\cong d_{\pi'}(g^* \F).$$
%\end{enumerate}
%\end{fact}

Given two line bundles $\M$ and $\L$ on the total space of a family of nodal curves $\pi: X\to S$,
the {\em Deligne pairing} of $\M$ and $\L$ is a line bundle $\langle \M, \L\rangle_{\pi}\in \Pic(S)$
which can be defined as
\begin{equation}\label{E:Deligne}
\langle \M, \L\rangle_{\pi}:=d_{\pi}(\M\otimes \L)\otimes d_{\pi}(\M)^{-1}\otimes d_{\pi}(\L)^{-1}\otimes d_{\pi}(\O_X).
\end{equation}
The Deligne pairing is functorial, symmetric and bilinear in each factor, and its first Chern class satisfies 
\begin{equation}\label{E:Chern-Del}
c_1(\langle \M, \L\rangle_{\pi})=\pi_*(c_1(\M)\cdot c_1(\L)).
\end{equation}
For more details, the reader is referred to  \cite[Chap. 13, Sec. 5]{ACG}.

%\begin{fact}\label{F:prop-Deligne}
%Let $\pi:X\to S$ be a family of nodal curves.
%\begin{enumerate}[(i)]
%\item The Deligne pairing is symmetric and bilinear in each factor, namely there are canonical isomorphisms
%$$\begin{aligned}
%& \langle \M, \L\rangle_{\pi}\cong \langle \L, \M\rangle_{\pi}, \\
%& \langle \M\otimes \M', \L\rangle_{\pi}\cong \langle \M, \L\rangle_{\pi}\otimes \langle \M', \L\rangle_{\pi},\\
%& \langle \M, \L\otimes \L'\rangle_{\pi}\cong \langle \M, \L\rangle_{\pi}\otimes \langle \M, \L'\rangle_{\pi},\\
%& \langle \M, \O_X\rangle_{\pi}\cong \langle \O_X, \M\rangle_{\pi}\cong \O_S.
%\end{aligned}$$
%\item The first Chern class of $\langle \M, \L\rangle_{\pi}$ is equal to
%$$c_1(\langle \M, \L\rangle_{\pi})=\pi_*(c_1(\M)\cdot c_1(\L)).$$
%\item The formation of the Deligne pairing is functorial in the following sense: given a Cartesian diagram
%$$\xymatrix{
%X' \ar[r]^g \ar[d]_{\pi'} \ar@{}[dr]|{\square} & X \ar[d]^{\pi} \\
%S' \ar[r]^f & S
%}
%$$
%we have a canonical isomorphism
%$$f^* \langle \M, \L\rangle_{\pi}\cong \langle g^*(\M), g^*(\L)\rangle_{\pi'}.$$
%\end{enumerate}
%\end{fact}

\begin{remark}\label{R:func-stack}
Since the determinant of cohomology and the Deligne pairing are functorial, we can extend their definition to the case when $\pi:\Y\to \X$ is a representable, proper and flat morphism of Artin stacks whose geometric fibers are nodal curves.
\end{remark}

\end{nota}

\begin{nota}{\emph{The Picard group of $\mgb$}}\label{sec-pic-mg}

In this subsection,  in order to fix the notation, we recall  the description of the Picard group $\Pic(\mgb)$.

The universal family $\ov{\pi}:\ov{\M}_{g,1}\to \mgb$ is a representable, proper and flat morphism whose geometric fibers
are nodal curves. Applying the determinant of cohomology to the relative dualizing sheaf $\omega_{\ov{\pi}}$ (see \ref{Pic-stack}), we define the  {\em Hodge line bundle}
\begin{equation}\label{E:Hodge}
\Lambda:= d_{\ov{\pi}}(\omega_{\ov{\pi}})\in \Pic(\mgb).
\end{equation}
%\begin{equation}\label{E:Hodge}
%\begin{aligned}
%& \Lambda(n):= d_{\ov{\pi}}(\omega_{\ov{\pi}}^{\otimes n}) \text{ for any } n\in \Z, \\
%& K_1:=\langle \omega_{\ov{\pi}}, \omega_{\ov{\pi}} \rangle_{\ov{\pi}}.\\
%\end{aligned}
%\end{equation}
%The line bundle $\Lambda(1)$ is called the {\em Hodge line bundle} and it is denoted by $\Lambda$. 
Using the functoriality of the determinant of cohomology, it is easily checked that $\Lambda$
associates to a family of stable curves $\{f:\C\to S\}\in \mgb(S)$ the line bundle
$$\Lambda(f)=\det f_*(\omega_{\C/S})\otimes \det(R^1f_*(\omega_{\C/S}))^{-1}=\bigwedge^g f_*(\omega_{\C/S})
\in \Pic(S).$$
We will abuse the notation and denote also with $\Lambda$
the restriction of $\Lambda$ to $\mg$ is also denoted by $\Lambda$. 
%We will apply the same abuse of notation to the other tautological line bundles in \eqref{E:taut-bundles-mgb}.

Recall that the boundary $\mgb\setminus \mg$ decomposes as the union of irreducible divisors $\delta_i$
for $i=0,\ldots,[g/2]$ which are defined as follows:
$\delta_0$ is the boundary divisor of $\mgb$ whose generic point is  an irreducible nodal curve
of genus $g$ with one node while, for any $1\leq i\leq [g/2]$, $\delta_i$ is the boundary divisor of $\mgb$
whose generic point is a stable curve formed by two irreducible components of genera $i$ and $g-i$ meeting in
one point. We will denote by $\Delta_i\subset \ov{M}_g$ the image of $\delta_i\subset \mgb$ via the natural map
$\mgb\to \ov{M}_g$. We set $\delta:=\sum_i \delta_i$ and denote by $\O(\delta)$ the associated line bundle on
$\mgb$ (see Fact \ref{Fact-Pic}\eqref{Fact-Pic1}). Similarly for $\O(\delta_i)\in \Pic(\mgb)$.

%Mumford  showed in \cite{Mum3} that all the tautological line bundles \eqref{E:taut-bundles-mgb} on $\mgb$ can be expressed in terms of the Hodge line bundle $\Lambda$ and the line bundle $\O(\delta)$. More precisely, he proved the following result (see also \cite[Chap. 13, Sec. 7]{ACG} for a proof).

%\begin{thm}[Mumford]\label{T:Mumford-rela}
%The  tautological line bundles on $\mgb$ satisfy the following relations
%$$\begin{aligned}
%& K_1=\Lambda^{12}\otimes \O(-\delta), \\
%& \Lambda(n)=\Lambda^{6n^2-6 n+1}\otimes \O\left(-\binom{n}{2}\delta\right).
%\end{aligned}
%$$
%\end{thm}

The Picard groups of $\mgb$ and of $\mg$ are described by the following theorem proved by Arbarello-Cornalba in
\cite[Thm. 1]{arbcorn}, based upon a  result of Harer \cite{Har}.

\begin{thm}[Harer, Arbarello-Cornalba]\label{pic-mg}
Assume that $g\geq 3$. Then
\noindent
\begin{enumerate}[(i)]
\item $\Pic(\mg)$  is freely generated by $\Lambda$.
\item  $\Pic(\mgb)$  is freely generated by $\Lambda, \O(\delta_0), \cdots, \O(\delta_{[g/2]})$.
\end{enumerate}
\end{thm}

If $g=2$, then $\Pic(\mg)$ (resp. $\Pic(\mgb)$) is still generated by $\Lambda$ (resp. by $\Lambda, \O(\delta_0),
\O(\delta_1)$) but with the extra relation $\Lambda^{10}=0$ (resp. $\Lambda^{10}\otimes \O(-\delta_0-2\delta_1)=0$),
see respectively \cite{Vis} and \cite{Cor2}.

\end{nota}

\section{Boundary divisors of $\pdtb$}\label{S:bound-div}

The aim of this Section is to describe the irreducible components of the boundary divisor $\pdtb$
and their relationship with the boundary divisors of $\mgb$.

Consider the following divisors in the boundary of $\pdtb$:

\begin{enumerate}[(A)]
\item $\w{\delta}_0$ is the divisor whose generic point is a pair $(C,L)$ where $C$ is an irreducible
curve of genus $g$ with one node and $L$ is a degree $d$ line bundle on it.

\item For $1\leq i\leq g/2$ and $\kdg\nmid ( 2i-1)$, $\w{\delta}_i$ is the divisor whose generic point is a pair $(C,L)$, where
$C$ is formed by two smooth irreducible curves $C_1$ and $C_2$ of genera respectively  $i$ and $g-i$ meeting in one point, and $L$
is a line bundle of multidegree
$$(\deg_{C_1}L , \deg_{C_2}L)=
%([M_{C_1}(d)], [M_{C_2}(d)])=
\left(\left[d\frac{2i-1}{2g-2}+\frac{1}{2}\right],
\left[d\frac{2(g-i)-1}{2g-2}+\frac{1}{2}\right]\right).$$

\item For $1\leq i<g/2$ and $\kdg \mid  (2i-1)$, $\w{\delta}_i^1$ (resp. $\w{\delta}_i^2$) is  the divisor whose generic point is a pair $(C,L_1)$
(resp. $(C,L_2)$), where $C$ consists of
two smooth irreducible curves $C_1$ and $C_2$ of genera respectively  $i$ and $g-i$ meeting in one point, and $L_1$
and $L_2$ are line bundles of multidegree
$$(\deg_{C_1}L_1 , \deg_{C_2}L_1)=
%([M_{C_1}(d)], [M_{C_2}(d)])=
\left(d\frac{2i-1}{2g-2}-\frac{1}{2}, d\frac{2(g-i)-1}{2g-2}+\frac{1}{2}\right).$$
$$(\deg_{C_1}L_2 , \deg_{C_2}L_2)=
%([M_{C_1}(d)], [M_{C_2}(d)])=
\left(d\frac{2i-1}{2g-2}+\frac{1}{2}, d\frac{2(g-i)-1}{2g-2}-\frac{1}{2}\right).$$

\item If $g$ is even and $\kdg \mid(g-1)$ (i.e. $d$ is odd), $\w{\delta}_{g/2}$
%(:=\w{\delta}_{g/2}^1=\w{\delta}_{g/2}^2$)
is  the divisor whose generic point is a pair $(C,L)$, where
$C$ is formed by two smooth irreducible curves $C_1$ and $C_2$ both of genera $g/2$ meeting in one point, and $L$
is a line bundle of multidegree
$$(\deg_{C_1}L , \deg_{C_2}L)=\left(\frac{d-1}{2},\frac{d+1}{2} \right).$$

\end{enumerate}

Note that in the above cases (C) and (D), the divisibility condition $\kdg\mid(2i-1)$ is equivalent to the condition that
%$C=C_1\cup C_2$ is a $d$-special vine curve (see Proposition \ref{d-special}), or equivalently to the fact that
$M_{C_i}(d)$ and $m_{C_i}(d)$ are integers
(see Definition \ref{balanced}). Moreover, the case (D) is different from the case (C) since in the case (D)
the two components $C_1$ and $C_2$ have the same genus and
hence it is not possible to distinguish ``numerically'' a line bundle of multidegree $(\deg_{C_1}L , \deg_{C_2}L)=\left(\frac{d-1}{2},\frac{d+1}{2} \right)$
from one of multidegree $(\deg_{C_1}L , \deg_{C_2}L)=\left(\frac{d+1}{2},\frac{d-1}{2} \right)$.

\begin{notation}\label{nota-div}
{\bf Notation}: Sometimes it is convenient to unify the notation for the cases (A) and (B) and for the cases (C) and (D). For this reason, we always assume
that $\kdg\nmid (2\cdot 0-1)=-1$ (even when $\kdg=1$) and we set $\w{\delta}_{g/2}^1=\w{\delta}_{g/2}^2=\w{\delta}_{g/2}$ if
$g$ is even and $\kdg \mid(g-1)$ (i.e. if $g$ is even and $d$ is odd).
\end{notation}

As usual, we denote by $\O(\w{\delta}_i)$ the line bundle on $\pdtb$ associated to $\delta_i$ and similarly for
$\O(\w{\delta}_i^1)$ and $\O(\w{\delta}_i^2)$.
Using the above Notation \ref{nota-div}, we also set
\begin{equation}\label{E:tot-div}
\w{\delta}:=\sum_{k_{d,g}\nmid (2i-1)} \w{\delta}_i  +\sum_{k_{d,g}\mid (2i-1)} (\w{\delta}_i^1+\w{\delta}_i^2),
\end{equation}
and we denote by $\O(\w{\delta})\in \Pic(\pdtb)$ its associated line bundle. Note that, according to Notation \ref{nota-div}, if $g$ is even and $d$ is odd
then $\w{\delta}_{g/2}=\w{\delta}_{g/2}^1=\w{\delta}_{g/2}^2$ appears with coefficient two in $\w{\delta}$.

Via the natural forgetful map  $\w{\Phi}_d: \pdtb\to \mgb$, we can relate the boundary divisors of $\pdtb$ with those of $\mgb$ as follows.

\begin{thm}\label{bound-pdb}
\noindent
\begin{enumerate}[(i)]
\item \label{bound1} The boundary $\pdtb\setminus \pdt$ of $\pdtb$ consists of the irreducible divisors $\{\w{\delta}_i\: : \: \kdg \,  \nmid (2i-1) \text{ or } i=g/2\}$
and $\{\w{\delta}_i^1, \w{\delta}_i^2 \: : \: \kdg \mid (2i-1)\text{ and } i<g/2\}$.
\item \label{bound2} For any $0\leq i\leq g/2$, we have
$$\w{\Phi}_d^*\,\O(\delta_i)=\begin{sis}
&\O(\w{\delta}_i) & \text{ if } \kdg\nmid(2i-1), \\
&\O(\w{\delta}_i^1+\w{\delta}_i^2) & \text{ if } \kdg\mid(2i-1).
\end{sis}
$$
In particular, $\w{\Phi}_d^*\,\O(\delta)=\O(\w{\delta})$.
 \end{enumerate}
\end{thm}

%For any $0\leq i\leq [g/2]$, denote by $\w{\delta}_i$ the inverse image of the boundary divisor $\delta_i$ of $\ov{\mathcal M}_g$
%via the morphism $\Phi_d:\pdb\to \mgb$.
%Since $\Phi_d$ is surjective and the $\delta_i$'s are divisors of $\mgb$, the $\w{\delta}_i$'s will be divisors of $\ov{\mathcal P}_{d,g}$ as well.

%\begin{prop}\label{irrdivisors}
%$\D_0$ is an irreducible divisor of $\pdb$. For $i=0,\dots,{\lfloor\frac{g}{2}\rfloor}$, $\D_i$ is irreducible if and only if the general curve of $\Delta_i$ is $d$-general, and this
%happens precisely when
%$$\kdg\nmid \: 2i-1.$$
%Otherwise, $\D_i=\D_i^1\cup \D_i^2$, where $\D_i^1$ and $\D_i^2$ are irreducible.
 %\dots, D_{\lfloor\frac{g}{2}\rfloor}$ are irreducible divisors of $\ov{\mathcal P}_{d,g}$.
%\end{prop}

\begin{proof}
By construction we have that $\pdtb\setminus \pdt=\w{\Phi}_d^{-1}(\mgb\setminus \mg)$
(see \ref{desc-stacks}) and moreover $\mgb\setminus \mg=\bigcup_i \delta_i$ (see \ref{sec-pic-mg}).
By the Definition \ref{balanced}, it is easy to check that we have a set-theoretical equality
\begin{equation}\label{pull-back-div}
\w{\Phi}_d^{-1}(\delta_i)=\begin{sis}
&\w{\delta}_i & \text{ if } \kdg\nmid(2i-1), \\
&\w{\delta}_i^1\cup \w{\delta}_i^2 & \text{ if } \kdg\mid(2i-1).
\end{sis}
\end{equation}
Finally, by looking at their definition, it is easy to see that the divisors $\w{\delta}_i, \w{\delta}_i^1, \w{\delta}_i^2$ are irreducible. This completes the proof of part (i).

Part (ii) is equivalent to proving that we have a scheme-theoretic equality in (\ref{pull-back-div}).
To achieve that, we need a local description of the morphism $\w{\Phi}_d:\pdtb\to \mgb$ at a general point
$(C,L)$ of $\w{\delta}_i$ or of $\w{\delta}_i^1\cap \w{\delta}_i^2$.
Recall that locally at $(C,L)$, the morphism $\w{\Phi}_d$ looks like
$$q:[\Def_{(C,L)}/\Aut(C,L)]\to [\Def_{C^{\rm st}}/\Aut(C^{\rm st})],$$
where $\Def_{C^{\rm st}}$(resp. $\Def_{(C,L)}$) is the miniversal deformation space of the stabilization $C^{\rm st}$
of $C$ (resp. of the pair $(C,L)$)
and $\Aut(C^{\rm st})$ (resp. $\Aut(C,L)$) is the automorphism group of $C^{\rm st}$ (resp. the automorphism group  of the pair $(C,L)$).
Using the results on the local structure of $\pdtb$ given in \cite[Sec. 2.15]{BFV},
we can describe explicitly the above morphism $q$  at a general point of $\w{\delta}_i$ or of $\w{\delta}_i^1\cap \w{\delta}_i^2$
in the boundary of $\pdtb$.
To this aim,  we need to distinguish between the case
$\kdg\nmid (2i-1)$ (cases (A) and (B)) and the case  $\kdg \mid (2i-1)$ (cases (C) and (D)).

Suppose first that $\kdg\nmid(2i-1)$. Consider a general point $(C,L)$ of $\w{\delta}_i$.
Since $C=C^{\rm st}$ is a general element of $\delta_i$, it is well-known that
$\Def_C=\Spf k[[x_1,\cdots, x_{3g-3}]]$ and
\begin{equation}\label{aut-gen}
\Aut(C)=\begin{cases}
\{1\} & \text{ if } i\neq 1,\\
\Z/2\Z & \text{ if } i=1,
\end{cases}
\end{equation}
where, in the second case, the unique non-trivial automorphism is the elliptic involution
on the elliptic tail of $C$. On the other hand, we have that
$\Def_{(C,L)}=\Spf k[[x_1,\cdots,x_{3g-3},t_1,\cdots, t_g]]$ and $\Aut(C,L)=\Gm$ acts trivially on it
(see \cite[Proof of Thm. 1.5, Cases (1) and (2)]{BFV}), where the coordinates $x_i$'s correspond
to the deformation of the curve $C$ and the coordinates $t_j$'s correspond to the deformation of the line bundle
$L$. The morphism $q$ is given by the natural equivariant projection $\Def_{(C,L)}\twoheadrightarrow \Def_{C}$.
Moreover, we can choose local coordinates $x_1,\cdots,x_{3g-3}$ for $\Def_C$ in such a way that
the first coordinate $x_1$ corresponds to the smoothing of the unique node of $C$ and, if $i=1$, the action of
the generator of $\Aut(C)=\Z/2\Z$ sends $x_1$ into $-x_1$ and fixes the other coordinates.
For such a choice of the coordinates, we have that the equation of $\delta_i$ inside
$\Def_C$ is given by $(x_1=0)$ and the equation of $\w{\delta}_i$ inside $\Def_{(C,L)}$ is given
by $(x_1=0)$. Since $q^*(x_1)=(x_1)$, we conclude in this case.

Suppose now that $\kdg\mid(2i-1)$ (hence that $i>0$ by Notation \ref{nota-div}). If $i<g/2$ then a general
point $(C,L)$ of $\w{\delta}_i^1\cap \w{\delta}_i^2$ consists of the two general curves $C_1$ and $C_2$ of genera
respectively $i$ and $g-i$ joined by a rational curve $R\cong \P^1$. By convention, in the case $i=g/2$
and $\kdg\mid(g-1)$, we set $\w{\delta}_{g/2}^1\cap \w{\delta}_{g/2}^2$ to be the closure of the locus
of curves consisting of two smooth curves of genera $g/2$ joined by a rational curve $R\cong \P^1$.
The stabilization $C^{\rm st}$ is obtaining by contracting the rational curve $R$ to a node $n$
and it will be a general point of $\delta_i$. As before, we have that $\Def_{C^{\rm st}}=\Spf k[[x_1,\cdots,
x_{3g-3}]]$, where $x_1$ can be chosen as the coordinate corresponding to the smoothing of the node $n$,
and $\Aut(C^{\rm st})$ is as in (\ref{aut-gen}). On the other hand, by \cite[Proof of Theorem 1.5, Case (3)]{BFV},
we have that $\Aut{(C,L)}=\Gm^2$, $\Def_{(C,L)}=\Spf k[[u_1,v_1,x_2,\cdots, x_{3g-3},t_1,\cdots,t_g]]$
where $u_1$ corresponds to the node $C_1\cap R$ and $v_1$ corresponds to the node $C_2\cap R$. Moreover,
the action of $\Gm^2$ on $\Def_{(C,L)}$ is given by $(\lambda, \mu)\cdot (u_1,v_1)=(\lambda \mu^{-1} u_1, \lambda^{-1} \mu v_1)$
while it is the identity on the other coordinates.
The morphism $q$ is induced by the equivariant morphism $\Def_{(C,L)}\to \Def_{C^{\rm st}}$ that, at the level of
rings, sends $x_1$ into $u_1\cdot v_1$ and $x_i$ into $x_i$ for $i>1$.
The equation of $\delta_i$ inside $\Def_{C^{\rm st}}$ is given by $(x_1=0)$ while the equations
of $\w{\delta}_i^1$ and $\w{\delta}_i^2$ inside $\Def_{(C,L)}$ are given by $(u_1=0)$ and $(v_1=0)$
(note that in the special case $i=g/2$ and $\kdg \mid(g-1)$, the divisor $\w{\delta}_{g/2}$, even though irreducible,
has two branches locally at $(C,L)$, which we call $\w{\delta}_{g/2}^1$ and
$\w{\delta}_{g/2}^2$, whose equations are $(u_1=0)$ and $(v_1=0)$). Since $q^*(x_1)=(u_1\cdot v_1)$,
we conclude also in this case.

\end{proof}

As a Corollary of the above Theorem \ref{bound-pdb}, we can determine also the irreducible components of
the boundary of $\pdb$. We set $\ov{\delta}_i:=\nu_d(\w{\delta}_i)$, $\ov{\delta}_i^1=\nu_d(\w{\delta}_i^1)$ and $\ov{\delta}_i^2:=\nu_d(\w{\delta}_i^2)$ according to the above Cases (A)--(B), where as usual $\nu_d:\pdtb\to \pdb$ is the
rigidification map.

\begin{cor}\label{C:bound-rig}
\noindent
\begin{enumerate}[(i)]
\item \label{C:bound-rig1} The boundary $\pdb\setminus \pd$ of $\pdb$ consists of the irreducible divisors
$\{\ov{\delta}_i\: : \: \kdg \,  \nmid (2i-1) \text{ or } i=g/2\}$
and $\{\ov{\delta}_i^1, \ov{\delta}_i^2 \: : \: \kdg \mid (2i-1)\text{ and } i<g/2\}$.
\item \label{C:bound-rig2} For any $0\leq i\leq g/2$, we have
$$\begin{sis}
& \nu_d^*\,\O(\ov{\delta}_i)= \O(\w{\delta}_i) & \text{ if } \kdg\nmid(2i-1), \\
&\nu_d^*\,\O(\ov{\delta}_i^j) = \O(\w{\delta}_i^j)  & \text{ if } \kdg\mid(2i-1) \text{ and } j=1, 2.
\end{sis}
$$
 \end{enumerate}
\end{cor}
\begin{proof}
The Corollary follows straightforwardly from Theorem \ref{bound-pdb} and the fact that $\nu_d:\pdtb\to \pdb$ is a $\Gm$-gerbe
such that $\nu_d^{-1}(\pd)=\pdt$.
\end{proof}

\section{Independence of the boundary divisors}\label{S:indip-bound}

The aim of this Section is to prove that the line bundles corresponding to the irreducible components of the boundary of $\pdtb$ are linearly independent
in $\Pic(\pdtb)$.
More precisely, we will prove the following result.

\begin{thm}\label{T:ex-seq}
We have an exact sequence
\begin{equation}\label{ex-Pic}
0\to \bigoplus_{\stackrel{\kdg \: \nmid\:\: 2i-1}{\text{Êor } i=g/2}}\langle \O(\w{\delta}_i)\rangle
\bigoplus_{\stackrel{\kdg \mid 2i-1}{\text{Êand } i\neq g/2}}\langle \O(\w{\delta}_i^1)\rangle \oplus \langle \O(\w{\delta}_i^2)\rangle\to
\Pic(\pdtb)\to \Pic(\pdt)\to 0,
\end{equation}
where the right map is the natural restriction morphism and the left map is the natural inclusion.
\end{thm}
Using Theorem \ref{bound-pdb}(\ref{bound1}) and Fact \ref{Fact-Pic}\eqref{Fact-Pic2},
we have that the exact sequence \eqref{ex-Pic} is exact except perhaps to the left.
It remains to prove that the map on the left is injective, or in other words that the line bundles associated to the boundary divisors of $\pdtb$
 are linearly independent in $\Pic(\pdtb)$.

 The strategy that we will use to prove this is the same
as the one used by Arbarello-Cornalba in \cite{arbcorn}: we shall construct maps
$B\to \pdtb$ from irreducible smooth projective curves $B$
(i.e. families of quasistable curves of genus $g$ parametrized
by $B$, endowed with a properly balanced line bundle of relative degree $d$) and compute the
degree of the pullbacks of the boundary divisors of $\Pic(\pdtb)$ to $B$.
Actually, we will construct liftings of the families $F_h$ (for $1\leq h\leq (g-2)/2$), $F$
and $F'$ used by Arbarello-Cornalba in \cite[p. 156-159]{arbcorn}.
For that reason, we will be using their notations.

 Note that, for every $n\in \Z$,  there are isomorphisms
\begin{equation}\label{shift-iso}
\begin{aligned}
\w{\phi_{d}^n}: \pdtb & \stackrel{\cong}{\longrightarrow} \ov{\mathcal Jac}_{d+n(2g-2),g}\\
(\C\to S,\L) & \mapsto (\C\to S, \L\otimes \omega_{\C/S}^{\otimes n}). \\
\end{aligned}
\end{equation}
Clearly, $\w{\phi_{d}^n}$ is an isomorphism of $\Gm$-stacks and therefore,
by passing to the $\Gm$-rigidification, it induces an isomorphism $\phi_{d}^n: \pdb\stackrel{\cong}{\rightarrow}  \ov{\mathcal J}_{d+n(2g-2),g}$.

Since $\pdtb\cong \ov{{\mathcal Jac}}_{d',g}$ if $d\equiv d' \mod (2g-2)$ (see \ref{shift-iso}),
throughout this section we can make the following
\begin{ass}
The degree $d$ satisfies $0\leq d<2g-2$.
\end{ass}

{\bf The Family $\w{F}$}

Start from a general pencil of conics in $\P^2$. Blowing up the four base points of the pencil,
we get a conic bundle $\phi:X\to \P^1$. The four exceptional divisors $E_1,E_2,E_3,E_4\subset X$ of
the blow-up of $\P^2$ are sections of $\phi$ through the smooth locus of $\phi$.
Note that $\phi$ will have three
singular fibers consisting of two incident lines.
Let $C$ be a fixed irreducible, smooth and projective curve of genus $g-3$ and $p_1,p_2, p_3, p_4$
four points of $C$. We construct a surface $Y$ by setting
$$Y=\left(X\coprod (C\times \P^1) \right)/(E_i\sim \{p_i\}\times \P^1\: :\: i=1,\cdots, 4).
$$
We get a family $f:Y \to \P^1$ of stable curves of genus $g$: the general fiber of $f$ consists of $C$ and a smooth
conic $Q$ meeting in  $4$  points (see Figure \ref{Figure1} below),
%(see \cite[Fig. 4]{arbcorn}),
while the three special fibers consist of $C$ and two lines $R_1$ and $R_2$
such that $|R_1\cap R_2|=1$, $|R_1\cap C|=|R_2\cap C|=2$ (see Figure \ref{Figure2} below).
% (see \cite[Fig. 5]{arbcorn}).

\begin{figure}[ht]
%\label{Figure1}
\begin{center}
%\vspace{.5cm}
%TeXCAD Picture [Figura4_2.pic]. Options:
%\grade{\on}
%\emlines{\off}
%\epic{\off}
%\beziermacro{\on}
%\reduce{\on}
%\snapping{\off}
%\pvinsert{% Your \input, \def, etc. here}
%\quality{8.000}
%\graddiff{0.005}
%\snapasp{1}
%\zoom{5.6569}
\unitlength .65mm % = 2.845pt
\linethickness{0.4pt}
\ifx\plotpoint\undefined\newsavebox{\plotpoint}\fi % GNUPLOT compatibility
\begin{picture}(85.55,50.324)(0,110)
\qbezier(12.78,153.574)(19.905,97.574)(34.53,116.574)
\qbezier(34.53,116.574)(46.155,136.574)(56.28,116.574)
\qbezier(56.28,116.574)(71.53,96.699)(78.78,154.324)
\put(10.828,145.765){\makebox(0,0)[cc]{$C$}}
\put(4.625,120.248){\line(1,0){.21}}
\put(4.835,120.248){\line(1,0){88.715}}
\put(88.504,122.561){\makebox(0,0)[cc]{$Q$}}
\put(36.532,111.723){\makebox(0,0)[cc]{$g-3$}}
%\put(30.532,100.723){\makebox(0,0)[cc]{Figure $1$}}
\end{picture}
\end{center}
\caption{The general fiber of $f:Y\to \P^1$}\label{Figure1}
\end{figure}
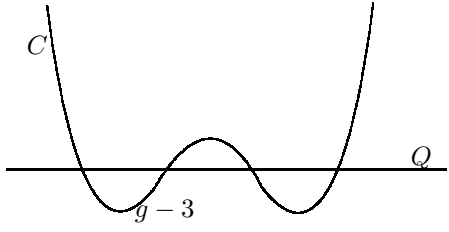
\vspace*{1.cm}

\begin{figure}[ht]
\begin{center}
%\vspace{-.5cm}
%TeXCAD Picture [Figura5.pic]. Options:
%\grade{\on}
%\emlines{\off}
%\epic{\off}
%\beziermacro{\on}
%\reduce{\on}
%\snapping{\off}
%\pvinsert{% Your \input, \def, etc. here}
%\quality{8.000}
%\graddiff{0.005}
%\snapasp{1}
%\zoom{4.7568}
\unitlength .65mm % = 2.845pt
\linethickness{0.4pt}
\ifx\plotpoint\undefined\newsavebox{\plotpoint}\fi % GNUPLOT compatibility
\begin{picture}(80.75,50.75)(0,150)
\put(10.25,182.25){\line(-1,0){.25}}
\put(10,182.25){\line(1,0){.25}}
%\emline(10.25,182.25)(52.5,150.5)
\multiput(10.25,182.25)(.04485138004,-.03370488323){942}{\line(1,0){.04485138004}}
%\end
%\emline(35,150.75)(82.75,184)
\multiput(35,150.75)(.04842799189,.03372210953){986}{\line(1,0){.04842799189}}
%\end
\qbezier(12.25,196)(19.375,140)(34,159)
\qbezier(34,159)(45.625,179)(55.75,159)
\qbezier(55.75,159)(71,139.125)(78.25,196.75)
\put(24.5,176.25){\makebox(0,0)[cc]{$R_1$}}
\put(62.75,175){\makebox(0,0)[cc]{$R_2$}}
\put(8,196.5){\makebox(0,0)[cc]{$C$}}
\end{picture}
\end{center}
\caption{The three special fibers of $f:Y\to \P^1$}\label{Figure2}
\end{figure}
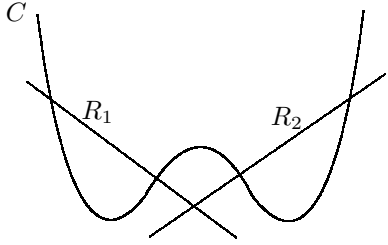

%\vspace{-6cm}

Choose a line bundle $L$ of degree $d$ on $C$, pull it back to $C\times \P^1$ and call it again
$L$.  Since $L$ is trivial when restricted to $\{p_i\}\times \P^1$, we can
glue it with the trivial line bundle on
$X$ and, thus, we obtain a line bundle $\mathcal{L}$ on the family $Y\to \P^1$ of relative degree $d$.

\begin{lemma}\label{L:bal1}
The line bundle $\mathcal{L}$ is properly balanced.
\end{lemma}
\begin{proof}
Since the property of being properly balanced is an open condition, it is enough to check that
$\L$ is properly balanced on the three special fibers of $f:Y\to \P^1$.
According to Remark \ref{rmk-ineq}, it is enough to check the basic inequality for the
three subcurves $R_1\cup R_2$, $R_1$ and $R_2$. The balancing condition for $R_1\cup R_2$
$$\left|\deg_{R_1\cup R_2}(\mathcal{L})-\frac{d\cdot 2}{2g-2} \right|\leq \frac{4}{2}, $$
is true because $\deg_{R_1\cup R_2}(\mathcal{L})=0$ and $0\leq d<2g-2$.
The balancing condition for each of the subcurves $R_i$ ($i=1, 2$) is
$$\left|\deg_{R_i}(\mathcal{L})-\frac{d\cdot 1}{2g-2} \right|\leq \frac{3}{2}, $$
which is satisfied because $\deg_{R_i}(\mathcal{L})=0$ and $0\leq d<2g-2$.

%On the general fiber of $f:Y\to \P^1$, the balancing condition amounts to prove that
%$$\left|\deg_{Q}(\mathcal{L})-\frac{d\cdot 2}{2g-2} \right|\leq \frac{4}{2}, $$
%which is true because $\deg_{Q}(\mathcal{L})=0$ and $0\leq d<2g-2$.
\end{proof}

We call $\w{F}$ the family $f: Y\to \P^1$ endowed with the line bundle $\mathcal{L}$.
Forgetting the line bundle $\mathcal{L}$, we are left with the family $F$ of
\cite[p. 158]{arbcorn}.
%Using the results of loc. cit.,
We can compute the degree of the pull-backs of the
boundary classes in $\Pic(\pdtb)$ to the curve $\w{F}$:
\begin{equation}\label{int-F}
\left\{\begin{aligned}
& \deg_{\w{F}}\O(\w{\delta_0})=-1, &\\
& \deg_{\w{F}}\O(\w{\delta}_i)= 0 & \text{ if } 1\leq i \text{ and } \kdg\nmid (2i-1) \text{ or } i=g/2, \\
& \deg_{\w{F}}\O(\w{\delta}_i^1)=\deg_{\w{F}}\O(\w{\delta}_i^2)=0 &\text{ if } 1\leq i <g/2  \text{ and } \kdg \mid (2i-1). \\
\end{aligned}\right.
\end{equation}
The first relation follows from the fact that $\deg_{\w{F}}\O(\w{\delta_0})=\deg_F \O(\delta_0)$
(by using the projection formula) and the relation $\deg_F\O( \delta_0)=-1$ proved
in \cite[p. 158]{arbcorn}. The last two relations follow by the obvious fact
that $\w{F}$ does not meet the divisors $\w{\delta}_i$ or $\w{\delta}_i^1$ and $\w{\delta}_i^2$ for $i\geq 1$.

{\bf The Families $\w{F_1'}$ and $\w{F_2'}$}

We start with the same family of conics $\phi: X\to \P^1$ that we considered in the
construction of the family $\w{F}$. Let $C$ be a fixed irreducible, smooth and projective
curve of genus $g-3$, $E$ be a fixed irreducible, smooth and projective elliptic curve
and take points $p_1\in E$ and $p_2,p_3,p_4\in C$.
We construct a surface $Z$ by setting
$$Z=\left(X\coprod (C\times \P^1) \coprod (E\times \P^1)\right)/(E_i\sim \{p_i\}\times \P^1\: :\: i=1,\cdots, 4).
$$
We get a family $g:Z \to \P^1$ of stable curves of genus $g$: the general fiber of $g$ consists of
$C$, $E$ and a smooth conic $Q$ intersecting as in Figure \ref{Figure3}.
%such that $|Q\cap C|=3$, $|Q\cap E|=1$ a(see \cite[Fig. 6]{arbcorn}).
The three special fibers consist of $C$, $E$ and two lines $R_1$ and $R_2$,
intersecting as shown in Figure \ref{Figure 4}.
%such that $|R_1\cap R_2|=1$, $|R_1\cap C|=2$, $|R_2\cap E|=|R_2\cap C|=1$,
%$|E\cap R_1|=|E\cap C|=0$ (see \cite[Fig. 7]{arbcorn}).

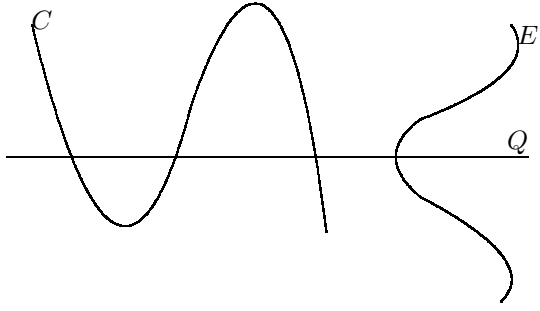
\begin{figure}[ht]
\begin{center}
%TeXCAD Picture [Figura6.pic]. Options:
%\grade{\on}
%\emlines{\off}
%\epic{\off}
%\beziermacro{\on}
%\reduce{\on}
%\snapping{\off}
%\pvinsert{% Your \input, \def, etc. here}
%\quality{8.000}
%\graddiff{0.005}
%\snapasp{1}
%\zoom{4.0000}
\unitlength .6mm % = 2.845pt
\linethickness{0.4pt}
\ifx\plotpoint\undefined\newsavebox{\plotpoint}\fi % GNUPLOT compatibility
\begin{picture}(131,70.125)(0,115)
\qbezier(18.5,176)(37.5,96.75)(53.5,158.5)
\qbezier(53.5,158.5)(72.75,214.125)(83,130.25)
\put(13,146.75){\line(1,0){1.75}}
\put(14.75,146.75){\line(1,0){112.5}}
\qbezier(123.5,176)(131,165.5)(103.5,155)
\qbezier(103.5,155)(93,147)(103.5,138)
\qbezier(103.5,138)(130.625,123.625)(121.25,114.75)
\put(127.25,173.75){\makebox(0,0)[cc]{$E$}}
\put(125,150){\makebox(0,0)[cc]{$Q$}}
\put(20.75,177){\makebox(0,0)[cc]{$C$}}
\end{picture}
\end{center}
\caption{The general fibers of $g:Z\to \P^1$.}\label{Figure3}
\end{figure}

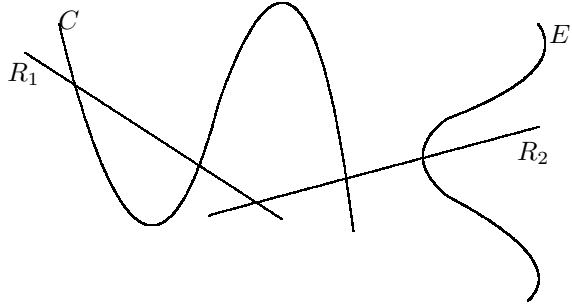
\begin{figure}[ht]
\begin{center}
%TeXCAD Picture [Figura7.pic]. Options:
%\grade{\on}
%\emlines{\off}
%\epic{\off}
%\beziermacro{\on}
%\reduce{\on}
%\snapping{\off}
%\pvinsert{% Your \input, \def, etc. here}
%\quality{8.000}
%\graddiff{0.005}
%\snapasp{1}
%\zoom{4.0000}
\unitlength .6mm % = 2.845pt
\linethickness{0.4pt}
\ifx\plotpoint\undefined\newsavebox{\plotpoint}\fi % GNUPLOT compatibility
\begin{picture}(131,60.125)(0,120)
\qbezier(18.5,176)(37.5,96.75)(53.5,158.5)
\qbezier(53.5,158.5)(72.75,214.125)(83,130.25)
\qbezier(123.5,176)(131,165.5)(103.5,155)
\qbezier(103.5,155)(93,147)(103.5,138)
\qbezier(103.5,138)(130.625,123.625)(121.25,114.75)
\put(128.25,173.75){\makebox(0,0)[cc]{$E$}}
\put(20.75,177){\makebox(0,0)[cc]{$C$}}
\put(11,169.5){\line(1,0){.25}}
%\emline(11.25,169.5)(67.5,132.75)
\multiput(11.25,169.5)(.05160550459,-.03371559633){1090}{\line(1,0){.05160550459}}
%\end
\put(51.5,133.5){\line(0,1){.25}}
%\emline(51.5,133.75)(123.75,153.25)
\multiput(51.5,133.75)(.125,.0337370242){578}{\line(1,0){.125}}
%\end
\put(10.75,165){\makebox(0,0)[cc]{$R_1$}}
\put(122.5,147.5){\makebox(0,0)[cc]{$R_2$}}
\end{picture}
\end{center}
\caption{The three special fibers of $g:Z\to \P^1$.}\label{Figure 4}
\end{figure}

We choose two line bundles of degree $d$ and $d-3$ on $C$, we pull them back to $C\times \P^1$ and call
them, respectively, $L_1$ and $L_2$. Similarly, we choose two line bundles of degree $0$ and $1$ on $E$,
we pull them back to $E\times \P^1$ and call them, respectively, $M_1$ and $M_2$.
We glue the line bundle $L_1$ (resp. $L_2$) on $C\times \P^1$, the line bundle $M_1$ (resp. $M_2$) on
$E\times \P^1$ and the line bundle $\O_X$ (resp. $\omega_{X/\P^1}^{-1}$, the relative anti-canonical
bundle of $\phi:X\to \P^1$) on $X$, obtaining a line bundle $\M_1$ (resp. $\M_2$) on $Z$
of relative degree $d$.

%Since the line bundles $L_i$ are trivial in a neighborhood of each of the three sections
%$\{p_j\}\times \P^1$ (for $j=2,3,4$) and the line bundles $M_i$ are trivial in a neighborhood of
%the section $\{p_1\}\times \P^1$, we can glue (for $i=1, 2$) the trivial line bundle on $X$,
%the line bundle $L_i$ on $C\times \P^1$ and the line bundle $M_i$ on $E\times \P^1$, obtaining
%a line bundle $\mathcal{M}_i$ on $Z$ of relative degree $d$.

\begin{lemma}
The line bundle $\mathcal{M}_1$ is properly balanced if $0\leq d\leq g-1$. The line bundle
$\mathcal{M}_2$ is properly balanced if $g-1\leq d<2g-2$.
\end{lemma}
\begin{proof}
The proof is straightforward and similar to the one of Lemma \ref{L:bal1}: we leave it to the reader.

%Since the property of being properly balanced is an open condition, it is enough to check that $\M$ is properly balanced on the three special fibers of $g:Z\to \P^1$. By Remark \ref{rmk-ineq}, it is enough to check the basic inequality for the subcurves $E$, $C$, $R_1$ and $R_2\cup E$. The basic inequality for $C$,
%\begin{equation*}
%\left|\deg_{C}(\mathcal{M}_i)-\frac{d\cdot (2g-5)}{2g-2} \right|\leq \frac{3}{2}
%\end{equation*}
%is satisfied since $\deg_{C}(\mathcal{M}_i)=d$ if $0\leq d\leq g-1$ and $\deg_{C}(\mathcal{M}_i)=d-3$ if $g-1\leq d<2g-2$. For $E$, we get
%\begin{equation*}
%\left|\deg_{E}(\mathcal{M}_i)-\frac{d\cdot 1}{2g-2} \right|\leq \frac{1}{2},
%\end{equation*}
%which is satisfied since $\deg_{E}(\mathcal{M}_i)=0$ if $0\leq d\leq g-1$ and $\deg_{E}(\mathcal{M}_i)=1$ if $g-1\leq d< 2g-2$.
%The basic inequality for $R_1$ is
%$$\left|\deg_{R_1}(\mathcal{M}_i)-\frac{d\cdot 1}{2g-2} \right|\leq \frac{3}{2},$$
%which is satisfied since $\deg_{R_1}(\mathcal{M}_i)=0$ if $0\leq d\leq g-1$ and $\deg_{R_1}(\mathcal{M}_i)=1$ if $g-1\leq d< 2g-2$ (note that the relative anti-canonical bundle $\omega_{X/\P^1}^{-1}$ has degree $1$ on $R_1$ and $R_2$). Finally, the basic inequality for $R_2\cup E$
%$$\left|\deg_{R_2\cup E}(\mathcal{M}_i)-\frac{d\cdot 2}{2g-2} \right|\leq \frac{2}{2}=1,$$
%is satisfied since $\deg_{R_2\cup E}(\mathcal{M}_i)=0$ if $0\leq d\leq g-1$ and $\deg_{R_2\cup E}(\mathcal{M}_i)=2$ if $g-1\leq d< 2g-2$.
\end{proof}

If $0\leq d\leq g-1$, we call $\w{F_1'}$ the family $g: Z\to \P^1$ endowed with the line bundle
$\mathcal{M}_1$; if $g-1\leq d<2g-2$, we call $\w{F_2'}$ the family $g:Z\to \P^1$ endowed with the line bundle
$\mathcal{M}_2$. Both families $\w{F_1'}$ and $\w{F_2'}$, when defined, are liftings of the family
$F'$ of \cite[p. 158]{arbcorn}. We can compute the degree of the pull-backs of some
of the boundary classes in $\Pic(\pdtb)$ to the curves $\w{F_1'}$ and $\w{F_2'}$, in the ranges of degrees where
they are defined (note that $\w{\Phi}_d^{-1}(\delta_1)$ is the  union of two irreducible divisors if and only if
$\kdg=1$, i.e. iff $d=g-1$):
\begin{equation}\label{int-F'}
\left\{\begin{aligned}
%& \deg_{\w{F_1'}}\w{\delta_0}=\deg_{\w{F_1'}}\w{\delta_0} = \deg_{F'} \delta_0=0, \\
& \deg_{\w{F_1'}}\O(\w{\delta_1})=\deg_{\w{F_2'}}\O(\w{\delta_1})=-1 \text{ if } d \neq g-1, \\
%& \deg_{\w{F_2'}}\w{\delta_1}=-1 \text{ if } g-1< d <2g-2, \\
& \deg_{\w{F_1'}}\O(\w{\delta_1^1})=\deg_{\w{F_2'}}\O(\w{\delta_1^2})=-1 \text{ and }\deg_{\w{F_1'}}\O(\w{\delta_1^2})=\deg_{\w{F_2'}}\O(\w{\delta_1^1})=0 \text{ if } d=g-1, \\
& \deg_{\w{F_1'}}\O(\w{\delta}_i)= \deg_{\w{F_2'}}\O(\w{\delta}_i)= 0 \text{ if } 1< i \text{ and }
\kdg\nmid (2i-1) \text{Ê\: or } i=g/2, \\
& \deg_{\w{F_1'}}\O(\w{\delta_i^j})= \deg_{\w{F_2'}}\O(\w{\delta_i^j})
=0 \text{ if } 1<i<g/2 \text{ and } \kdg \mid (2i-1), \text{ for } j=1, 2. \\
\end{aligned}\right.
\end{equation}
The first relation follow, by using the projection formula, from the relation
$\deg_{F'}\O(\delta_1)=-1$ proved in \cite[p. 159]{arbcorn}. The second relation is deduced in a similar way
using the projection formula and the (easily checked) fact that $\w{F_1'}$ does not meet $\w{\delta}_1^2$
and that $\w{F_2'}$ does not meet $\w{\delta}_1^1$.
The last two relations
follow from the fact that $\w{F_1'}$ and $\w{F_2'}$ do not meet the divisors $\w{\delta}_i$ or $\w{\delta}_i^1$
and $\w{\delta}_i^2$  for $i>1$.

{\bf The Families $\w{F_{h,1}}$ and $\w{F_{h,2}}$} (for $1\leq h\leq \frac{g-2}{2}$)

Fix irreducible, smooth and projective curves $C_1$, $C_2$ and $\Gamma$ of genera $h$, $g-h-1$ and
$1$, and points $x_1\in C_1$, $x_2\in C_2$ and $\gamma\in \Gamma$. Consider the surfaces
$Y_1=C_1\times \Gamma$, $Y_3=C_2\times \Gamma$ and $Y_2$ given by the blow-up of $\Gamma\times \Gamma$
at $(\gamma, \gamma)$. Let us denote by $p_2:Y_2\to \Gamma$ the map given by composing the blow-down
$Y_2\to \Gamma\times \Gamma$ with the second projection, and by $\pi_1: Y_1\to \Gamma$ and $\pi_3:Y_3\to \Gamma$ the projections
along the second factor.
As in \cite[p. 156]{arbcorn}, we set (see also Figure \ref{Figure 5}):
$$\begin{aligned}
&A=\{x_1\}\times \Gamma, \\
&B=\{x_2\}\times \Gamma, \\
&E= \text{ exceptional divisor of the blow-up of } \Gamma\times \Gamma \text{ at } (\gamma,\gamma),\\
&\Delta= \text{ proper transform of the diagonal in } Y_2, \\
&S= \text{ proper transform of } \{\gamma\}\times \Gamma \text{ in } Y_2,\\
& T= \text{ proper transform of } \Gamma \times \{\gamma\}\text{ in } Y_2.
\end{aligned}
$$

\begin{figure}[ht]
\begin{center}

%\vspace{-3cm}

%TeXCAD Picture [Figura1.pic]. Options:
%\grade{\on}
%\emlines{\off}
%\epic{\off}
%\beziermacro{\on}
%\reduce{\on}
%\snapping{\off}
%\pvinsert{% Your \input, \def, etc. here}
%\quality{8.000}
%\graddiff{0.005}
%\snapasp{1}
%\zoom{4.0000}
\unitlength .65mm % = 2.845pt
\linethickness{0.4pt}
\ifx\plotpoint\undefined\newsavebox{\plotpoint}\fi % GNUPLOT compatibility
\begin{picture}(150,60.75)(30,65)
\put(88.5,66.5){\framebox(43.5,48.25)[cc]{}}
\put(22.75,66.5){\framebox(43.5,48.25)[cc]{}}
\put(156.25,66.25){\framebox(43.5,48.25)[cc]{}}
\put(88.75,102){\line(1,0){28.25}}
\put(23,102){\line(1,0){43.25}}
\put(44,104){\makebox(0,0)[cc]{$A$}}
\put(14,100.25){\vector(0,-1){26.5}}
\put(147.5,100){\vector(0,-1){26.5}}
\put(109.25,63.5){\makebox(0,0)[cc]{}}
\put(43.5,63.5){\makebox(0,0)[cc]{}}
\put(177,63.25){\makebox(0,0)[cc]{}}
\put(109.5,63){\makebox(0,0)[cc]{$\Gamma$}}
\put(43.75,63){\makebox(0,0)[cc]{$\Gamma$}}
\put(177.25,62.75){\makebox(0,0)[cc]{$\Gamma$}}
\put(18.25,86.5){\makebox(0,0)[cc]{$\pi_1$}}
\put(151.75,86.25){\makebox(0,0)[cc]{$\pi_3$}}
\put(61.75,86.5){\makebox(0,0)[cc]{}}
\put(128.75,87.25){\makebox(0,0)[cc]{}}
\put(63,87.25){\makebox(0,0)[cc]{}}
\put(196.5,87){\makebox(0,0)[cc]{}}
\put(63,86.5){\makebox(0,0)[cc]{$C_1$}}
%\emline(88.5,66.5)(132,114.75)
\multiput(88.5,66.5)(.03372093023,.03740310078){1290}{\line(0,1){.03740310078}}
%\end
\put(123.75,98){\line(0,-1){31.5}}
\qbezier(112.25,106.5)(112.75,97)(126.25,90.5)
\put(156.25,78.25){\line(1,0){43.75}}
\put(177,81.75){\makebox(0,0)[cc]{$B$}}
\put(94.25,104.5){\makebox(0,0)[cc]{$S$}}
\put(83.5,86.75){\makebox(0,0)[cc]{$\Gamma$}}
\put(101.25,86.75){\makebox(0,0)[cc]{$\Delta$}}
\put(119.5,78){\makebox(0,0)[cc]{$T$}}
\put(113.5,108.75){\makebox(0,0)[cc]{$E$}}
\put(196.75,87.25){\makebox(0,0)[cc]{$C_2$}}
\end{picture}
\end{center}
\caption{Constructing $f:X\to \Gamma$.}\label{Figure 5}
\end{figure}
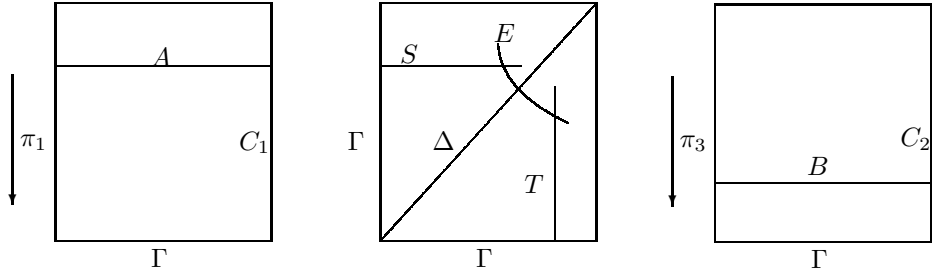

We construct a surface $X$ by identifying $S$ with $A$ and $\Delta$ with $B$. The surface $X$ comes
equipped with a projection $f:X\to \Gamma$. The fibers over all the points $\gamma'\neq \gamma$ are
shown in Figure \ref{Figure6},
%in \cite[Fig. 2]{arbcorn},
while the fiber over the point $\gamma$ is shown in Figure \ref{Figure7}.
%as in \cite[Fig. 3]{arbcorn}.

\begin{figure}[ht]
\begin{center}

%TeXCAD Picture [Figura2.pic]. Options:
%\grade{\on}
%\emlines{\off}
%\epic{\off}
%\beziermacro{\on}
%\reduce{\on}
%\snapping{\off}
%\pvinsert{% Your \input, \def, etc. here}
%\quality{8.000}
%\graddiff{0.005}
%\snapasp{1}
%\zoom{4.0000}
\unitlength .65mm % = 2.845pt
\linethickness{0.4pt}
\ifx\plotpoint\undefined\newsavebox{\plotpoint}\fi % GNUPLOT compatibility
\begin{picture}(122.75,50.5)(0,145)
\qbezier(49.25,173.5)(76.875,158.125)(102,164.25)
\qbezier(87,153.5)(95.125,175.625)(122.75,183.25)
\qbezier(17.5,153.25)(52.125,160.125)(62.25,184.5)
\put(17.75,157){\makebox(0,0)[cc]{$C_1$}}
\put(30.75,153.75){\makebox(0,0)[cc]{$h$}}
\put(72.5,167){\makebox(0,0)[cc]{$\Gamma$}}
\put(64.25,162.5){\makebox(0,0)[cc]{$1$}}
\put(113.25,183.75){\makebox(0,0)[cc]{$C_2$}}
\put(115,173.75){\makebox(0,0)[cc]{$g-h-1$}}
\end{picture}

\end{center}
\caption{The general fiber of $f:X\to \Gamma$.}\label{Figure6}
\end{figure}
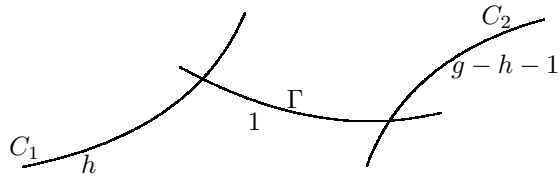

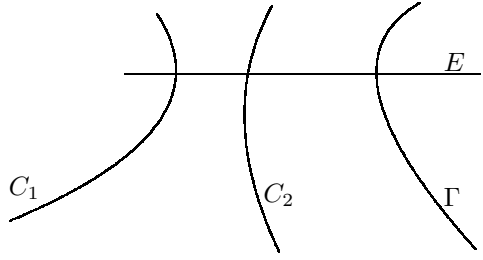
\begin{figure}[ht]
\begin{center}

%TeXCAD Picture [Figura3.pic]. Options:
%\grade{\on}
%\emlines{\off}
%\epic{\off}
%\beziermacro{\on}
%\reduce{\on}
%\snapping{\off}
%\pvinsert{% Your \input, \def, etc. here}
%\quality{8.000}
%\graddiff{0.005}
%\snapasp{1}
%\zoom{4.0000}
\unitlength .65mm % = 2.845pt
\linethickness{0.4pt}
\ifx\plotpoint\undefined\newsavebox{\plotpoint}\fi % GNUPLOT compatibility
\begin{picture}(93.5,55)(0,135)
\put(21.5,170.5){\line(1,0){72}}
\qbezier(-1.75,140.5)(43.375,159.375)(28,182.75)
\qbezier(51.25,184.25)(39.5,162.25)(52.75,134.25)
\qbezier(81.25,185)(59.125,172.125)(92.5,134.75)
\put(88.25,173){\makebox(0,0)[cc]{$E$}}
\put(1.25,147.25){\makebox(0,0)[cc]{$C_1$}}
\put(52.75,145.75){\makebox(0,0)[cc]{$C_2$}}
\put(87.75,145.5){\makebox(0,0)[cc]{$\Gamma$}}
\end{picture}
\end{center}
\caption{The special fiber of $f:X\to \Gamma$.}\label{Figure7}
\end{figure}

We will first construct several  line bundles over the three surfaces $Y_1$, $Y_2$ and $Y_3$, and then
we will glue them in a suitable way.

Consider the line bundles $M_i$ ($i=1,\cdots, 4$) on $Y_2$ given by
\begin{equation*}
M_1:=\O_{Y_2}, \: M_2:=\O_{Y_2}(\Delta), \: M_3:=\O_{Y_2}(\Delta+E), \: M_4:=\O_{Y_2}(2\Delta+E).
\end{equation*}
Using that $\deg_E \O(E)=-1$, we get that the restrictions of $M_i$ to $E$ and $T$ have degrees:
\begin{equation*}
(\deg_E M_i, \deg_T M_i)=\begin{cases}
(0,0) & \text{ if } i=1, \\
(1,0) & \text{ if } i=2, \\
(0,1) & \text{ if } i=3, \\
(1,1) & \text{ if } i=4. \\
\end{cases}
\end{equation*}
Notice that the diagonal $\ov{\Delta}$ of $\Gamma\times \Gamma$ is such that
$\O_{\Gamma\times \Gamma}(\ov \Delta)_{|\ov \Delta}=
\O_{\ov \Delta}$ since $\Gamma$ is an elliptic curve. By applying the projection
formula to the blow-up $Y_2\to \Gamma\times \Gamma$, we get that $\O_{Y_2}(\Delta)_{|\Delta}=
\O_{\Delta}(-\gamma)$. Using this, we can easily compute the
restrictions of $M_i$ to $S$ and $\Delta$ (which are canonically isomorphic to $\Gamma$):
\begin{equation}\label{restr-M}
(M_{i})_{|\Delta}=\begin{cases}
\O _{\Gamma}& \text{ if } i=1, 3\\
\O_{\Gamma}(-\gamma) & \text{ if } i=2, 4
\end{cases}
\: \text{ and } \:
(M_{i})_{|S}=\begin{cases}
\O _{\Gamma}& \text{ if } i=1, 2\\
\O_{\Gamma}(\gamma) & \text{ if } i=3, 4.
\end{cases}
\end{equation}
Consider now the integers $\alpha_1$, $\alpha_2$  defined by:
$$\begin{sis}
& \alpha_1:=\left\lfloor\frac{d(2g-2h-3)}{2g-2}\right\rfloor , \:
\alpha_2:=\left\lceil\frac{d(2g-2h-3)}{2g-2}\right\rceil  \text{, if }
\frac{d(2g-2h-3)}{2g-2}\equiv \frac{1}{2} \mod \Z \\
& \alpha_1=\alpha_2:= \text{ the unique integer which is closest to } \frac{d(2g-2h-3)}{2g-2}
\text{, otherwise. } \\
\end{sis}
$$
Take two line bundles on $C_2$ of degrees $\alpha_1$ and $\alpha_2$, and call, respectively,
$L_1$ and $L_2$ their pull-backs to $Y_3=C_2\times \Gamma$. We may assume that $L_1=L_2$
if $\alpha_1=\alpha_2$.
%Observe that the divisor $D_{h+1}\subset \pdb$ is irreducible if and only
%$\alpha_1=\alpha_2$.

Analogously, consider the integers $\beta_1$, $\beta_2$  defined by:
$$\begin{sis}
& \beta_1:=\left\lfloor\frac{d(2h-1)}{2g-2}\right\rfloor , \:
\beta_2:=\left\lceil\frac{d(2h-1)}{2g-2}\right\rceil  \text{, if }
\frac{d(2h-1)}{2g-2}\equiv \frac{1}{2} \mod \Z \\
& \beta_1=\beta_2:= \text{ the unique integer which is closest to } \frac{d(2h-1)}{2g-2}
\text{, otherwise. } \\
\end{sis}
$$
Consider two line bundles on $C_1$ of degrees $\beta_1$ and $\beta_2$, and call, respectively,
$N_1$ and $N_2$ their pull-back to $Y_1=C_1\times \Gamma$. We may assume that $N_1=N_2$
if $\beta_1=\beta_2$.
%Observe that the divisor $D_{h}\subset \pdb$ is irreducible if and only $\beta_1=\beta_2$.

We now want to define two (possibly equal) line bundles $\I_1$ and $\I_2$ on $X$, by gluing in a suitable way some of the line bundles on $Y_1$, $Y_2$ and $Y_3$, we have just defined.
We shall distinguish between several cases:

\underline{CASE A: } $\frac{d(2g-2h-3)}{2g-2}\not\equiv  \frac{1}{2} \mod \Z$ (i.e. $\alpha_1=\alpha_2$).
%and $\frac{d(2h-1)}{2g-2}\not\in \frac{1}{2}\Z/\Z $ (i.e. $\beta_1=\beta_2$).
In this case, we have that
\begin{equation}\label{deg1}
\alpha_1-\frac{1}{2}<\frac{d(2g-2h-3)}{2g-2}<\alpha_1+\frac{1}{2}
\:\mbox{ and }\:
\beta_1-\frac{1}{2}<\frac{d(2h-1)}{2g-2}\leq\beta_1+\frac{1}{2}.
\end{equation}

\underline{Subcase A1:} $0\leq d\leq g-1$.
Using the inequalities (\ref{deg1}),
we get that
$$-1\leq -1+\frac{d}{g-1}=-1+d-\frac{d(2g-2h-3)}{2g-2}-\frac{d(2h-1)}{2g-2} < d-\alpha_1-\beta_1<$$
\begin{equation}\label{ext-deg}
<1+d-\frac{d(2g-2h-3)}{2g-2}-\frac{d(2h-1)}{2g-2} =
1+\frac{d}{g-1}<2.
\end{equation}

If $d-\alpha_1-\beta_1=0$ then we define $\I_1=\I_2$ to be equal to the line bundle on $X$ obtained by gluing
$N_1$, $M_1$ and  $L_1=L_2$, which is possible since, by (\ref{restr-M}), we have that
$(N_1)_{|A}=\O_{\Gamma}=(M_1)_{|S}$ and $(L_1)_{|B}=\O_{\Gamma}=(M_1)_{|\Delta}$.
%We call $\w{F'_{h,1}}=\w{F'_{h,2}}$
%the family $f:X\to \Gamma$ endowed with this line bundle.

Otherwise, if  $d-\alpha_1-\beta_1=1$, then we define $\I_1=\I_2$ to be equal to the line bundle on $X$ obtained by
gluing the sheaves $N_1$, $M_2$ and $L_1\otimes \pi_3^*\O_{\Gamma}(-\gamma)$,
which is possible since, by (\ref{restr-M}), we have that
$(N_1)_{|A}=\O_{\Gamma}=(M_2)_{|S}$ and $(L_1\otimes \pi_3^*\O_{\Gamma}(-\gamma))_{|B}=\O_{\Gamma}(-\gamma)=(M_2)_{|\Delta}$.
%We call $\w{F'_{h,1}}=\w{F'_{h,2}}$
%the family $f:X\to \Gamma$ endowed with this line bundle.

\underline{Subcase A2:} $g-1< d <2g-2$.

Arguing similarly to the above inequality (\ref{ext-deg}), we get that $d-\alpha_1-\beta_1=1, 2$.

If $d-\alpha_1-\beta_1=1$,  then we define $\I_1=\I_2$ to be equal to the line bundle on $X$ obtained by gluing $N_1\otimes \pi_1^*\O_{\Gamma}(\gamma)$, $M_3$ and $L_1$, which is possible since, by
(\ref{restr-M}), we have that $(N_1\otimes \pi_1^*\O_{\Gamma}(\gamma))_{|A}=\O_{\Gamma}(\gamma)=(M_3)_{|S}$ and
$(L_1)_{|B}=\O_{\Gamma}=(M_3)_{|\Delta}$.

If $d-\alpha_1-\beta_1=2$,  then we define $\I_1=\I_2$ to be equal to the line bundle on $X$ obtained by gluing $N_1\otimes \pi_1^*\O_{\Gamma}(\gamma)$, $M_4$ and $L_1\otimes \pi_3^*\O_{\Gamma}(-\gamma)$,
which is possible since, by (\ref{restr-M}), we have that $(N_1\otimes \pi_1^*\O_{\Gamma}(\gamma))_{|A}=\O_{\Gamma}(\gamma)=(M_4)_{|S}$ and
$(L_1\otimes \pi_3^*\O_{\Gamma}(-\gamma))_{|B}=\O_{\Gamma}(-\gamma)=(M_4)_{|\Delta}$.

%\underline{CASE B: } $\frac{d(2g-2h-3)}{2g-2}\not\in \frac{1}{2}\Z/\Z$ (i.e. $\alpha_1=\alpha_2$)
%and $\frac{d(2h-1)}{2g-2}\in \frac{1}{2}\Z/\Z$ (i.e. $\beta_1\neq\beta_2$).

%\underline{Subcase B1:} $0\leq d< g-1$.

%\underline{Subcase B2:} $g-1< d <2g-2$.

\underline{CASE B: } $\frac{d(2g-2h-3)}{2g-2}\equiv \frac{1}{2} \mod \Z$ (i.e. $\alpha_1=\alpha_2-1$).
%and $\frac{d(2h-1)}{2g-2}\not \in \frac{1}{2}\Z/\Z$ (i.e. $\beta_1=\beta_2$).

In this case, we have that
$\alpha_1+\frac{1}{2}\frac{d(2g-2h-3)}{2g-2}=\alpha_2-\frac{1}{2}$, $\beta_1-\frac{1}{2}<\frac{d(2g-2h-3)}{2g-2}\leq\beta_1+\frac{1}{2}$,
%\:\mbox{ and }\:
and that
$\beta_2-\frac{1}{2}\leq\frac{(2h-1)}{2g-2}<\beta_2+\frac{1}{2}.
$
So, arguing similarly to the above inequality (\ref{ext-deg}), we get that
$$d-\alpha_1-\beta_2=\begin{cases}
1 & \text{ if } 0\leq d\leq  g-1,\\
2 & \text{ if } g-1 < d <2g-2.
\end{cases}
$$

If $0\leq d\leq g-1$, we define $\I_1$ to be equal to the line bundle on $X$ obtained by
gluing the sheaves $N_2$, $M_2$ and $L_1\otimes \pi_3^*\O_{\Gamma}(-\gamma)$, which is possible since,
by (\ref{restr-M}), we have that
$(N_2)_{|A}=\O_{\Gamma}=(M_2)_{|S}$ and $(L_1\otimes
\pi_3^*\O_{\Gamma}(-\gamma))_{|B}=\O_{\Gamma}(-\gamma)=(M_2)_{|\Delta}$.

If $g-1< d <2g-2$, we define $\I_1$ to be equal to the line bundle on $X$ obtained by gluing
$N_2\otimes \pi_1^*\O_{\Gamma}(\gamma)$, $M_4$ and $L_1\otimes \pi_3^*\O_{\Gamma}(-\gamma)$.
which is possible since, by (\ref{restr-M}), we have that $(N_2\otimes
\pi_1^*\O_{\Gamma}(\gamma))_{|A}=\O_{\Gamma}(\gamma)=(M_4)_{|S}$ and
$(L_1\otimes \pi_3^*\O_{\Gamma}(-\gamma))_{|B}=\O_{\Gamma}(-\gamma)=(M_4)_{|\Delta}$.

Similarly,  we get that
$$d-\alpha_2-\beta_1=\begin{cases}
0 & \text{ if } 0\leq d< g-1,\\
1 & \text{ if } g-1 \leq d <2g-2.
\end{cases}
$$

If $0\leq d <g-1$, we define $\I_2$ to be equal to the line bundle on $X$ obtained by gluing
$N_1$, $M_1$ and $L_2$,
which is possible since, by (\ref{restr-M}), we have that $(N_1)_{|A}=\O_{\Gamma}=(M_1)_{|S}$ and
$(L_1)_{|B}=\O_{\Gamma}=(M_1)_{|\Delta}$.

If $g-1\leq d <2g-2$, we define $\I_2$ to be equal to the line bundle on $X$ obtained by gluing
$N_1\otimes \pi_1^*\O_{\Gamma}(\gamma)$, $M_3$ and $L_2$,
which is possible since, by (\ref{restr-M}), we have that $(N_1\otimes
\pi_1^*\O_{\Gamma}(\gamma))_{|A}=\O_{\Gamma}(\gamma)=(M_3)_{|S}$ and
$(L_2)_{|B}=\O_{\Gamma}=(M_3)_{|\Delta}$.

\begin{lemma}
The line bundles $\I_1$ and $\I_2$ on $X$ are properly balanced of relative degree $d$.
\end{lemma}
\begin{proof}
The proof is straightforward and similar to the one of Lemma \ref{L:bal1}: we leave it to the reader.

\end{proof}

We call $\w{F_{h,1}}$ the family $f: X\to \Gamma$ endowed with the line bundle
$\I_1$ and  $\w{F_{h,2}}$ the family $f:X\to \Gamma$ endowed with the line bundle
$\I_2$. Note that $\w{F_{h,1}}=\w{F_{h,2}}$ if and only if we are in case A, which happens exactly when
$\kdg\nmid 2h+1$.
% or in other words when the divisor $\D_{h+1}\subset \pdb$ is irreducible.
Both families $\w{F_{h,1}}$ and $\w{F_{h,2}}$ are liftings of the family
$F_h$ of \cite[p. 156]{arbcorn}. We can compute the degrees of the pull-backs of some
of the boundary classes in $\Pic(\pdtb)$ to the curves $\w{F_{h,1}}$ and $\w{F_{h,2}}$:

\begin{equation}\label{int-Fh}
\left\{\begin{aligned}
& \deg_{\w{F_{h,1}}}\O(\w{\delta_{h+1}})=
%\deg_{\w{F_{2,h}}}\w{\delta_{h+1}}=
-1 \text{ if } \kdg \nmid 2h+1 \text{ or } h+1=g/2, \\
& \deg_{\w{F_{h,1}}}\O(\w{\delta_{h+1}^1})=\deg_{\w{F_{h, 2}}}\O(\w{\delta_{h+1}^2})=-1 \text{ if } \kdg \mid 2h+1 \text{ and } h+1\neq g/2, \\
& \deg_{\w{F_{h,1}}}\O(\w{\delta_{h+1}^2})= \deg_{\w{F_{h, 2}}}\O(\w{\delta_{h+1}^1})= 0 \text{ if }  \kdg \mid 2h+1 \text{ and }Êh+1\neq g/2,\\
& \deg_{\w{F_{h, 1}}}\O(\w{\delta}_i)
%\deg_{\w{F_{2,h}}}\w{\delta}_i
=0 \text{ if } h+1<i \text{ and } \kdg \nmid (2i-1) \text{ or } i=g/2, \\
& \deg_{\w{F_{h, 1}}}\O(\w{\delta_i^j})= \deg_{\w{F_{h, 2}}}\O(\w{\delta_i^j})=0 \text{ if } h+1<i<g/2 \text{ and } \kdg \mid (2i-1), \text{ for } j=1, 2. \\
\end{aligned}\right.
\end{equation}
The first relation follow, by using the projection formula, from the relation
$\deg_{F_h}$ $\O(\delta_{h+1})=-1$ proved in \cite[p. 157]{arbcorn}. The second and third relations
are deduced in a similar way using the projection formula and the (easily checked) fact that $\w{F_{h,1}}$ does not meet
$\w{\delta}_{h+1}^2$ and  $\w{F_{h,2}}$ does not meet  $\w{\delta}_{h+1}^1$.
The last two relations
follow from the fact that $\w{F_{h,1}}$ and $\w{F_{h,2}}$ do not meet the divisors $\w{\delta}_i$ or $\w{\delta}_i^1$
and $\w{\delta}_i^2$ for  $i>h+1$.

With the help of the above families, we can finally conclude the proof of our main theorem.

\begin{proof}[Proof of Theorem \ref{T:ex-seq}]
As observed before,  it is enough to prove that the line bundles  associated to the boundary
divisors $\{\w{\delta}_i\: : \: \kdg\nmid 2i-1 \text{ or } i=g/2\}$,  $\{\w{\delta}_i^1, \w{\delta}_i^2\: :\: \kdg \mid 2i-1 \text{ and }Êi\neq g/2\}$
(for $0\leq i\leq g/2$) are  linearly independent on $\pdtb$.
Suppose there is a linear relation
\begin{equation}\label{rel-div}
\O\left(\sum_{\stackrel{\kdg\nmid \: 2i-1}{\text{ or } i=g/2}} a_i \w{\delta}_i+\sum_{\stackrel{\kdg \mid 2i-1}{\text{ and } i\neq g/2}}  (a_i^1\w{\delta}_i^1+a_i^2\w{\delta_2^i})\right)=\O,
\end{equation}
in the Picard group of $\pdtb$. We want to prove that all the above coefficients $a_i$, $a_i^1$ and $a_i^2$ are
zero. Pulling back the above relation (\ref{rel-div}) to the curve $\w{F}\to \pdtb$ and using the formulas
(\ref{int-F}), we get that $a_0=0$. Pulling back (\ref{rel-div}) to the curves $\w{F_1'}\to \pdtb$ and $\w{F_2'}\to \pdtb$
(in the range of degrees in which they are defined) and using the formulas (\ref{int-F'}),
we get that $a_1=0$ if $\kdg\nmid 1$ (i.e. if $d\neq g-1$) or that $a_1^1=a_1^2=0$ if $\kdg\mid1$ (i.e. if $d=g-1$).
Finally, by pulling back the relation (\ref{rel-div}) to the families $\w{F_{h,1}}\to \pdtb$ and $\w{F_{h,2}}\to \pdtb$
(for any $1\leq h\leq (g-2)/2$) and using the formulas
(\ref{int-Fh}), we get that $a_{h+1}=0$ if $\kdg \nmid (2h+1)$ or $h+1=g/2$ and $a_{h+1}^1=a_{h+1}^2=0$ if $\kdg \mid(2h+1)$
and $h+1\neq g/2$, which concludes the proof.
\end{proof}

As a corollary of the above Theorem \ref{T:ex-seq}, we can prove that the boundary line bundles of $\pdb$ are linearly independent.

\begin{cor}\label{C:ex-seq-rig}
We have an exact sequence
\begin{equation}\label{ex-Pic-rig}
0\to \bigoplus_{\stackrel{\kdg \: \nmid\:\: 2i-1}{\text{ or } i=g/2}}\langle \O(\ov{\delta}_i)\rangle
\bigoplus_{\stackrel{\kdg \mid 2i-1}{\text{ and } i\neq g/2}}\langle \O(\ov{\delta}_i^1)\rangle \oplus \langle \O(\ov{\delta}_i^2)\rangle\to
\Pic(\pdb)\to \Pic(\pd)\to 0,
\end{equation}
where the right map is the natural restriction morphism and the left map is the natural inclusion.
\end{cor}
\begin{proof}
As observed before, the only thing to prove is that the above sequence is exact on the left, or in other words
that the boundary line bundles $\{\O(\ov{\delta}), \O(\ov{\delta}_i^1), \O(\ov{\delta}_i^2)\}$ are linearly independent
in $\Pic(\pdb)$. This follows from Theorem \ref{T:ex-seq} using Corollary \ref{C:bound-rig}\eqref{C:bound-rig2}
and the fact that the pull-back map $\nu_d^*:\Pic(\pdb)\to \Pic(\pdtb)$ is injective, as observed in the introduction
(see diagram \eqref{E:4-Picard}).

\end{proof}

\section{Tautological line bundles}\label{S:taut-lb}

The aim of this section is to introduce some natural line bundles on $\pdtb$, which we call tautological
line bundles, and to determine the relations among them.

Let $\pi: \univb \rightarrow \pdtb$ be the universal family over $\pdtb$ (see \cite{melo2} for a modular description of $\univb$). The stack $\univb$ comes equipped with two natural line bundles: the universal line bundle $\L_d$ and
the relative dualizing sheaf $\omega_{\pi}$. Since  $\pi$ is a representable, flat and proper morphism whose geometric fibers are nodal curves, we can apply the formalism of the determinant of cohomology and of the Deligne
pairing (see  \ref{Pic-stack}) to produce some natural line bundles on $\pdtb$ which we call
{\em tautological} line bundles:
\begin{equation}\label{E:taut-lb}
\begin{aligned}
&K_{1,0}:=\langle \omega_{\pi}, \omega_{\pi} \rangle_{\pi}, \\
&K_{0,1}:=\langle \omega_{\pi}, \L_d \rangle_{\pi}, \\
&K_{-1,2}:=\langle \L_d, \L_d \rangle_{\pi}, \\
&\Lambda(n,m)=d_{\pi}(\omega_{\pi}^n\otimes \L_d^m) & \text{ for } m, n \in \Z.
\end{aligned}
\end{equation}
By abuse of notation, we use the same notation for the restriction of a tautological class to the open
substack $\pdt$. Using Facts \ref{E:Chern-det} and \ref{E:Chern-Del}, the first Chern classes of the above tautological line bundles are given by
\begin{equation}\label{E:taut-Chern}
\begin{aligned}
& \kappa_{1,0}:=c_1(K_{1,0})= \pi_*(c_1(\omega_{\pi})^2),\\
& \kappa_{0,1}:=c_1(K_{0,1})= \pi_*(c_1(\omega_{\pi})\cdot c_1(\L_d)), \\
& \kappa_{-1,2}:=c_1(K_{-1,2})= \pi_*(c_1(\L_d)^2), \\
& \lambda(n, m)=c_1(\Lambda(n,m))=c_1(\pi_{!}(\omega_{\pi}^n\otimes \L_d^m)) \text{ for any } n, m\in \Z.
\end{aligned}
\end{equation}
Note that, if $k={\mathbb C}$, the image of the classes $\kappa_{i,j}$ via the natural map
$A^1(\pdt)\to H^2(\pdt, \Z)\to H^2(\Hol_g^d,\Z)$
are, up to sign, the $\kappa_{i,j}$ classes that were considered by Erbert and Randal-Williams in \cite[Sec. 2.4]{ERW}. 

The pull-back of the Hodge line bundle  \eqref{E:Hodge} of $\mgb$ via the natural map $\w{\Phi}_d:\pdtb\to \mgb$ is a tautological line bundle on $\pdtb$.

\begin{lemma}\label{L:comp-taut}
We have that  $\w{\Phi}_d^*(\Lambda)=\Lambda(1,0)$.
%\begin{equation}\label{E:pull-taut}
%\begin{aligned}
%& \w{\Phi}_d^*(K_1)=K_{1,0}, \\
%& \w{\Phi}_d^*(\Lambda(n))=\Lambda(n,0).
%\end{aligned}
%\end{equation}
%In particular, $\Lambda(1,0)$ is the pull-back of the Hodge line bundle $\Lambda$ on $\mgb$.
\end{lemma}
\begin{proof}
Consider the diagram
\begin{equation}\label{univ-fam}
\xymatrix{
\univb \ar[r]^{\w{\Phi}_{d,1}} \ar[d]_{\pi} & \mguniv \ar[d]^{\ov{\pi}}\\
\pdtb \ar[r]^{\w{\Phi}_d} &  \mgb \\
}
\end{equation}
Recall from Section \ref{desc-stacks} that the map $\w{\Phi}_d$ sends an element  $(\C\to S, \L)\in \pdtb(S)$ into the stabilization $\C^{\rm st}\to S\in \mgb(S)$.
Now it is well-known that for every quasi-stable (or more generally semistable) curve $X$ with stabilization morphism
$\psi:X\to X^{\rm st}$, the pull-back via $\psi$ induces an isomorphism
$\psi^*:H^0(X^{\rm st},\omega_{X^{\rm st}})\stackrel{\cong}{\to} H^0(X,\omega_X)$.
Therefore, the relative dualizing sheaves of the families $\pi$ and $\ov{\pi}$ are related by
\begin{equation}\label{E:pullback-can}
\w{\Phi}_{d,1}^*(\omega_{\ov{\pi}})=\omega_{\pi}.
\end{equation}
We conclude by using the functoriality of the determinant of cohomology. 
% we get that
%\begin{equation*}
%\begin{aligned}
%& \w{\Phi}_d^*(K_1)=\w{\Phi}_d^*(\langle \omega_{\ov{\pi}},\omega_{\ov{\pi}}\rangle_{\ov{\pi}})=
%\langle \w{\Phi}_{d,1}^*(\omega_{\ov{\pi}}),\w{\Phi}_d^*(\omega_{\ov{\pi}})\rangle_{\pi}=
%\langle \omega_{\pi},\omega_{\pi}\rangle_{\pi}= K_{1,0}, \\
%& \w{\Phi}_d^*(\Lambda(n))=\w{\Phi}_d^*(d_{\ov{\pi}}(\omega_{\ov{\pi}}^n))=d_{\pi}(\w{\Phi}_d^*(\omega_{\ov{\pi}}^n))=
%d_{\pi}(\omega_{\pi}^n)= \Lambda(n,0),
%\end{aligned}
%\end{equation*}
\end{proof}

There are some relations between the tautological line bundles on $\pdtb$, as shown in the following.

\begin{thm}\label{T:taut-rel}
The tautological line bundles on $\pdtb$ satisfy the following relations in the rational Picard group $\Pic(\pdtb)\otimes \Q$:
\begin{enumerate}[(i)]
\item \label{T:taut-rel1} $K_{1,0}=\Lambda(1,0)^{12} \otimes \O(-\w{\delta})$,
\item \label{T:taut-rel2} $K_{0,1}= \Lambda(1,1) \otimes \Lambda (0,1)^{-1}$,
\item \label{T:taut-rel3} $K_{-1,2}=\Lambda(0,1)\otimes \Lambda(1,1)\otimes \Lambda(1,0)^{-2}$,
\item \label{T:taut-rel4} $\Lambda(n,m)=\Lambda(1,0)^{6n^2-6n-m^2+1}\otimes \Lambda(1,1)^{mn+\binom{m}{2}}\otimes
\Lambda(0,1)^{-mn+\binom{m+1}{2}}\otimes \O\left(-\binom{n}{2}\cdot \w{\delta}\right).$
\end{enumerate}
\end{thm}
\begin{proof}
Since the first Chern class map $c_1:\Pic(\pdtb)\to A^1(\pdtb)$ is an isomorphism by Fact
\ref{Fact-Pic}\eqref{Fact-Pic1}, it is enough to prove the above relations in the rational Chow group
$A^1(\pdtb)\otimes \Q$.

Following the same strategy as in the proof of Mumford's relations among the tautological classes of $\M_g$
(see \cite[Chap. 13, Sec. 7]{ACG}), we apply the Grothendieck-Riemann-Roch Theorem
%for quotient stacks (see \cite{EGb})
to the morphism $\pi: \univb\to \pdtb$:
\begin{equation}
\label{GRR}
\ch\left(\pi_!\left(\omega_{\pi}^n \otimes \L_d^m\right)\right)=\pi_*\left(\ch(\omega_{\pi}^n \otimes \L_d^m)\cdot \Td(\Omega_{\pi})^{-1}\right),
\end{equation}
where $\ch$ denotes the Chern character, $\Td$ denotes the Todd class and $\Omega_{\pi}$ is the sheaf of relative K\"{a}hler differentials.

Using \eqref{E:Chern-det}, we can compute the degree one part of the left hand side of \eqref{GRR}:
\begin{equation}\label{LHS}
\ch\left(\pi_!\left(\omega_{\pi}^n \otimes \L_d^m\right)\right)_1=c_1\left(\pi_!\left(\omega_{\pi}^n \otimes \L_d^m\right)\right)=c_1\left(d_{\pi}\left(\omega_{\pi}^n \otimes \L_d^m\right)\right)=\lambda(n,m).
\end{equation}
Let us now compute the degree one part of the right hand side of (\ref{GRR}).
Note that, as proved in \cite[p. 383]{ACG},
we have that $c_1(\Omega_{\pi})=c_1(\omega_{\pi})$ and that $c_2(\Omega_{\pi})$ is the class of the nodal locus of the morphism $\pi$.
In particular, we have that
\begin{equation}\label{E:push-eta}
\pi_*(c_2(\Omega_{\pi}))=\w{\delta}\in A^1(\pdtb),
\end{equation}
where $\w{\delta}$ is the total boundary divisor \eqref{E:tot-div} of $\pdtb$.
The first three terms of the inverse of the Todd class of
$\Omega_{\pi}$ are equal to
{\small
\begin{equation}\label{Todd}
\Td(\Omega_{\pi})^{-1}=1-\frac{c_1(\Omega_{\pi})}{2}+\frac{c_1^2(\Omega_{\pi})+c_2(\Omega_{\pi})}{12} + \ldots=
1-\frac{c_1(\omega_{\pi})}{2}+\frac{c_1(\omega_{\pi})^2+c_2(\Omega_{\pi})}{12}+\ldots
\end{equation}}
Using the multiplicativity of the Chern character, we get

\begin{align}\label{ch1}
&\ch(\omega_{\pi}^n\otimes \L_d ^m)\\
&=\left( 1+c_1(\omega_{\pi})+\frac{c_1(\omega_{\pi})^2}{2}+\ldots  \right)^n
\cdot\left( 1+c_1(\L_d)+\frac{c_1(\L_d)^2}{2}+ \ldots  \right)^m\nonumber\\
&=\left( 1+ n c_1(\omega_{\pi})+\frac{n^2 c_1(\omega_{\pi})^2}{2}+ \ldots  \right)\cdot
\left( 1+m c_1(\L_d)+\frac{m^2 c_1(\L_d)^2}{2}+ \ldots  \right)\nonumber\\
&=1+[n c_1(\omega_{\pi})+ m c_1(\L_d)]+\nonumber\\
&\qquad\qquad\qquad\kern-2pt+\left[\frac{n^2c_1(\omega_{\pi})^2}{2}+nm c_1(\omega_{\pi})\cdot c_1(\L_d)+
\frac{m^2c_1(\L_d)^2}{2}   \kern-2pt\right] +\ldots\nonumber
\end{align}

Combining (\ref{Todd}) and (\ref{ch1}) and using \eqref{E:taut-Chern} together with \eqref{E:push-eta},
we can compute the degree one part of the right hand side of \eqref{GRR}
$$\displaylines{
\left[\pi_*\left(\ch(\omega_{\pi}^n \otimes \L_d^m)\cdot \Td(\Omega_{\pi})^{-1}\right)\right]_1=
\pi_*\left(\left[\ch(\omega_{\pi}^n \otimes  \L_d^m)\cdot \Td(\Omega_{\pi})^{-1}\right]_2\right)\hfill\cr
\hfill=\kern-2pt\pi_*\kern-3pt\left[\frac{6n^2-6n+1}{12}c_1(\omega_{\pi})^2+\frac{2nm-m}{2}c_1(\omega_{\pi})\cdot c_1(\L_d)
+\frac{m^2}{2}c_1(\L_d)^2+\frac{c_2(\Omega_{\pi})}{12}\kern-2pt\right]}$$
\begin{equation}\label{RHS}
=\frac{6n^2-6n+1}{12}\kappa_{1,0}+\frac{2nm-m}{2}\kappa_{0,1}
+\frac{m^2}{2}\kappa_{-1,2}+\frac{\w{\delta}}{12}.
\end{equation}
Putting together \eqref{LHS} and \eqref{RHS}, we get the relation
\begin{equation}\label{E:basic-for}
\lambda(n,m)=\frac{6n^2-6n+1}{12}\kappa_{1,0}+\frac{2nm-m}{2}\kappa_{0,1}
+\frac{m^2}{2}\kappa_{-1,2}+\frac{\w{\delta}}{12}.
\end{equation}
Formula \eqref{E:basic-for} for $n=1$ and $m=0$ gives that
\begin{equation*}\tag{*}
\lambda(1,0)=\frac{\kappa_{1,0}}{12}+\frac{\w{\delta}}{12},
\end{equation*}
which proves part \eqref{T:taut-rel1}. By substituting (*) into \eqref{E:basic-for}, we get
\begin{equation}\label{E:basic-for2}
\lambda(n,m)=(6n^2-6n+1)\lambda(1,0)+\frac{2nm-m}{2}\kappa_{0,1}+
\frac{m^2}{2}\kappa_{-1,2}-\binom{n}{2}\w{\delta}.
\end{equation}
Formula \eqref{E:basic-for2} for $(n,m)=(0,1)$ and $(n,m)=(1,1)$ gives that
\begin{equation*}\tag{**}
\begin{sis}
& \lambda(0,1)=\lambda(1,0)-\frac{\kappa_{0,1}}{2}+\frac{\kappa_{-1,2}}{2},\\
&  \lambda(1,1)=\lambda(1,0)+\frac{\kappa_{0,1}}{2}+\frac{\kappa_{-1,2}}{2}, \\
\end{sis}
\end{equation*}
The system of equations (**) is equivalent to the system
\begin{equation*}\tag{***}
\begin{sis}
& \kappa_{0,1}= \lambda(1,1)-\lambda(0,1),\\
& \kappa_{-1,2}=-2\lambda(1,0)+\lambda(0,1)+\lambda(1,1),
\end{sis}
\end{equation*}
which also proves parts \eqref{T:taut-rel2} and \eqref{T:taut-rel3}. Substituting (***) into \eqref{E:basic-for2},
we get the following relation
\begin{multline}\label{E:basic-for3}
\lambda(n,m)=(6n^2-6n+1-m^2)\lambda(1,0)+\\
+\left[-mn+\binom{m+1}{2}\right]\lambda(0,1)+
\left[mn+\binom{m}{2}\right]\lambda(1,1)-\binom{n}{2}\w{\delta},
\end{multline}
%$$\left[mn+\binom{m}{2}\right]\lambda(1,1)-\binom{n}{2}\w{\delta},$$
which proves part \eqref{T:taut-rel4}.
\end{proof}
By a slight generalization of Lemma \ref{L:comp-taut}, it is easy to see that the relations in Theorem \ref{T:taut-rel}\eqref{T:taut-rel1}
and in Theorem \ref{T:taut-rel}\eqref{T:taut-rel4} with $m=0$ are the pull-back to $\pdtb$ of Mumford's relations
among the tautological classes of $\mgb$ (see  \cite[Chap. 13, Thm. (7.6)]{ACG}).

\begin{remark}\label{R:rela-int}
The proof of Theorem \ref{T:taut-rel} works a priori only  in the rational Picard group of $\pdtb$, since it uses the Grothedieck-Riemann-Roch theorem which is valid only in the rational Chow group.  However, since the Picard group of $\pdtb$ is torsion-free (as it follows from Theorem \ref{T:MainThmA}\eqref{T:MainThmA2}, to be proved in \S\ref{S:Pic-rigid}), the relations in the above Theorem holds true a posteriori  also in the integral Picard group of $\pdtb$.
\end{remark}

Motivated by Theorem \ref{T:taut-rel}, we can now define the tautological subgroup of the Picard group of the stacks $\pdtb$ and $\pdt$.

\begin{defi}\label{D:taut-Pic}
The {\em tautological} subgroup $\Pict(\pdtb)\subseteq \Pic(\pdtb)$  is the subgroup
generated by the line bundles associated to the boundary divisors of $\pdtb$ (see Section \ref{S:bound-div}) and
by the tautological line bundles $\Lambda(1,0)$, $\Lambda(0,1)$ and $\Lambda(1,1)$.

The image of $\Pict(\pdtb)\subseteq \Pic(\pdtb)$ via the natural restriction map $\Pic(\pdtb)\to \Pic(\pdt)$
is defined to be $\Pict(\pdt)$; hence, $\Pict(\pdt)\subseteq \Pic(\pdt)$ is the subgroup generated by
the tautological line bundles $\Lambda(1,0)$, $\Lambda(0,1)$ and $\Lambda(1,1)$.
\end{defi}

%As a direct consequence of Theorem \ref{T:taut-rel}, we get a set of generators for the tautological subgroup of the Picard group of $\pdtb$ and of $\pdt$.

%\begin{cor}\label{C:taut-Pic}
%\noindent
%\begin{enumerate}[(i)]
%\item The tautological subgroup $\Pict(\pdtb)\subseteq \Pic(\pdtb)$ is generated by the boundary divisors and the tautological line bundles $\Lambda(1,0)$, $\Lambda(0,1)$ and $\Lambda(1,1)$.
%\item The tautological subgroup $\Pict(\pdt)\subseteq \Pic(\pdt)$ is generated by the tautological line bundles $\Lambda(1,0)$, $\Lambda(0,1)$ and $\Lambda(1,1)$.
%\end{enumerate}
%\end{cor}

\section{Comparing the Picard groups of $\pdt$ and $\pd$}\label{S:comp-Pic}
%\section{On the gerbe $\nu_d:\pdt\to \pd$}

%The aim of this section is to
%to study the pull-back map $\nu_d^*:\Pic(\pd)\to \Pic(\pdt)$.

The aim of this Section is to study the pull-back map
$$\nu_d^*:\Pic(\pd)\to \Pic(\pdt)$$
induced by the map $\nu_d:\pdt\to \pd$ (see Section \ref{desc-stacks}). To this aim,
consider the Leray spectral sequence for the \'etale sheaf $\Gm$ with respect to the map $\nu_d$:
$$E_2^{p,q}=H^p_{\rm{\acute et}}(\pd,(R^q\nu_d)_*\Gm)\Longrightarrow H^{p+q}_{\rm{\acute et}}(\pdt,\Gm).$$
The first terms of the above spectral sequence give rise to the exact sequence
$$0\to H^1_{\rm{\acute et}}(\pd,(R^0\nu_d)_*\Gm)\stackrel{}{\longrightarrow} H^1_{\rm{\acute et}}(\pdt,\Gm)
\stackrel{}{\longrightarrow} H^0_{\rm{\acute et}}(\pd, (R^1\nu_d)_*\Gm) \stackrel{}{\longrightarrow}
H^2_{\rm{\acute et}}(\pd, (R^0\nu_d)_*\Gm).$$
Since $\nu_d$ is a $\Gm$-gerbe, we have that $(R^0\nu_d)_*\Gm=\Gm$ and $(R^1\nu_d)_*\Gm=\Pic B\Gm$,
where $\Pic B\Gm$ is canonically identified with the group $(\Gm)^*\cong \Z$ of characters of $\Gm$.
By plugging these isomorphisms into the above long exact sequence, we get the exact sequence
\begin{equation}\label{succ-Pic}
 0\to \Pic(\pd)\stackrel{\nu_d^*}{\longrightarrow} \Pic(\pdt)
\stackrel{\res}{\longrightarrow} \Z
% \Pic B\Gm=(\Gm)^*
\stackrel{\obs}{\longrightarrow}  \Br(\pd),
\end{equation}
where the above maps admits the following interpretation (which one can easily check
via standard cocycle computations): $\nu_d^*$ is the pull-back map induced by $\nu_d$;
$\res$ is the restriction to the fibers of $\nu_d$ (it coincides with the weight map defined in
\cite[Def. 4.1]{Hof1} and with the character appearing in the decomposition in \cite[Prop. 3.1.1.4]{Lie})
%, i.e. it  sends a line bundle $\L$ on $\pdt$ onto the character of $\Gm$
%with which $\Gm$ acts on $\L$;
and $\obs$ (the obstruction map) sends $1\in \Z=(\Gm)^*$ into the class $[\nu_d]$ of the
$\Gm$-gerbe $\nu_d$ in the (cohomological) Brauer group  $\Br(\pd):=H^2_{\rm{\acute et}}(\pd, \Gm)$
(see \cite[Chap. IV.3]{Gir}).

Since $\nu_d^*$ is injective, we can define a tautological subgroup of $\Pic(\pd)$ by intersecting
$\Pic(\pd)$ (which we identify with its image via $\nu_d^*$) with the tautological subgroup $\Pict(\pdt)$, as follows.

\begin{defi}\label{D:taut-rigid}
The tautological subgroup of $\Pic(\pd)$ is defined as
$$ \Pict(\pd):=\Pict(\pdt)\cap \Pic(\pd)\subseteq \Pic(\pdt).$$
\end{defi}

In order to compute generators for $\Pict(\pd)$, we
need first to compute the map ${\rm res}$ from \eqref{succ-Pic} on the generators of $\Pict(\pdt)$.

\begin{lemma}\label{exis-linebun}
We have that
$$\begin{sis}
& {\rm res}(\Lambda(1,0))=0,\\
& {\rm res}(\Lambda(0,1))=d-g+1,\\
& {\rm res}(\Lambda(1,1))=d+g-1.\\
\end{sis}
$$
\end{lemma}
\begin{proof}
Using the functoriality of the determinant of cohomology, we get that the fiber of $\Lambda(1,0)=d_{\pi}(\omega_{\pi})$ over a point $(C,L)\in \pdt$ is canonically isomorphic to
$\det H^0(C,\omega_C)\otimes \det^{-1}  H^1(C,\omega_C)$. Since $\Gm$ acts trivially on $H^0(C,\omega_C)$ and on
$H^1(C, \omega_C)$, we get that ${\rm res}(\Lambda(1,0))=0$.

Similarly, the fiber of $\Lambda(0,1)$ over a point $(C,L)\in \pdt$ is canonically isomorphic to
$\det H^0(C,L)\otimes \det^{-1}  H^1(C,L)$. Since $\Gm$ acts with weight one on the vector spaces
$H^0(C, L)$ and $H^1(C,L)$, Riemann-Roch gives that
$${\rm res}(\Lambda(0,1))=\dim H^0(C,L)-\dim H^1(C,L)=\chi(C, L)=d+1-g.$$
Finally, the fiber of $\Lambda(1,1)$ over a point $(C,L)\in \pdt$ is canonically isomorphic
to $\det H^0(C,L\otimes \omega_C)\otimes \det^{-1}  H^1(C,L\otimes \omega_C)$. Since $\Gm$ acts with weight
one on the vector spaces $H^0(C, \omega_C\otimes L)$ and $H^1(C,\omega_C\otimes L)$, Riemann-Roch gives that
$${\rm res}(\Lambda(1,1))=\dim H^0(C,\omega_C\otimes L)-\dim H^1(C,\omega_C\otimes L)=
\chi(C, \omega_C\otimes L)= d+2g-2+1-g=d-1+g. $$
%$$=d+2g-2+1-g=d-1+g.$$
\end{proof}

Combining the above Lemma \ref{exis-linebun} with Corollary \ref{C:taut-Pic}, we get the following

\begin{cor}\label{C:res-taut}
\noindent
\begin{enumerate}[(i)]
\item \label{C:res-taut1} The image of $\Pict(\pdt)$ via the map ${\rm res}$ of \eqref{succ-Pic}
is the subgroup generated by $(d+g-1,d-g+1)=(d+g-1,2g-2)$.
\item \label{C:res-taut2} $\Pict(\pd)$ is generated by $\Lambda(1,0)$ and
\begin{equation}\label{D:defi-Xi}
\Xi:= \Lambda(0,1)^{ \frac{d+g-1}{(d+g-1,d-g+1)}}\otimes \Lambda(1,1)^{-\frac{d-g+1}{(d+g-1,d-g+1)}}.
\end{equation}
\end{enumerate}
\end{cor}

%We now compute the order of $[\nu_d]$ in the Brauer group $\Br(\pd)$.
Corollary \ref{C:res-taut}\eqref{C:res-taut1} combined with the exact sequence (\ref{succ-Pic}) gives that
the order of $[\nu_d]$ in the Brauer group $\Br(\pd)$ divides
$(d+g-1,2g-2)$. Indeed the following is true:

%The main result of the present section is the following theorem which yields the order of $[\nu_d]$ in the Brauer group %$\Br(\pd)$.

%Indeed, our proof will also compute the index ${\rm Ind}([\nu_d])$ of $[\nu_d]$, i.e. the smallest $m\in \N$ such that
%$[\nu_d]$ is represented by a projective bundle over $\pd$ of relative dimension $m-1$.

%According to the general formalism of \cite{Gir}, the $\Gm$-gerbe $\nu_d$ gives rise to an element
%$[\nu_d]$ in the Brauer group $\Br(\pd):= H^2_{\rm{\acute et}}(\pd, \Gm)$ of $\pd$. We have the following result for the order of $[\nu_d]$ in $\Br(\pd)$.
%The following result will allow us to compare the Picard groups of $\pd$ and of $\pdt$.
% The aim of the present section is to prove the following
%Our first result is the following

\begin{thm}\label{order-gerbe}
 The order of $[\nu_d]$ in $\Br(\pd)$ is equal to
% \begin{equation}\label{mdg-number}
%{\rm Per}([\mu_d])={\rm Ind}([\mu_d])
$(d+1-g, 2g-2)$.
%\end{equation}
\end{thm}

In order to prove the theorem, we will reinterpret the order of  $[\nu_d]$ in terms of the existence of
a (generalized) Poincar\'e bundle.

Consider the universal family $\pi:\univ\to \pdt$.
% introduced in the proof of Lemma \ref{exis-linebun}.
%The fiber of $\univ$ over a scheme $S$ consists of the groupoid
%whose objects are triples $(\X\to S, \sigma, L)$, where $(\X\to S, L)\in \pdt(S)$ and $\sigma:S\to \X$ is a section of the family $\X\to S$, and whose
%isomorphisms are the obvious ones.
The $\Gm$-rigidification of $\univ$, denoted by $\u:=\univ\fatslash \Gm$, has a natural map
$\w{\pi}:\u\to \pd$ which  is indeed the universal family over $\pd$. However, the universal (or Poincar\'e)
line bundle $\L_d$ on $\univ$ does not necessarily descend to a line bundle on $\u$.
Instead, it turns out that there always exists on $\u$ an $m$-Poincar\'e line bundle as in the definition below.

%The morphism $u_g$ is representable and the fiber
%of $u_g$ over a geometric point $(C,L)$ of $\pdb$ is equal to $C$.

%Denote by $u_g:\cgb\to \mgb$ the universal family over $\mgb$. Consider the following Cartesian diagram:
%(in which all the maps are representable):
%\begin{equation}\label{diag-univ}
%\xymatrix{
%\cgb\times_{\mgb} \pdb \ar^{u_g'}[r] \ar[d]^{\Phi_d'}& \pdb \ar[d]^{\Phi_d}\\
%\cgb \ar[r]^{u_g} & \mgb\
%}\end{equation}

\begin{defi}\label{m-Poincare}
Let $m\in \Z$. An $m$-Poincar\'e line bundle for $\pd$ is a line bundle $\L$ on $\u$ such that
the restriction of $\L$ to the fiber ${\w{\pi}}^{-1}(C,L)\cong C$ over a  geometric point
$(C,L)$ of $\pd$ is isomorphic to $L^m$.
\end{defi}

The above definition generalizes the classical definition of  Poincar\'e line bundle,
which corresponds to the case $m=1$.
%The order of $[\nu_d]$ in the Brauer group $\Br(\pd)$ admits the following interpretation
%in terms of the existence of an $m$-Poincar\'e line bundle.

\begin{prop}\label{Poinc-order}
The order of $[\nu_d]$ in the group  $\Br(\pd)$ is equal to the smallest number
$m\in \N$ such that there exists an $m$-Poincar\'e line bundle for $\pd$.
\end{prop}

\begin{proof}
In order to prove the statement,
we need to introduce some auxiliary stacks.
Given $m\in \Z$, consider the stack $\pdt^m$ whose fiber $\pdt^m(S)$ over a scheme $S$ consists of families $\C\to S$ of smooth curves of genus $g$
endowed with a line bundle $\L$ of relative degree $d$ and whose morphisms between two objects
$(\C'\to S', \L')$ and $(\C\to S, \L)$  are given by a triple $(g, \phi, \eta)$ where
$$\xymatrix{
\C'\ar[r]^{\phi}\ar[d] \ar@{}[dr]|{\square} & \C \ar[d] \\
S'\ar[r]^g & S
}
$$
is a Cartesian diagram and $\eta: \L'^m\to \phi^*(\L^m)$ is an isomorphism of line bundles on $\C'$. Note that
$\pdt^1\cong \pdt$.

%fiber  over a scheme $S$ is the groupoid $\pdtb^m(S)$ whose objects
%are the same as the objects of $\pdtb(S)$, i.e. families  of quasistable curves $\C\to S$ of genus $g$ endowed with
%a properly balanced line bundle $\L$ of relative degree $d$, and whose morphisms between two objects $(\C\to S, \L)$ and
%$(\C'\to S, \L')$ of $\pdtb^m(S)$   are given by pairs $(\phi, \eta)$, where $\phi:\C'\to \C$ is an isomorphism over $S$
%and $\eta: \L'^m\to \phi^*(\L^m)$ is an isomorphisms of line bundles on $\C'$.
The multiplicative group $\Gm$ injects into the automorphism group of every object $(\C\to S, \L)\in \pdt^m(S)$
as multiplication by scalars on $\L$. The rigidification $\pdt^m\fatslash \Gm$ is isomorphic to $\pd$ and
the natural map $\nu_d^m:\pdt^m\to \pd$ is a $\Gm$-gerbe. By construction, the class of $[\nu_d^m]$ in
$\Br(\pd)$ is equal to $[\nu_d^m]=m\cdot [\nu_d]$.
%Therefore,  the order of $[\nu_d]$ is equal to the smallest $m\in \N$ such that  the $\Gm$-gerbe $\nu_d^m:\pdtb^m\to \pdb$
%is trivial and this happens precisely when there exists a section $\sigma_d^m:\pdb\to \pdtb^m$ of $\nu_d^m$.

%In order to see when this happens, we need to introduce
Consider the  universal family $\pi^m:\univ^m \to \pdt^m$.
The fiber of $\univ^m$ over a scheme $S$
consists of the triples $(\C\to S, \sigma, \L)$, where $(\C\to S, \L)\in \pdt(S)$ and $\sigma$ is a section of the morphism
$\C\to S$. The morphisms between two objects
$(\C'\to S', \sigma', \L')\in \univ^m(S')$ and $(\C\to S, \sigma, \L)\in \univ^m(S)$  are given by the isomorphisms
$(g, \phi, \eta)$ as above satisfying the relation $\sigma\circ g=\phi\circ \sigma' $.
%Note that $\univb^1\cong \univb$.
The $\Gm$-rigidification of $\univ^m$ is isomorphic to $\u$ and therefore we get a Cartesian diagram:
\begin{equation}\label{fam-univ}
\xymatrix{
\univ^m \ar[r]^{\pi^m}\ar[d]_{\nu_d'^m} \ar@{}[dr]|{\square} & \pdt^m\ar[d]^{\nu_d^m} \\
\u\ar[r]^{\w{\pi}} & \pd
}
\end{equation}
%between the objects $(\C\to S,  \L)\in \pdtb^m(S)$ and $(\C'\to S',  \L')\in \pdtb^m(S')$ such that
On the stack $\univ^m$ there is a universal line bundle $\NN_m$, defined as follows:
%(see \ref{Pic-stack}):
to every morphism from a scheme $f:S\to \univ^m$, which corresponds to an object $(\C\to S, \sigma, \L)\in \univ^m(S)$
as above, we associate the line bundle $\NN_m(f):=\sigma^*(\L^m)\in \Pic(S)$;
to every morphism $S'\stackrel{g}{\to} S\stackrel{f}{\to}\univ^m$, corresponding to the morphism $(g, \phi, \eta)$
between two objects  $(\C\to S, \sigma, \L)$ and $(\C'\to S', \sigma', \L')$ as above, we associate the isomorphism
$$\NN_m(f\circ g)=\sigma'^*(\L'^m)\stackrel{\sigma'^*(\eta)}{\longrightarrow} \sigma'^*\phi^*(\L^m)=g^*\sigma^*(\L^m)=g^*\NN_m(f).
$$
We have now the tools that we need to prove the result.
%With these preliminaries, we can now  prove the theorem.
Since $[\nu_d^m]=m[\nu_d]\in \Br(\pd)$, the period of $[\nu_d]$ is equal to the smallest $m\in \N$ such that  the $\Gm$-gerbe $\nu_d^m$
is trivial and this happens precisely when there exists a section $\sigma_d^m:\pd\to \pdt^m$ of $\nu_d^m$.
Since the diagram (\ref{fam-univ}) is Cartesian, the existence of a section $\sigma_d^m$ of $\nu_d^m$ is equivalent to the existence
of a section $\sigma_d'^m$ of $\nu_d'^m$. If such a section exists, then the pull-back $(\sigma_d'^m)^*\NN_m$ is an $m$-Poincar\'e line bundle on $\pd$, by the above description of $\NN_m$. Conversely, the existence of a Poincar\'e line bundle on $\pd$ allows us to define a section $\sigma_d'^m$ of $\nu_d'^m$ by
the above description of $\univ^m$.

%\begin{equation}\label{diag-univ2}
%\xymatrix{
%\cgb\times_{\mgb} \pdtb^m \ar^{u_g^m}[r] \ar[dr]^{\w{\nu_d^m}}\ar[dd]_{\Phi_d^m} \ar@{}[drr]|{\square} & \pdtb^m \ar[dr]^{\nu_d^m} &\\
%&\cgb\times_{\mgb} \pdb \ar^{u_g'}[r] \ar[ld]_{\Phi_d'} \ar@{}[d]|{\square} & \pdb \ar[ld]^{\Phi_d}\\
%\cgb \ar[r]^{u_g} & \mgb &}
%\end{equation}
\end{proof}

\begin{proof}[Proof of Theorem \ref{order-gerbe}]
Consider the group
$$A_{d,g}:=\{m\in \Z\: :\: \text{ there exists an $m$-Poincar\'e line bundle } \L \text{ on } \univ\}
$$
Proposition \ref{Poinc-order} gives that the positive generator of $A_{d,g}$ is equal to
the order of $[\nu_d]$ in $\Br(\pd)$.
On the other hand, the positive generator of $A_{d,g}$ is equal to $(d+g-1,2g-2)$ by
\cite[Application at p. 514]{kou2}. This concludes the proof.

%Let $m\in \N$ be equal to the order of $[\nu_d]$ in $\Br(\pd)$.
%By combining the exact sequence (\ref{succ-Pic}) and Lemma \ref{exis-linebun}, we get that $m$ divides $(d+g-1,2g-2)$.
%In order to prove that $m$ is a multiple of $(d+g-1,2g-2)$,
%let $\PP$ be a $m$-Poincar\'e line bundle on $\u$, which exists by Proposition \ref{Poinc-order}.
%The proof of \cite[Lemma 2.3]{MR}
%(which deals with $m=1$)
%can be easily generalized from the case of an (ordinary) Poincar\'e line bundle to the case of an $m$-Poincar\'e
%line bundle, and gives that
%there exists $n\in \Z$ such that $\Phi_{n(d+1-g)-m}: {\mathcal P}_{n(d+1-g)-m}\to \mg$  has a section.
%According to the (so called) strong Franchetta's conjecture (see \cite{Mes} for a proof in characteristic zero
%and \cite{Sch} for a proof in positive characteristic), $\Phi_e: {\mathcal P}_{e,g} \to \mg$ has a regular
%(or rational) section if and only if $e$ is a multiple of $2g-2$. Therefore we conclude that $n(d+1-g)-m=l(2g-2)$ for some %$l\in \Z$,
%hence  $m$ is a multiple of $(d+1-g, 2g-2)$.
\end{proof}

\begin{remark}
From Proposition \ref{Poinc-order} and Theorem \ref{order-gerbe}, we recover the following well-known result due to
Mestrano-Ramanan (\cite[Cor. 2.9]{MR}):  there exists a Poincar\'e line bundle on $\u$
if and only if $(d+1-g,2g-2)=1$.
\end{remark}

\begin{remark}
It is possible to prove that the index  of $[\nu_d]$ is equal to $(d+g-1,2g-2)$
 (recall that the index of $[\nu_d]$ is the smallest $m\in \N$ such that
$[\nu_d]$ is represented by a projective bundle over $\pd$ of relative dimension $m-1$).
Since we will not need this result, we do not include a proof here.
\end{remark}

We make the following

\begin{conj}\label{C:prob-Brauer}
The cohomological Brauer group $\Br(\pd)$ of $\pd$ is generated by the class $[\nu_d]$ of the $\Gm$-gerbe
$\nu_d: \pdt\to \pd$.
\end{conj}

Using the notation of Section \ref{S:comp-top}, the above conjecture must be compared with the result of Ebert and Randal-Williams who proved in \cite[Thm. B]{ERW} that, for $g\geq 6$, $H^3(\Pic_g^d,\Z)$ is cyclic of order $(2g-2,d+g-1)$ and generated by the
Dixmier-Douday class of the $\CC^*$-gerbe $\phi_g^d:\Hol_g^d\to \Pic_g^d$.
From the diagram \eqref{E:comp-stacks} and the coboundary map coming from the exponential sequence
of locally constant sheaves $0\to \Z \to \CC \stackrel{{\rm exp}}{\longrightarrow} \CC^*\to 0$,
we get a map ${\rm cl}: \Br(\pd)\to H^2(\Pic_g^d,\CC^*)\to H^3(\Pic_g^d,\Z)$ which clearly sends the class of $\nu_d$ into the class of $\phi_g^d$. A positive answer to Conjecture \ref{C:prob-Brauer} together with Theorem \ref{order-gerbe}
would imply that the above map ${\rm cl}$ is an isomorphism for $g\geq 6$.

\vspace{0,2cm}

From the above Theorem \ref{order-gerbe}, we deduce the following

\begin{cor}\label{C:comp-Pic}
\noindent
\begin{enumerate}[(i)]
\item \label{C:comp-Pic1} The image of $\Pic(\pdt)$ via the map ${\rm res}$ of \eqref{succ-Pic}
is the subgroup generated by $(d+g-1,2g-2)$.
\item \label{C:comp-Pic2} The pull-back map $\nu_d^*$ induces an isomorphism
$$\nu_d^*:\Pic(\pd)/\Pict(\pd) \stackrel{\cong}{\longrightarrow} \Pic(\pdt)/\Pict(\pdt).$$
\end{enumerate}
\end{cor}
\begin{proof}
Part \eqref{C:comp-Pic1} follows from the exact sequence \eqref{succ-Pic} together with Theorem \ref{order-gerbe}.

Part \eqref{C:comp-Pic2}: using Corollary \ref{C:res-taut}\eqref{C:res-taut1} and part \eqref{C:comp-Pic1}, we get
the following commutative diagram with exact rows:
$$\xymatrix{
 0\ar[r] & \Pic(\pd)\ar[r]^{\nu_d^*}&  \Pic(\pdt) \ar[r]^(0.4){\res} &
 \Z\cdot \langle (d+g-1,2g-2)\rangle \ar[r] & 0 \\
 0\ar[r] & \Pict(\pd)\ar[r]^{\nu_d^*}\ar@{^{(}->}[u]&  \Pict(\pdt) \ar[r]^(0.4){\res} \ar@{^{(}->}[u] &
\Z\cdot  \langle (d+g-1,2g-2)\rangle \ar[r] \ar@{=}[u] & 0 \\
}
$$
The conclusion follows from the snake lemma.
\end{proof}

\section{The Picard group of $\pd$}\label{S:Pic-rigid}

In this subsection we will determine the Picard group of the stack $\pd$, using a strategy similar to the one used by
Kouvidakis \cite{kouvidakis} to determine the Picard group of $J_{d,g}^0$, the open subset of $J_{d,g}$ consisting of pairs $(C,L)$ where
$C$ is a smooth curve without non-trivial automorphisms.

Consider the representable morphism $\Phi_d: \pd\to \mg$. Clearly the fiber of $\Phi_d$ over $C\in \mg$
is the degree-$d$ Jacobian $J^d(C)$ of $C$. Since $\Phi_d$ has connected fibers, the pull-back map
$\Phi_d^*: \Pic(\mg)\to \Pic(\pd)$ is injective.
%Recall that, if $g\geq 3$, then  $\Pic(\mg)$ is freely generated over $\Z$ by the Hodge line bundle $\lambda$
%(see \ref{sec-pic-mg}).
The cokernel of $\Phi_d^*$ is denoted by $\RPic(\pd)$ and is called classically
the group of rationally determined line bundles of the family $\pd\to \mg$ (see e. g. \cite{Cil}).
Therefore, we have the following exact sequence
\begin{equation}\label{ex-Pic-open}
0\to \Pic(\mg) \stackrel{\Phi_d^*}{\to} \Pic(\pd)\to \RPic(\pd)\to 0.
\end{equation}
Since the fiber of $\Phi_d$ over $C\in \mg$  is the degree-$d$ Jacobian $J^d(C)$ of $C$, we have a natural
map
\begin{equation}\label{res-C}
\rho_C:\Pic(\pd)\to \Pic(J^d(C))\to NS(J^d(C)),
\end{equation}
where the first map is the restriction to the fiber $\Phi_d^{-1}(C)=J^d(C)$ and the second map is the projection of the Picard  group of $J^d(C)$
onto the N\'eron-Severi group of $J^d(C)$, which parametrizes  divisors on $J^d(C)$ up to algebraic equivalence. We will use additive notation for the group law
on $NS(J^d(C))$.

Consider now the theta divisor $\Theta(C)\subset J^{g-1}(C)$ and denote by $\theta_C\in $ $ NS(J^{g-1}(C))$ its algebraic equivalence class.
By choosing an isomorphism $t_M:J^d(C)\stackrel{\cong}{\to} J^{g-1}(C)$ given by sending $L\in J^d(C)$ into $L\otimes M\in J^{g-1}(C)$ for some $M\in J^{g-1-d}(C)$,
we can pull-back  $\theta_C$ to get a well-defined (i.e. independent of the chosen isomorphism $t_M$) class in $NS(J^d(C))$ which, by
a slight abuse of notation, we will still denote by $\theta_C$.
Since, for a very general curve $C\in \mg$,  $NS(J^d(C))$  is generated by  $\theta_C$ (see e. g.
\cite[Lemma 2]{kouvidakis}), it follows that there is a morphism of groups
\begin{equation}\label{E:map-i}
\chi_d:\Pic(\pd)\longrightarrow \Z
\end{equation}
sending $\L\in \Pic(\pd)$ to the integer $m$ such that  $\rho_C(\L)=m\theta_C$ for every $C\in \mg$
 (see also \cite[p. 840]{kouvidakis}).
We will need the following two results of Kouvidakis, describing the image and the kernel of the
above map $\chi_d$. Actually, Kouvidakis  proves these results in \cite{kouvidakis} for the variety $J_{d,g}^0$, but a close inspection reveals that the same proof works for $\pd$.

\begin{thm}[Kouvidakis]\label{T:Kouvi}
\noindent
\begin{enumerate}[(i)]
\item \label{T:Kouvi1} $\ker \chi_d=\Im \Phi_d^*$.
\item \label{T:Kouvi2} $\displaystyle \Im \chi_d\subseteq   \frac{2g-2}{(2g-2, d+g-1)}\cdot \Z \subseteq \Z$.
\end{enumerate}
\end{thm}
Part \eqref{T:Kouvi1} follows from \cite[Thm. 3]{kouvidakis}; part \eqref{T:Kouvi2} follows from
\cite[Formula (*), p. 844]{kouvidakis}. Note that part \eqref{T:Kouvi1} implies (and it is indeed equivalent to)
that the map $\chi_d$ factors as
\begin{equation}\label{E:factor-chi}
\chi_d: \Pic(\pd)\twoheadrightarrow \RPic(\pd)\hookrightarrow \Z.
\end{equation}

%By \cite[Thm. 3]{kouvidakis}, two line bundles $\L, \L'\in \Pic(\pd)$ define the same
%element in $\RPic(\pd)$ if and only if their restrictions to the geometric fibers
%(equivalently to a general fiber, or equivalently to a very general fiber if the base field $k$ is uncountable)
%of $\Phi_d$ are algebraically equivalent.

\noindent We now compute the image of the map $\chi_d$ on the tautological subgroup $\Pict(\pd)$ of $\Pic(\pd)$
(see Definition \ref{D:taut-rigid}).

\begin{thm}\label{T:chi-taut}
We have that
$$\chi_d(\Pict(\pd))=\frac{2g-2}{(2g-2, d+g-1)}\cdot \Z \subseteq \Z.$$
\end{thm}
\begin{proof}
According to Corollary \ref{C:res-taut}\eqref{C:res-taut2}, $\Pict(\pd)$ is generated by the tautological
classes $\Lambda(1,0)$ and $\Xi$. Lemma \ref{L:comp-taut} gives that $\Lambda(1,0)=\Phi_d^*(\Lambda)$;
hence clearly $\chi_d(\Lambda(1,0))=0$ (this is the easy inclusion in Theorem \ref{T:Kouvi}\eqref{T:Kouvi1}).
Therefore, the proof will follow if we show that
\begin{equation}\label{E:chi-xi}
\chi_d(\Xi)=\frac{2g-2}{(2g-2, d+g-1)},
\end{equation}
or equivalently that
\begin{equation}\label{E:resC-xi}
\rho_C(\Xi)=\frac{2g-2}{(2g-2, d+g-1)}\theta_C
\end{equation}
for any $C\in \mg$.
%We claim that there exists a morphism of groups $\w{\chi}_d:\Pict(\pdt)\to \Z$ making the following diagram commutative
%\begin{equation}\label{E:exte-chi}
%\xymatrix{
%\Pict(\pdt) \ar[rd]^{\w{\chi}_d} & \\
%\Pict(\pd)\ar[r]^(0.6){\chi_d} \ar@{^{(}->}[u] & \Z
%}
%\end{equation}
%and such that the values of the morphism $\w{\chi}_d$ on the generators of $\Pict(\pdt)$ (see Corollary  \ref{C:taut-Pic}) are given by
%\begin{equation}\label{E:chi-gen}
%\begin{sis}
%& \w{\chi}_d(\Lambda(1,0))=0, \\
%& \w{\chi}_d(\Lambda(0,1))=1, \\
%& \w{\chi}_d(\Lambda(1,1))=1. \\
%\end{sis}
%\end{equation}
In order to prove this, consider the following diagram
\begin{equation}\label{E:big-diag}
\xymatrix@=1pc{
& \L_C \ar@{-}[d] &&& \L_d\ar@{-}[d] & \\
& C\times \Pics^d(C) \ar[dl]_p\ar[ddr]^(0.3){\id\times \nu_C} \ar[rrr] &&& \univ \ar[dl]_{\pi} \ar[ddr] & \\
\Pics^d(C) \ar[ddr]_{\nu_C}  \ar@{}[drr]|{\square} \ar[rrr] &&& \pdt \ar[ddr]_(0.2){\nu_d}  \ar@{}[drr]|{\square} && \\
&&C\times J^d(C)\ar[dl]^{p_2} \ar[rrr] &&& \u\ar[dl]^{\w{\pi}} \\
& J^d(C)\ar[d] \ar[rrr] &&& \pd \ar[d]^{\Phi_d} & \\
& C \ar@{^{(}->}[rrr] &&& \mg &
}
\end{equation}
where the Cartesian square on the left is the fiber of the Cartesian square on the right over the point $C\in \mg$
and $\L_C$ is the fiber of  the universal line bundle $\L_d$ over $C\in \mg$.  In particular, the stack $\Pics^d(C)$ is the degree-$d$ Jacobian stack of
$C$ (i.e. the stack whose fiber over a scheme $S$ is the groupoid of line bundles on $C\times S$ of relative degree $d$ over $S$) and $\L_C$
is the universal (or Poincar\'e) line bundle for $\Pics^d(C)$.

The map $\nu_C:\Pics^d(C)\to J^d(C)$ is a $\Gm$-gerbe which is well-known to be trivial, or in other words $\Pics^d(C)\cong J^d(C)\times B\Gm$.
Therefore, there exists a section $s$ of $\nu_C$ and we can define  $\w{\L}_C:=(\id\times s)^*(\L_C)$. By construction, we have that
$\w{\L}_{|C\times \{M\}}=M$ for any $M\in J^d(C)$.  Any line bundle on $C\times J^d(C)$ with this property is called  a Poincar\'e line bundle for $J^d(C)$.
Indeed, any Poincar\'e line bundle for $J^d(C)$ is isomorphic to $(\id\times s)^*(\L_C)$ for a uniquely determined section  $s$ of $\nu_C$.
Moreover, two Poincar\'e line bundles for $J^d(C)$ differ by the tensor product with the pull-back of a line bundle on $J^d(C)$.
Note that for any Poincar\'e line bundle $\w{\L}_C=(\id\times s)^*(\L_C)$ for $J^d(C)$, we have that
$(\id\times \nu_C)^*(\w{\L}_C)=(\id\times \nu_C)^*((\id\times s)^*(\L_C))=\L_C$.

%From now on, we fix a Poincar\'e line bundle $\w{\L}_C=(\id\times s)^*(\L_C)$ on $C\times J^d(C)$.
Recalling the definition of $\Xi$ from Corollary \ref{C:res-taut}\eqref{C:res-taut2} and applying the functoriality of the determinant of cohomology 
to the above diagram \eqref{E:big-diag}, we get that
\begin{equation}\label{E:rho-xi1}
\rho_C(\Xi)= \frac{d+g-1}{(d+g-1,d-g+1)}[d_{p_2}(\w{\L}_C)] - \frac{d-g+1}{(d+g-1,d-g+1)} [d_{p_2}(\w{\L}_C\otimes p_1^*(\omega_C))],
\end{equation}
where $p_1:C\times J^d(C)$ denotes the projection onto the first factor and $\w{\L}_C$ is \emph{any }ÊPoincar\'e line bundle for $J^d(C)$.
Note that the fact that $\Xi\in \Pic(\pd)$ guarantees that the right hand side of \eqref{E:rho-xi1}
is independent of  the choice   $\w{\L}_C$.

In order to compute the right hand side of \eqref{E:rho-xi1}, we  can choose a Poincar\'e line bundle
$\w{\L}_C$ for $J^d(C)$ that satisfies the following

\un{Condition (*)}: $[(\w{\L}_C)_{|p_1^{-1}(r)}]=0\in NS(J^d(C))$ for any $r\in C$.

Indeed, since $\w{\L}_C$ can be seen as a family of line bundles on $J^d(C)$ parametrized by $C$, if condition (*) holds for a certain point $r_0\in C$ then it holds
for all points $r\in C$. However, up to tensoring $\w{\L}_C$ with the pull-back of a line bundle on $J^d(C)$, we can always assume that $(\w{\L}_C)_{|p_1^{-1}(r_0)}$ is the trivial
line bundle on $J^d(C)$, q.e.d.

With the above condition on $\w{\L}_C$, we can prove the following two claims.

\un{Claim 1}: If $\w{\L}_C$ satisfies condition (*) then
$$[d_{p_2}(\w{\L}_C\otimes p_1^*(M))]=[d_{p_2}(\w{\L}_C)]\in NS(J^d(C)) \text{ for any } M\in J(C).$$

Indeed, write $M=\O_C(-\gamma+\delta)$ with  $\gamma=\sum_i a_i r_i$ and $\delta=\sum_j b_j r_j$ effective divisors on $C$.
From the exact sequences defining the structure sheaves of $p_1^{-1}(\delta)\subset C\times J^d(C)$ and $p_1^{-1}(\gamma)\subset C\times J^d(C)$, we get
$$
\begin{sis}
&0\to \w{\L}_C\otimes p_1^*\O_C(-\gamma)\to \w{\L}_C\to (\w{\L}_C)_{|p_1^{-1}(\gamma)}\to 0,  \\
&0\to \w{\L}_C\otimes p_1^*\O_C(-\gamma)\to \w{\L}_C\otimes p_1^*(M)\to (\w{\L}_C)_{|p_1^{-1}(\delta)}\to 0. \\
\end{sis}
$$
From the multiplicativity of the determinant of cohomology applied to the above exact sequences, we get
\begin{equation*}
d_{p_2}(\w{\L}_C\otimes p_1^*M)\otimes d_{p_2}(\w{\L}_C)^{-1}= d_{p_2}((\w{\L}_C)_{|p_1^{-1}(\delta)}) \otimes d_{p_2}((\w{\L}_C)_{|p_1^{-1}(\gamma)})^{-1}=
\end{equation*}
$$=\bigotimes_j (\w{\L}_C)_{p_1^{-1}(r_j)}^{ b_j} \bigotimes_i (\w{\L}_C)_{p_1^{-1}(r_i)}^{- a_i}.$$
Claim 1 follows now by condition (*).

\un{Claim 2}: If $\w{\L}_C$ satisfies condition (*) then
$$[d_{p_2}(\w{\L}_C)]=\theta_C\in NS(J^d(C)).$$

Indeed, choose a line bundle $M\in J^{d-g+1}(C)$ and consider the Cartesian diagram
$$\xymatrix{
(\id\times t_M)^*(\w{\L}_C)\ar@{-}[d] & \w{\L}_C \ar@{-}[d] \\
C\times J^{g-1}(C) \ar[r]^{\id \times t_M} \ar[d]^{p_2'}& C\times J^d(C) \ar[d]^{p_2} \\
J^{g-1}(C) \ar[r]^{t_M} & J^d(C),
}
$$
where $t_M$ is the map sending $L\in J^{g-1}(C)$ into $L\otimes N\in J^d(C)$. The line bundle $\w{\L}'_C:=(\id\times t_M)^*(\w{\L}_C)\otimes p_1^*(M)^{-1}$
is clearly a Poincar\'e line bundle for $J^{g-1}(C)$ and it satisfies condition (*) since $\w{\L}_C$ satisfies condition (*) by assumption.
Therefore, using the functoriality of the determinant of cohomology and Claim 1, we get the following equality in $NS(J^{g-1}(C))$:
\begin{equation}\label{E:back-det}
[t_M^* d_{p_2}(\w{\L}_C)]=[d_{p_2'}((\id\times t_M)^*(\w{\L}_C))]=[d_{p_2'}(\w{\L}_C'\otimes p_1^*(M))]=[d_{p_2'}(\w{\L}_C')].
\end{equation}
Claim 2 now follows from the well-known fact that $d_{p_2'}(\w{\L}_C')\in \Pic(J^{g-1}(C))$ is the line bundle associated to the theta divisor $\Theta(C)\subset J^{g-1}(C)$ for
any Poincar\'e line bundle $\w{\L}_C'$ for $J^{g-1}(C)$.

Now choosing a Poincar\'e line bundle $\w{\L}_C$ that satisfies condition (*), formula \eqref{E:rho-xi1} together with Claim 1 and Claim 2 gives that
$$\rho_C(\Xi)=\frac{d+g-1}{(d+g-1,d-g+1)} \theta_C- \frac{d-g+1}{(d+g-1,d-g+1)}\theta_C=$$
$$=\frac{2g-2}{(2g-2, d+g-1)}\theta_C,$$
which proves \eqref{E:resC-xi}.
\end{proof}

By combining the above results, we can now prove the main Theorems \ref{T:MainThmA} and \ref{T:MainThmB}
from the introduction.

\begin{proof}[Proof of Theorem \ref{T:MainThmB}]

Let us first prove Theorem \ref{T:MainThmB}\eqref{T:MainThmB1}.
By combining Theorem \ref{T:Kouvi}\eqref{T:Kouvi2} with Theorem \ref{T:chi-taut}, we get
that $\chi_d(\Pic(\pd))=\chi_d(\Pict(\pd))$. By Theorem  \ref{T:Kouvi}\eqref{T:Kouvi1}, the kernel of $\chi_d$ is equal to $\Phi_d^*(\Pic(\mg))$, which is generated by $\Lambda(1,0)=\Phi_d^*(\Lambda)$ by Theorem \ref{pic-mg} and Lemma \ref{L:comp-taut}; hence $\Im \Phi_d^*\subset \Pict(\pd)$. We deduce that
\begin{equation}\label{E:equa-taut}
\Pict(\pd)=\Pic(\pd).
\end{equation}
Therefore, $\Pic(\pd)$ is generated by $\Lambda(1,0)$ and by $\Xi$ by Corollary \ref{C:res-taut}\eqref{C:res-taut2}.
Consider now the exact sequence \eqref{ex-Pic-open}. Combining the factorization of $\chi_d$ provided by \eqref{E:factor-chi} with formula \eqref{E:chi-xi}, we get that $\RPic(\pd)$ is free of rank one.
On the other hand, using Theorem \ref{pic-mg} (since $g\geq 3$ by assumption), we know that
$\Pic(\mg)$ is free of rank one. Therefore the exact sequence \eqref{ex-Pic-open} gives that
$\Pic(\pd)$ is free of rank two, which concludes the proof of part \eqref{T:MainThmB1}.

Theorem \ref{T:MainThmB}\eqref{T:MainThmB2} follows now from part \eqref{T:MainThmB1} and Corollary \ref{C:ex-seq-rig}.
\end{proof}

\begin{proof}[Proof of Theorem \ref{T:MainThmA}]
Let us first prove Theorem \ref{T:MainThmA}\eqref{T:MainThmA1}.
From \eqref{E:equa-taut} and Corollary \ref{C:comp-Pic}\eqref{C:comp-Pic2}, we deduce that
\begin{equation}\label{E:equa-taut2}
\Pict(\pdt)=\Pic(\pdt).
\end{equation}
Therefore, $\Pic(\pdt)$ is generated by $\Lambda(1,0)$, $\Lambda(0,1)$ and $\Lambda(1,1)$ by  Corollary
\ref{C:taut-Pic}. Moreover, the exact sequence \eqref{succ-Pic} together with Theorem \ref{T:MainThmB}\eqref{T:MainThmB1} implies that $\Pic(\pdt)$ is free of rank three.
Part \eqref{T:MainThmA1} is now proved.

Theorem \ref{T:MainThmA}\eqref{T:MainThmA2} follows now from part \eqref{T:MainThmA1} and Theorem \ref{T:ex-seq}.

\end{proof}

We can now compare our computation of $\Pic(\pd)$ (see Theorem \ref{T:MainThmB}\eqref{T:MainThmB1}) with the
computation of $\Pic(J_{d,g}^0)$ carried out by Kouvidakis in \cite{kouvidakis}.

\begin{remark}\label{Kou-mistake}
Assume that $g\geq 3$. Then the natural map $\Psi_d:\pd\to J_{d,g}$ is an isomorphism over the open subset
$J_{d,g}^0\subset J_{d,g}$ parametrizing pairs $(C,L)\in J_{d,g}$ such that $C$ does not have non-trivial
automorphisms. In other words, the map $\Psi_d$ induces an isomorphism
$$\Psi_d:\pdo:=\Psi_d^{-1}(J^0_{d,g})\stackrel{\cong}{\longrightarrow} J_{d,g}^0.$$
Therefore, we get a natural homomorphism
\begin{equation}\label{E:comp-kou}
\psi:\Pic(\pd)\to \Pic(\pdo)\xrightarrow[\Psi_d^*]{\cong} \Pic(J^0_{d,g}),
\end{equation}
where the first homomorphism is the natural restriction map.

If $g\geq 4$, then the codimension of $\pd\setminus \pdo$ inside $\pd$ is at least two
and hence the map $\psi$ is an isomorphism by Fact \ref{Fact-Pic}\eqref{Fact-Pic3}.
Hence Theorem \ref{T:MainThmB}\eqref{T:MainThmB1} recovers \cite[Thm. 4]{kouvidakis}.
However, this does not hold anymore if $g=3$ since in this
case $\pd\setminus \pdo$  is a divisor inside $\pd$, namely the pull-back of the hyperelliptic (irreducible)
divisor in ${\mathcal M}_3$, whose class in $A^1(\mg)$ is  equal to $9 \lambda$
(see \cite[Chap. 3, Sec. E]{HM}). Therefore, by Fact \ref{Fact-Pic}\eqref{Fact-Pic2}, we get that
$\Pic(\pdo)\cong \Pic(J^0_{d,g})$ is the quotient of $\Pic(\pd)$ by the relation $\Lambda(1,0)^9=0$.
%Hence, \cite[Thm. 4]{kouvidakis} does not hold for $g=3$.
\end{remark}

\subsection{Relation between $\Xi$ and the universal theta divisor}\label{S:rela-theta}

There is a close relationship between the line bundle $\Xi\in \Pic(\pd)\subset \Pic(\pdt)$ and the universal theta divisor
$\Theta\subset \pmidt$, which is the closed substack parametrizing pairs $(C,L)\in \pmidt$ such that $h^0(C,L)>0$. Observe that
$\Theta$ naturally descends to a divisor on the rigidification $\pmid$, which we  denote by $\ov{\Theta}$ and we call the universal theta divisor
on $\pmid$. By construction, the restriction of $\ov{\Theta}$ to any fiber $\Phi_d^{-1}(C)=J^{g-1}(C)$ is isomorphic to the theta divisor $\Theta(C)\subset J^{g-1}(C)$.

Consider first the special case $d=g-1$. From the definition \eqref{D:defi-Xi} of $\Xi$ and using the definition \eqref{E:taut-lb} of the tautological line bundles,
we get that $\Xi=\Lambda(0, 1)=d_{\pi}(\L_{g-1})$, where $\L_{g-1}$ is the universal line bundle on the universal family over $\pmidt$. It is well know that
$d_{\pi}(\L_{g-1})$ is the line bundle associated to the universal theta divisor, or in other words we have that
\begin{equation}\label{Xi-theta1}
\Xi=\O(\Theta) \: \: \:  \text{ if } d=g-1.
\end{equation}

For an arbitrary $d$,  we consider the stack $\skd$ of $k_{d,g}$-spin curves, where as usual
$$\kdg=\displaystyle \frac{2g-2}{(2g-2,d+1-g)}.$$
Recall that $\skd$ is the stack whose fiber over a scheme
$S$ consists of the groupoid of families of smooth curves  $\C\to S$ of genus $g$, plus a line bundle
$\eta$ on $\C$ of relative degree $\mcd$  over $S$ endowed with
an isomorphism $\eta^{\otimes \kdg}\cong \omega_{\C/S}$. The stack $\skd$ is a smooth Deligne-Mumford
stack endowed with a (forgetful) finite and \'etale map $\skd\to \mg$ of degree $(2g)^{\kdg}$.
We have a diagram
\begin{equation}\label{E:diag-theta}
\xymatrix{
& \F \ar[dd]^{\pi} \ar[dr]^{\w{s}} \ar[dl]_{\w{p}_2}Ê\ar@{}[dddr]|{\square}  Ê\ar@{}[dddl]|{\square} & \\
\univ \ar[dd]_{\pi_2}& & \unimid \ar[dd]^{\pi_1}\\
 &  \skd\times_{\mg} \pdt \ar[dl]^{p_2} \ar[dr]_s & \\
 \pdt & & \pmidt \\
}
\end{equation}
where $p_2$ is the projection onto the second factor and  $s$ sends the element $(\C\to S, \eta, \L)\in \skd\times_{\mg} \pdt(S)$ into
$(\C\to S, \L \otimes \eta^{-\edg}) \in {\mathcal Jac}_{g-1,g}(S)$, where
$$\edg:=\frac{d-g+1}{\mcd}.$$
The universal family $\F$ is endowed with a universal line bundle $\L_d$ of relative degree $d$ which is the pulled-back from $\univ$ and
a universal spin line bundles $\eta_{\kdg}$ which is pulled-back from the universal family above $\skd$. By the definition of the morphism
$s$, we get that
\begin{equation}\label{E:pullback-L}
\w{s}^*(\L_{g-1})= \eta_{\kdg}^{-\edg}\otimes \L_d.
\end{equation}

The relation between the line bundle $\Xi\in \Pic(\pdt)$ and the universal theta divisor $\Theta\subset \pmidt$ is provided by the following.

%Consider the diagram
%$$\xymatrix{
%\skd\times_{\mg} \pd \ar[dd]_{\Phi_d'} \ar[dr]_{q_d'}\ar[drr]^{s_d}& & \\
%& \pd \ar[d]_{\Phi_d}& \pmid \ar[dl]^{\Phi_{g-1}}\\
%\skd \ar[r]^{q_d}& \mg & \\
%}
%$$
%where $s_d$ is the $\Gm$-rigidification of the map of $\Gm$-stacks $\w{s_d}: \skd\times_{\mg} \pdt \to {\mathcal Jac}_{g-1,g}$
%that sends the element $(\C\to S, \eta, \L)\in \skd\times_{\mg} \pdt(S)$ into $(\C\to S, \L \otimes \eta^{-\edg}) \in {\mathcal Jac}_{g-1,g}(S)$, where
%$$\edg:=\frac{d-g+1}{\mcd}.$$

\begin{lemma}\label{L:Xi-theta}
We have that
$$p_2^*(\Xi)=s^*\O(\kdg\cdot \Theta)\otimes  \langle \eta_{\kdg}, \eta_{\kdg}\rangle_{\pi}^{-\frac{\kdg (\kdg+\edg) \edg}{2}}.$$
 \end{lemma}
\begin{proof}
By the definition \eqref{D:defi-Xi} of $\Xi$ and the standard properties of the determinant of cohomology,
we compute
\begin{equation}\label{E:theta-LHS}
p_2^*(\Xi)=d_{\pi}(\L_d)^{\frac{d+g-1}{(2g-2,d+1-g)}} \otimes d_{\pi}(\omega_{\pi}\otimes \L_d)^{-\frac{d-g+1}{(2g-2, d+1-g)}}=
d_{\pi}(\L_d)^{\kdg+\edg} \otimes d_{\pi}(\eta_{\kdg}^{\kdg}\otimes \L_d)^{-\edg}.
\end{equation}
Using \eqref{Xi-theta1} and \eqref{E:pullback-L}  together with standard properties of the determinant of cohomology, we get that
\begin{equation}\label{E:theta-RHS}
s^*(\O(\kdg \cdot \Theta))=s^*(d_{\pi_1}(\L_{g-1})^{\kdg})=d_{\pi}( \eta_{\kdg}^{-\edg}\otimes \L_d)^{\kdg}.
\end{equation}
In order to compare \eqref{E:theta-LHS} and \eqref{E:theta-RHS}, we apply
the Grothedieck-Riemann-Roch theorem to the sheaf $\eta_{\kdg}^n\otimes \L_d^{m}$ on the universal family $\pi:\F\to  \skd\times_{\mg} \pdt$.
After  some easy computations similar to the ones done in the proof of Theorem \ref{T:taut-rel} which we leave to the reader, we get that
\begin{equation}\label{E:spin-GRR}
c_1(d_{\pi}(\eta_{\kdg}^n \otimes \L_d^m))= \frac{6 n^2- 6 \kdg n +\kdg^2}{12}c_1(\langle \eta_{\kdg}, \eta_{\kdg}\rangle_{\pi})+\frac{2mn-\kdg m}{2}c_1(\langle \eta_{\kdg},\L_d\rangle_{\pi})
+\frac{m^2}{2}c_1(\langle \L_d, \L_d\rangle_{\pi}).
\end{equation}
Using the above formula \eqref{E:spin-GRR}, we can compute the difference between the first Chern classes of the line bundles in \eqref{E:theta-LHS} and in \eqref{E:theta-RHS}:
$$c_1(p_2^*(\Xi))-c_1(s^*(\O(\kdg \cdot \Theta)))=(\kdg+\edg)c_1(d_{\pi}(\L_d))-\edg c_1(d_{\pi}(\eta_{\kdg}^{\kdg}\otimes \L_d))-\kdg c_1(d_{\pi}( \eta_{\kdg}^{-\edg}\otimes \L_d))=
$$
$$=-\frac{\kdg(\kdg+\edg)\edg}{2}c_1(\langle \eta_{\kdg}, \eta_{\kdg}\rangle_{\pi}).
$$
The result now follows since $c_1:\Pic( \skd\times_{\mg} \pdt)\to A^1( \skd\times_{\mg} \pdt)$ is an isomorphism (see Fact \ref{Fact-Pic}\eqref{Fact-Pic1}).

\end{proof}

\begin{remark}
 Using the computation of the Picard group of the moduli stacks of spin curves by Jarvis \cite{Jar}, it can be proved that the pull-back morphism
$p_2^*:\Pic(\pdt)\to \Pic( \skd\times_{\mg} \pdt)$ is injective. Therefore, Lemma \ref{L:Xi-theta} uniquely determines the line bundle $\Xi$.
However, while the definition \eqref{D:defi-Xi} extends naturally to  $\pdtb$, we do not know how to extend the formula of Lemma \ref{L:Xi-theta} to
$\pdtb$. The problem is that we do not know how to extend the correspondence between $\pdt$ and $\pmidt$ given in diagram \eqref{E:diag-theta} to a correspondence
between $\pdtb$ and $\pmidtb$.

%It remains to show the uniqueness of $\ldg$. This amounts to show the injectivity of the
%pull-back map $(q_d')^*:\Pic(\pd)\to \Pic(\skd\times_{\mg} \pd)$. Clearly, using that $\Phi_d$ and
%$\Phi_d'$ have connected fibers, we get the following commutative diagram with exact horizontal rows
%$$\xymatrix{
 %\Pic(\mg) \ar@{^{(}->}^{\Phi_d^*}[r] \ar^{q_d^*}[d]& \Pic(\pd)\ar^{(q_d')^*}[d]
%\ar@{->>}[r] &\RPic(\pd) \ar^{\ov{(q_d')^*}}[d]  \\
 %\Pic(\skd) \ar@{^{(}->}^(.4){(\Phi_d')^*}[r] & \Pic(\skd\times_{\mg} \pd)
%\ar@{->>}[r] &\RPic(\skd\times_{\mg} \pd)
%}$$
%The map $q_d^*$ is known to be injective (it follows easily from \cite[Prop. 3.3]{Jar}) and
%$ \ov{(q_d')^*}$ is injective since $q_d'$ is a finite map.
%Therefore, from the above commutative diagram, we deduce that also $(q_d')^*$ is injective,
%as required.

% We will now construct a line bundle $\ldg\in \Pic(\pd)$, whose image in $\RPic(\pd)$ will turn out to be a generator
%of $\RPic(\pd)$. Indeed, our construction of $\ldg$ is inspired by  \cite[Rmk. 2, p. 849]{kouvidakis} and thus
%provides a positive answer to \cite[Rmk. 3, p. 850]{kouvidakis}. See also \cite[Rmk. 1, p. 521]{kou2} for a different
%construction of $\ldg$.
\end{remark}

\subsection{Relation between  $\pdt$ and the universal $d$-th symmetric product}
\label{S:univ-sym}

The referee pointed out to us an interesting connection between the Picard groups of $\pdt$ and  
of the $d$-th symmetric product $\Sym^d \M_{g,1}$ of the universal curve $\M_{g,1}\to \M_{g}$, when $d>2g-2$. 

The fiber of the stack $\Sym^d \M_{g,1}$ (for $d\geq 1$) over a scheme $S$ is the groupoid whose objects are families of smooth curves
$\C\to S$ of genus $g$ together with an effective divisor $\D\subset \C$ of relative degree $d$ over $S$, and whose arrows are the obvious isomorphisms. Consider the universal Abel-Jacobi morphism
\begin{equation}\label{E:Abel-Jac}
\begin{aligned}
\w{A}_{d}: \Sym^d\M_{g,1} & \longrightarrow \pdt\\
(\C\to S, \D) & \mapsto (\C\to S, \O_{\C}(\D)),
\end{aligned}
\end{equation}
and the induced commutative diagram 
\begin{equation}\label{E:diag-Abel}
\xymatrix{
\Sym^d\M_{g,1}\times_{\M_g} \M_{g,1} \ar[d]_{\w{\pi}} \ar@{}[drr]|{\square} \ar[rr]^{\h{A}_d} & &\u\ar[d]^{\pi}  \\
\Sym^d\M_{g,1} \ar[rr]^{\w{A}_d} \ar[rd]_{A_d} & &\pdt\ar[ld]^{\nu_d}  \\
& \pd \ar[d]^{\Phi_d} & \\
& \M_g &
}
\end{equation}
If $d>2g-2$ then $A_d$ is a projective bundle of relative dimension $d-g$ whose class
$[A_d]$ in the Brauer group $\Br(\pd)$ is equal to the class $[\nu_d]$ of the $\Gm$-gerbe $\nu_d$, as it follows easily from \cite[Lemma 2.1]{MR}.
Therefore, the exact sequence \eqref{succ-Pic} induced by the $\Gm$-gerbe $\nu_d$ maps into the analogous 
exact sequence for the projective bundle $A_d$:
\begin{equation}\label{E:2exseq}
\xymatrix{
 0\ar[r] & \Pic(\pd)\ar[r]^(0.4){\nu_d^*} \ar@{=}[d]& \Pic(\pdt) \ar[r]^(0.6){\res} \ar[d]^{\w{A}_d^*}& \Z \ar[d]^{\cong} \ar[r]^(0.3){\obs}&  \Br(\pd) \ar@{=}[d]\\
  0\ar[r] & \Pic(\pd)\ar[r]^(0.4){A_d^*} & \Pic(\Sym^d \M_{g,1}) \ar[r]^(0.7){\w{\res}}& \Z \ar[r]^(0.3){\w{\obs}}&  \Br(\pd)\\
}
\end{equation}
where the maps in the second exact sequence of the above diagram admit the following interpretation (which one can easily check
via standard cocycle computations): $A_d^*$ is the pull-back map induced by $A_d$;
$\w{\res}$ is the restriction to the generic fiber of $A_d$ and $\w{\obs}$ (the obstruction map) sends $1\in \Z$ into the class $[A_d]$ of the
projective bundle  $A_d$ in the (cohomological) Brauer group  $\Br(\pd):=H^2_{\rm{\acute et}}(\pd, \Gm)$.

The above diagram \eqref{E:2exseq} implies that the pullback map $\w{A}_d^*$ is an isomorphism. Moreover the pullback of the tautological line bundles on $\pdt$ can be expressed as tautological line 
bundles on $\Sym^d \M_{g,1}$. Indeed, from the Cartesian square at the top of diagram \eqref{E:2exseq}, we get that 
\begin{equation}\label{E:pull-Poin}
\h{A}_d^*(\omega_{\pi})=\omega_{\w{\pi}} \hspace{0.8cm} \text{ and }Ê\hspace{0.8cm} \h{A}_d^*(\L_d)=\O(\D_d),
\end{equation}Ê
where $\omega_{\w{\pi}}$ is the relative dualizing line bundle for  $\w{\pi}$ and  $\D_d$ is the universal degree-$d$ divisor on $\Sym^d\M_{g,1}\times_{\M_g} \M_{g,1}$.
Using the functoriality of the determinant of cohomology, we get 
\begin{equation}\label{E:pull-tauto}
\begin{aligned}
& \w{A}_d^*(\Lambda(1,0))=d_{\w{\pi}}(\omega_{\w{\pi}}):=\w{\Lambda}(1,0), \\
& \w{A}_d^*(\Lambda(0,1))=d_{\w{\pi}}(\O(\D_d)):=\w{\Lambda}(0,1), \\
& \w{A}_d^*(\Lambda(1,1))=d_{\w{\pi}}(\omega_{\w{\pi}}(\D_d)):=\w{\Lambda}(1,1). 
\end{aligned}
\end{equation} 
Therefore, combining Theorem \ref{T:MainThmA}\eqref{T:MainThmA1}, \eqref{E:2exseq} and \eqref{E:pull-tauto}, we deduce the following 
\begin{cor}\label{C:Pic-Sym}
Assume that $g\geq 3$ and that $d>2g-2$. 
The Picard group of $\Sym^d\M_{g,1}$ is freely generated by
$\w{\Lambda}(1,0)$, $\w{\Lambda}(0,1)$ and $\w{\Lambda}(1,1)$.
\end{cor}

\begin{remark}
The referee pointed out to us that Corollary \ref{C:Pic-Sym} could be proved independently  from Theorem \ref{T:MainThmA}\eqref{T:MainThmA1}, using the computations contained in \cite{kou3}.
In turn, this can be used to give an alternative proof  of Theorems \ref{T:MainThmA}\eqref{T:MainThmA1} and \ref{T:MainThmB}\eqref{T:MainThmB1} (at least for $d>2g-2$).
However, this alternative approach does not give a modular description of the generators of the Picard groups of $\pdtb$ and of $\pdb$, since it is not known how to extend the Abel-Jacobi morphism 
over the boundary of $\mgb$.
\end{remark}

\section{Relation with the  moduli space $\ov{J_{d,g}}$}\label{S:rel-modspace}

The aim of this section is to relate the Picard group of the stack $\pdb$ with the divisor
class group and the rational Picard group of its moduli space $\ov{J_{d,g}}$, computed by Fontanari
in \cite[Thm. 5, Cor. 1]{fontanari}, based upon the results of Kouvidakis \cite{kouvidakis}.

Recall that, given a variety $Y$, the divisor class group $\Cl(Y)$ is the group of Weil divisors modulo
rational equivalence. If $Y$ is normal, denoting by $Y_{\rm reg}$ the open subset of regular points
of $Y$, then we have that
\begin{equation}
\Pic(Y)\hookrightarrow \Cl(Y)\cong \Cl(Y_{\rm reg})\cong \Pic(Y_{\rm reg}).
\end{equation}
Recall that $\ov{J}_{d,g}$ is a normal variety (see Theorem \ref{T:st-sp}) and it is endowed with a morphism $\phi_d:\ov{J}_{d,g}\to \ov{M}_g$ into the coarse moduli space of stable curves of genus $g$ (see diagram \eqref{big-dia}).

%The divisor class group and the Picard group with rational coefficients of $\ov{J}_{d,g}$ were
%computed by Fontanari (see \cite[Thm. 4, Cor. 1]{fontanari}), based upon the result of Kouvidakis (see
%\cite{kouvidakis}) for the Picard group of $J_{d,g}^0$.

\begin{thm}[Fontanari]\label{divclass}
Set $\w{\Delta_i}:=\phi_d^{-1}(\Delta_i)\subset \ov{J}_{d,g}$ for $i=0,\cdots,[g/2]$.
\begin{enumerate}[(i)]
%\item The divisor class group $\Cl(J_{d,g})$ is freely generated over $\Z$ by $\phi_d^*(\lambda)$ and by a Cartier divisor %$L_{d,g}$ on $(J_{d,g})_{\rm reg}$, characterized by the property that
%the restriction of $L_{d,g}$ to the fiber $\phi_d^{-1}(C)=J^d(C)$ over a general point $C\in M_g$ has class equal to
%$\kdg [\Theta(C)]\in NS(J(C))$.
%, where $[\Theta(C)]\in NS(J(C))$ is the class of the theta divisor $\Theta(C)$ in the N\'eron-Severi group of the
%Jacobian $J(C)$ of $C$.
\item The divisors $\w{\Delta_i}$ are irreducible and we have an exact sequence
\begin{equation*}
0\to \bigoplus_{i=0}^{[g/2]} \Z\cdot \w{\Delta_i} \to \Cl(\ov{J}_{d,g})\to \Cl(J_{d,g})\to 0.
\end{equation*}
\item The natural inclusion $\Pic(\ov{J}_{d,g})\hookrightarrow \Cl(\ov{J}_{d,g})$ is of finite index, i.e. every Weil divisor on $\ov{J}_{d,g}$ is
$\Q$-Cartier.
\end{enumerate}
\end{thm}

We have therefore a commutative diagram with exact rows:

$$\xymatrix{
0\ar[r] & \bigoplus_{i=0}^{[g/2]} \Z\cdot \w{\Delta_i} \ar[r] \ar[d]^{\alpha_d}&  \Cl(\ov{J}_{d,g})\ar[r] \ar^{\Psi_d^*}[d]&  \Cl(J_{d,g})\ar[r] \ar[d]^{\beta_d}&  0\\
0 \ar[r] & \bigoplus_{\stackrel{\kdg  \nmid\: \:2i-1}{\text{ or } i=g/2}}\langle \O(\ov{\delta}_i)\rangle
\bigoplus_{\stackrel{\kdg \mid 2i-1}{\text{ and } i\neq g/2}}\langle \O(\ov{\delta}_i^1),\O(\ov{\delta}_i^2)\rangle\ar[r]&
\Pic(\pdb)\ar[r] & \Pic(\pd)\ar[r] & 0,
}$$
where the map $\Psi_d^*$ is the pull-back map induced by $\Psi_d:\pdb\to \ov{J}_{d,g}$.
We can now prove Theorem \ref{T:MainThmC} from the introduction.

%In the next theorem, we are going to use the Notation \ref{nota-div}.

\begin{proof}[Proof of Theorem \ref{T:MainThmC}]
In order to prove part \eqref{T:MainThmC1} of Theorem \ref{T:MainThmC}, consider the commutative diagram,
obtained by pulling back divisors along the two fibrations $\pd\to \mg$ and $J_{d,g}\to M_g$:
$$\xymatrix{
 & \Cl(M_g) \ar@{=}[d] & \Cl(J_{d,g}) \ar@{=}[d] & \\
0\ar[r] & \Pic ((M_g)_{\rm reg}) \ar[r] \ar^{\gamma_d}[d]&  \Pic ((J_{d,g})_{\rm reg}) \ar[r]
\ar^{\beta_d}[d]&  \RPic((J_{d,g})_{\rm reg}) \ar[r] \ar[d]^{\ov{\beta_d}}&  0\\
0 \ar[r] & \Pic(\mg)\ar[r]& \Pic(\pd)\ar[r] & \RPic(\pd)\ar[r] & 0,
}$$
The map $\gamma_d$ is well-known to be an isomorphism (see e. g. \cite[Prop. 2]{arbcorn}).
The map $\ov{\beta_d}$ is an isomorphism since the group of rational determined line bundles
$\RPic$ of a fibration is birational on the base (see \cite[Lemma 1.3]{Cil}) and the map $\pd\to \mg$
is representable. Since the rows of the above diagram are exact, we conclude that $\beta_d$ is an
isomorphism, q.e.d.

In order to prove part \eqref{T:MainThmC2} of Theorem \ref{T:MainThmC}, we need a  local description of the morphism $\Psi_d:\pdb\to \ov{J}_{d,g}$
at the general point of $\w{\Delta_i}$. This was carried on in \cite[Proof of Thm.  1.5]{BFV} for the morphism
$\nu_d\circ \Psi_d:\pdtb\to \ov{J}_{d,g}$, but it is very easy to adapt the description in loc. cit. to the morphism
$\Psi_d$ (simply by passing to the $\Gm$-rigidification).

If $\kdg\nmid(2i-1)$ (which corresponds  to the cases (1) and (2) of loc. cit.) then the morphism $\Psi_d$ is an isomorphism
locally at the general point of $\w{\Delta_i}$ (see \cite[p. 25]{BFV}). Therefore $\Psi_d^*(\w{\Delta_i})=\O(\ov{\delta}_i)$.

If $\kdg \mid(2i-1)$ (which corresponds to the case (3) of loc. cit.) then the morphism $\Psi_d$ looks like (after neglecting
trivial coordinates)
$${\mathcal X}:=[\Spf k[[x,y]]\widehat{\otimes} A/\Gm] \stackrel{p}{\longrightarrow}
X:=\Spf k[[x,y]]/\Gm\widehat{\otimes} A=\Spf k[[xy]]\widehat{\otimes} A,$$
where $A=\Spf k[[y_1,\cdots,y_{4g-4}]]$, $\Gm$ acts via $\lambda\cdot (x,y)=(\lambda x, \lambda^{-1} y)$
and trivially on $A$ (see \cite[p. 26]{BFV}).
In this local description, the divisor $\w{\Delta_i}$ corresponds to the divisor $(xy=0)$ on $X$ and
the divisors $\w{\delta}_i^1$ and $\w{\delta}_i^2$ correspond to the divisors $(x=0)$ and $(y=0)$ on
${\mathcal X}$ (note that in the particular case $i=g/2$ and $\kdg\mid(g-1)$, the divisor $\w{\delta}_{g/2}$, even
though irreducible, locally analytically splits into two components, which we can call $\w{\delta}_{g/2}^1$ and $\w{\delta}_{g/2}^2$,
so that the above description remains valid also in this case).
From the explicit form of the map $p$, it is clear that $p^*(xy=0)=(x=0)+(y=0)$, from which we deduce that
$$\Psi_d^*(\w{\Delta}_i)=
\begin{cases}
\O(\ov{\delta}_i^1+\ov{\delta}_i^2) & \text{ if } i < g/2,\\
\O(2\ov{\delta}_{g/2}) & \text{ if } i=g/2.
\end{cases}
$$
Part \eqref{T:MainThmC2} is now proved.

%Consider now the commutative diagram, obtained by pulling-back divisors along the natural maps
%indicated in the diagram (\ref{big-dia}):
%$$\xymatrix{
%\Cl(\ov M_g)  \ar[r]^{\phi_d^*} \ar^{f_d^*}[d]&
%\Cl(\ov{J}_{d,g})  \ar^{\Psi_d^*}[d] & \\
% \Pic(\mgb)\ar[r]^{\Phi_d^*}& \Pic(\pdb) ,
%}$$
%By definition of $D_i$ (see Theorem \ref{divclass}), we have that $\phi_d^*(\Delta_i)=D_i$.
%It is well-known (see \cite[Prop. 2]{arbcorn}) that
%$$f_d^*(\Delta_i)=
%\begin{cases}
%\delta_i & \text{ if } i\neq 1, \\
%2 \delta_i & \text{ if } i=1.
%\end{cases}
%$$
%Using this and the equations (\ref{pull-back-delta_i}), we conclude part (2) from the commutativity of the above
%diagram.
\end{proof}


\begin{thebibliography}{ACGH2}


\bibitem[ACV03]{ACV} D. Abramovich, A. Corti, A. Vistoli: Twisted bundles and admissible covers. Special
issue in honor of Steven L. Kleiman. {\em Comm. Algebra} {\bf 31} (2003), no. 8, 3547--3618.

\bibitem[AGV09]{AGV} D. Abramovich, T. Graber, A. Vistoli: Gromov-Witten theory of Deligne-Mumford stacks.
{\em Amer. J. Math.}  {\bf 130}  (2008),  no. 5, 1337--1398.

%\bibitem[Ale04]{alex} Alexeev, V.: Compactified Jacobians and Torelli map. {\em Publ. RIMS, Kyoto Univ.}
%{\bf 40} (2004), 1241--1265.

\bibitem[Alp1]{alper} J. Alper:  Good moduli spaces for Artin stacks. To appear in Annales de l'Institute de Fourier (preprint arXiv:0804.2242v2).

\bibitem[Alp2]{alper2} J. Alper: Adequate moduli spaces and geometrically reductive group schemes. Preprint
arXiv:1005.2398.

\bibitem[AC87]{arbcorn} E. Arbarello and M. Cornalba: The Picard groups of the moduli spaces of curves. {\em Topology} {\bf 26} (1987), no. 2, 153--171.

\bibitem[ACG11]{ACG} E. Arbarello, M. Cornalba, P.A. Griffiths: {\em Geometry of algebraic curves. Volume II.}
With a contribution by Joseph Daniel Harris. Grundlehren der Mathematischen Wissenschaften 268.   Springer, Heidelberg, 2011.

\bibitem[BL94]{BL} A. Beauville, Y. Laszlo: Conformal blocks and generalized theta functions. {\em Comm. Math. Phys.} {\bf 164} (1994),  385Ð-419.

\bibitem[BLS98]{BLS} A. Beauville, Y. Laszlo and C. Sorger: The Picard group of the moduli of G-bundles on a curve.
{\em Compositio Math.} {\bf 112} (1998), 183--216.

\bibitem[BFV12]{BFV} G. Bini, C. Fontanari, F. Viviani: On the birational geometry of the universal Picard variety.
{\em Int. Math. Res. Not.}   2012, No. 4, 740-780 (2012).

\bibitem[BFMV]{BMV} G. Bini, F. Felici, M. Melo, F. Viviani: GIT for polarized curves. To appear in Lecture Notes in Mathematics (preprint arXiv:1109.6908v3).

\bibitem[BMV12]{BMV2} G. Bini, M. Melo, F. Viviani:   On GIT quotients of Hilbert and Chow schemes of curves. {\em Electron. Res. Announc. Math. Sci.}  {\bf 19}  (2012), 33--40.

\bibitem[BH10]{BH} I. Biswas, N. Hoffmann: The line bundles on moduli stacks of principal bundles on a curve.
{\em Doc. Math.} {\bf 15} (2010), 35--72.

\bibitem[BK05]{BK} A. Boysal, S. Kumar: Explicit determination of the Picard group of moduli spaces of semistable $G$-bundles on curves.
{\em Math. Ann.} {\bf 332} (2005), 823Ð-842.

\bibitem[Cap94]{cap} L. Caporaso: A compactification of the universal Picard variety over the moduli space of stable curves.
 {\em J. Amer. Math. Soc.} {\bf 7} (1994), no. 3, 589--660.

\bibitem[Cap05]{capneron} L. Caporaso: N\'eron models and compactified Picard schemes over the moduli stack of stable curves.
 {\em Amer. J. Math.} {\bf 130} (2008), no. 1, 1--47.

%\bibitem[Cap]{Cap-Lis} L. Caporaso: Compactified Jacobians of nodal curves. Notes for a minicourse given at the %Istituto Superiore Tecnico of Lisbon,
%1--4 February  2010 (available at http://www.mat.uniroma3.it/users/caporaso/cjac.pdf).

\bibitem[CMKVa]{CMKV} S. Casalaina-Martin, J. L. Kass, F. Viviani: The Local Structure of Compactified Jacobians. Preprint arXiv:1107.4166v2.

\bibitem[CMKVb]{CMKVb} S. Casalaina-Martin, J. L. Kass, F. Viviani:  The singularities and birational geometry of the universal compactified Jacobian. In preparation.

\bibitem[Cil87]{Cil} C. Ciliberto: On rationally determined line bundles on a family of projective curves with general moduli.
 {\em Duke Math. J.}  {\bf 55}  (1987), no. 4, 909--917.

\bibitem[Cor91]{Cor1} M. Cornalba: A remark on the Picard group of spin moduli space.
{\em Atti Accad. Naz. Lincei Cl. Sci. Fis. Mat. Natur. Rend. Lincei (9) Mat. Appl.}  {\bf 2}  (1991),  no. 3, 211--217.

\bibitem[Cor07]{Cor2} M. Cornalba: The Picard group of the moduli stack of stable hyperelliptic curves.
{\em Atti Accad. Naz. Lincei Cl. Sci. Fis. Mat. Natur. Rend. Lincei (9) Mat. Appl.}  {\bf 18}  (2007),  no. 1, 109--115.

\bibitem[Del87]{Del} P. Deligne: Le d\'eterminant de la cohomologie. Current trends in arithmetical algebraic geometry
(Arcata, Calif., 1985),  93–177, Contemp. Math., 67, Amer. Math. Soc., Providence, RI, 1987.

\bibitem[DN89]{DN} J. M. Drezet, M. S. Narasimhan: Groupe de Picard des vari\'et\'es de modules de fibr\'es semi-stables sur les courbes alg\'ebriques.
{\em Invent. Math.} {\bf 97} (1989),  53Ð-94.

\bibitem[Edi12]{Edi} D. Edidin: Equivariant geometry and the cohomology of the moduli space of curves.
To appear in {\em Handbook of Moduli}, edited by  G. Farkas and I. Morrison, Advanced Lectures in Mathematics  (2012).
Preprint arXiv:1006.2364.

\bibitem[EG98]{EGa} D. Edidin, W. Graham: Equivariant intersection theory. {\em Invent. Math.}
{\bf 131} (1998), 595-634.

%\bibitem[EG00]{EGb} D. Edidin, W. Graham: Riemann-Roch for equivariant Chow groups. {\em Duke Math. J.}
%{\bf 102} (2000), no. 3, 567–-594.

\bibitem[ERW12]{ERW} J. Ebert, O. Randal-Williams: Stable cohomology of the universal Picard varieties and
the extended mapping class group. {\em Doc. Math.} {\bf 17} (2012), 417--450.

\bibitem[Fal03]{Fal} G. Faltings: Algebraic loop groups and moduli spaces of bundles. {\em J. Eur. Math. Soc.} {\bf 5} (2003), 41--68.

\bibitem[FV13]{FV} G. Farkas, S. Verra: The classification of universal Jacobians over the moduli space of curves.
  {\em Comment. Math. Helv.} 88 (2013), no. 3, 587--611.
  
\bibitem[Fon05]{fontanari} C. Fontanari: On the geometry of moduli of curves and line bundles.
{\em Ren. Mat. Acc. Lincei}
%Atti Accad. Naz. Lincei,  Cl. Sci. Fis. Mat. Natur. Rend. Lincei (9) Mat. Appl.}
{\bf 16} (2005), no. 1, 45--59.

\bibitem[Ful98]{Ful} W. Fulton: {\em Intersection theory.} Ergebnisse der Mathematik und ihrer Grenzgebiete. 3. Folge. A Series of Modern Surveys in Mathematics 2. Springer-Verlag, Berlin, 1998.

\bibitem[Gie82]{gieseker} D. Gieseker: {\em Lectures on moduli of curves.}
Tata Inst. Fund. Res. Lectures on Math. and Phys.  69, Springer-Verlag, Berlin-New York, 1982.

\bibitem[Gir71]{Gir} J. Giraud: {\em Cohomologie non ab\'elienne.}
Die Grundlehren der mathematischen Wissenschaften, Band 179. Springer-Verlag, Berlin-New York, 1971.

\bibitem[GV08]{GV} S. Gorchinskiy, F. Viviani: Picard group of moduli of hyperelliptic curves.
{\em Math. Z.}  {\bf 258}  (2008),  no. 2, 319--331.

\bibitem[Har83]{Har} J. Harer: The second homology group of the mapping class group of an orientable surface. {\em Inven. Math.} {\bf 72} (1983), 221--239.

\bibitem[HM98]{HM} J. Harris, I. Morrison: {\em Moduli of curves.} Graduate Texts in Mathematics, 187. Springer-Verlag, New York, 1998.

\bibitem[HH09]{HH}
B. Hassett, D. Hyeon: Log canonical models for the moduli space of curves:
first divisorial contraction. {\em Trans. Amer. Math. Soc.}  {\bf 361}  (2009), 4471--4489.

%\bibitem[HH]{HH2}
%B. Hassett, D. Hyeon: Log canonical models for the moduli space of curves: the first flip.
%Preprint available at arXiv:0806.3444.

\bibitem[Hof07]{Hof1} N. Hoffmann: Rationality and Poincar\'e families for vector bundles with extra structure on a curve.
{\em Int. Math. Res. Not.}  2007,  no. 3, Art. ID rnm010, 30 pp.

\bibitem[KM76]{KM} F.F. Knudsen, D. Mumford: The projectivity of the moduli space of stable curves. I.
Preliminaries on ÒdetÓ and ÒDivÓ. {\em Mathematica Scandinavica} {\bf 39} (1976), 19--55.

\bibitem[GIT65]{GIT} D. Mumford: {\em Geometric invariant theory.} Ergebnisse der
      Mathematik und ihrer Grenzgebiete, Neue Folge, Band 34 Springer-Verlag,
      Berlin-New York 1965.


%\bibitem[H]{hartshorne} R. Hartshorne: Algebraic Geometry.
%Graduate Texts in Mathematics, No. 52. Springer-Verlag, New York-Heidelberg, 1977.

\bibitem[Jar01]{Jar} T. J. Jarvis: The Picard group of the moduli of higher spin curves.
{\em New York J. Math.} {\bf 7}  (2001), 23--47.

\bibitem[Kou91]{kouvidakis} A. Kouvidakis:
The Picard group of the universal Picard varieties over the moduli space of curves,
{\em J. Differential Geom.} {\bf 34} (1991), no. 3, 839--850.

\bibitem[Kou93]{kou2} A. Kouvidakis: On the moduli space of vector bundles on the fibers of the universal curve.
{\em J. Differential Geom.}  {\bf 37}  (1993),  no. 3, 505--522.

\bibitem[Kou94]{kou3} A. Kouvidakis: Picard groups of Hilbert schemes of curves.
{\em J. Algebraic Geom.} {\bf 3}  (1994),  no. 4, 671--684.

\bibitem[KN97]{KN} S. Kumar and M. S. Narasimhan: Picard group of the moduli spaces of G-bundles.
{\em  Math. Ann.} {\bf 308} (1997), 155--173.

\bibitem[LS97]{LS} Y. Laszlo, C. Sorger: The line bundles on the moduli of parabolic G-bundles over curves
and their sections. {\em  Ann. Sci. \'Ecole Norm. Sup.} (4) {\bf 30} (1997),  499--525.

\bibitem[Lie08]{Lie} M. Lieblich: Twisted sheaves and the period-index problem. {\em Compositio Math.} {\bf 144} (2008), 1--31.

\bibitem[Mel09]{melo} M. Melo: Compactified Picard stacks over $\overline{\mathcal M}_g$. {\em Math. Z.}  {\bf 263}  (2009),  no. 4, 939--957.

\bibitem[Mel10]{melo2} M. Melo: Compactified Picard stacks over the moduli stack of stable of stable curves with marked
points. {\em Adv. Math.} {\bf 226} (2011), no. 1, 727–-763.

%\bibitem[Mes87]{Mes} N. Mestrano: Conjecture de Franchetta forte. {\em Invent. Math.} {\bf 87} (1987) ,  365-- 376.

\bibitem[MR85]{MR} N. Mestrano, S. Ramanan: Poincar\'e bundles for families of curves.
\emph{J. Reine Angew. Math.} {\bf 362} (1985), 169--178.

\bibitem[Mor01]{Mor} A. Moriwaki: The $\Q$-Picard group of the moduli space of curves in positive characteristic.
{\em Internat. J. Math.}  {\bf 12}  (2001),  no. 5, 519--534.

\bibitem[Mum65]{Mum} D. Mumford: Picard groups of moduli functors, in \emph{Arithmetical algebraic geometry} Ð Proceedings of a conference held at Purdue University,
O. Schilling, editor (1965), 33--81.

\bibitem[Mum77]{Mum2} D. Mumford: Stability of projective varieties. \emph{L'Ens. Math.} {\bf 23} (1977), 1--39.

\bibitem[Mum83]{Mum3} D. Mumford: Towards an enumerative geometry of the moduli space of curves. Arithmetic and geometry, Vol. II, 271–-328, Progr. Math., 36, Birkh\"auser Boston, Boston, MA, 1983.

\bibitem[Pan96]{Pan} R. Pandharipande: A compactification over $\mgb$ of the universal moduli space of slope-semi-stable vector bundles.  {\em J. Amer. Math. Soc.} {\bf 9} (1996), 425--471.

\bibitem[Put12]{Put} A. Putman: The Picard group of the moduli space of curves with level structures.
{\em Duke Math. J.}  {\bf 161} (2012), 623--674.

\bibitem[Rom05]{Rom} M. Romagny: Group actions on stacks and applications.  {\em Michigan Math. J.}  {\bf 53}
(2005),  no. 1, 209--236.

\bibitem[Sor99]{Sor} C. Sorger: On moduli of $G$-bundles of a curve for exceptional $G$. {\em Ann. Sci. \'Ecole Norm. Sup.} {\bf 32} (1999), 127--133.

%\bibitem[Sch03]{Sch} S. Schr\"oer: The strong Franchetta conjecture in arbitrary characteristics.
%{\em Internat. J. Math.}  {\bf 14}  (2003),  no. 4, 371--396.

\bibitem[Vis98]{Vis} A. Vistoli: The Chow ring of $\mathcal{M}_2$.
{\em Invent. Math.} {\bf 131} (1998), 635--644.


\end{thebibliography}
\end{document}